\documentclass[reqno,11pt]{amsart}
\usepackage[utf8]{inputenc}

\usepackage{amsmath,amsxtra,amsbsy,enumerate, latexsym, amsfonts, amssymb, amsthm, amscd, stmaryrd}
\usepackage{graphics,epsf,psfrag}
\usepackage{bbm, dsfont}
\usepackage{mathtools}
\usepackage{xcolor}
\usepackage{float}
\usepackage{tikz}
\usepackage[headings]{fullpage}
\usepackage{colonequals}

\numberwithin{equation}{section}
\allowdisplaybreaks

\usepackage{xcolor}
\usepackage[normalem]{ulem}
\usepackage{hyperref}

\hypersetup{
    colorlinks=true,
    linkcolor=blue,
    filecolor=magenta,      
    urlcolor=blue,
    citecolor=blue,
     pdfstartview=FitV,
    pdfpagemode=FullScreen,
    hyperindex,breaklinks
    }

\makeatletter
\DeclareUrlCommand\ULurl@@{%
  \def\UrlLeft{\uline\bgroup}%
  \def\UrlRight{\egroup}}
\def\ULurl@#1{\hyper@linkurl{\ULurl@@{#1}}{#1}}
\DeclareRobustCommand*\ULurl{\hyper@normalise\ULurl@}
\makeatother

\definecolor{darkred}{rgb}{0.7,0.1,0.1}
\newcommand{\red}{\color{darkred}}
\definecolor{darkgreen}{rgb}{0.1,0.7,0.1}

\let\oldtocsection=\tocsection
\let\oldtocsubsection=\tocsubsection
\let\oldtocsubsubsection=\tocsubsubsection
\renewcommand{\tocsection}[2]{\hspace{0em}\oldtocsection{#1}{#2}}
\renewcommand{\tocsubsection}[2]{\hspace{1em}\oldtocsubsection{#1}{#2}}
\renewcommand{\tocsubsubsection}[2]{\hspace{2em}\oldtocsubsubsection{#1}{#2}}
\DeclareRobustCommand{\SkipTocEntry}[5]{}

\newcommand{\bbE}{{\ensuremath{\mathbb E}} }
\newcommand{\bbF}{{\ensuremath{\mathbb F}} }

\newcommand{\bbL}{{\ensuremath{\mathbb L}} }

\newcommand{\bbN}{{\ensuremath{\mathbb N}} }

\newcommand{\bbP}{{\ensuremath{\mathbb P}} }

\newcommand{\bbR}{{\ensuremath{\mathbb R}} }

\newcommand{\bbZ}{{\ensuremath{\mathbb Z}} }

\renewcommand{\epsilon}{\varepsilon}

\newcommand{\fg}{{\ensuremath{\mathbf{g}}} }
\newcommand{\fh}{{\ensuremath{\mathbf{h}}} }
\newcommand{\fB}{{\ensuremath{\mathbf{B}}} }
\newcommand{\fA}{{\ensuremath{\mathbf{A}} }}

\newcommand{\bb}{{\ensuremath{\mathbf{b}} }}

\newcommand{\ga}{\alpha}

\newcommand{\gga}{\gamma}            

\newcommand{\gep}{\varepsilon}       

\newcommand{\gk}{\kappa}
\newcommand{\go}{\omega}

\newcommand{\gO}{\Omega}
\newcommand{\gl}{\lambda}

\newcommand{\cG}{{\ensuremath{\mathcal G}} }

\newcommand{\cA}{{\ensuremath{\mathcal A}} }

\newcommand{\cF}{{\ensuremath{\mathcal F}} }

\newcommand{\cE}{{\ensuremath{\mathcal E}} }
\newcommand{\cH}{{\ensuremath{\mathcal H}} }
\newcommand{\cC}{{\ensuremath{\mathcal C}} }

\newcommand{\cL}{{\ensuremath{\mathcal L}} }
\newcommand{\cT}{{\ensuremath{\mathcal T}} }
\newcommand{\cD}{{\ensuremath{\mathcal D}} }

\newcommand{\bF}{{\ensuremath{\mathbf F}} }
\newcommand{\bP}{{\ensuremath{\mathbf P}} }

\newcommand{\bE}{{\ensuremath{\mathbf E}} }

\newcommand{\ind}{\mathbf{1}}

\newcommand{\fL}{{\ensuremath{\mathfrak L}} }

\newcommand{\lint}{\llbracket}
\newcommand{\rint}{\rrbracket}

\newtheorem{theorem}{Theorem}[section]
\newtheorem{lemma}[theorem]{Lemma}
\newtheorem{proposition}[theorem]{Proposition}

\newtheorem{rem}[theorem]{Remark}

\newtheorem{assumption}[theorem]{Assumption}
\newtheorem{theorema}{Theorem}

\newcommand{\Tau}{\mathcal{T}}
\newcommand{\RN}[1]{%
  \textup{\uppercase\expandafter{\romannumeral#1}}%
}

\newcommand{\Var}{\mathrm{Var}}
\newcommand{\TV}{\mathrm{TV}}
\newcommand{\Gap}{\mathrm{gap}}

\newcommand{\dd}{\mathrm{d}}
\newcommand{\Mix}{\mathrm{mix}}
\newcommand{\tf}{\textsc{f}}

\renewcommand{\tilde}{\widetilde}
\renewcommand{\hat}{\widehat}

\makeatletter
\def\captionfont@{\footnotesize}
\def\captionheadfont@{\scshape}

\long\def\@makecaption#1#2{%
  \vspace{2mm}
  \setbox\@tempboxa\vbox{\color@setgroup
    \advance\hsize-6pc\noindent
    \captionfont@\captionheadfont@#1\@xp\@ifnotempty\@xp
        {\@cdr#2\@nil}{.\captionfont@\upshape\enspace#2}%
    \unskip\kern-6pc\par
    \global\setbox\@ne\lastbox\color@endgroup}%
  \ifhbox\@ne 
    \setbox\@ne\hbox{\unhbox\@ne\unskip\unskip\unpenalty\unkern}%
  \fi
  \ifdim\wd\@tempboxa=\z@ 
    \setbox\@ne\hbox to\columnwidth{\hss\kern-6pc\box\@ne\hss}%
  \else 
    \setbox\@ne\vbox{\unvbox\@tempboxa\parskip\z@skip
        \noindent\unhbox\@ne\advance\hsize-6pc\par}%
\fi
  \ifnum\@tempcnta<64 
    \addvspace\abovecaptionskip
    \moveright 3pc\box\@ne
  \else 
    \moveright 3pc\box\@ne
    \nobreak
    \vskip\belowcaptionskip
  \fi
\relax
}
\makeatother

\newcommand{\gap}{\mathrm{gap}}

\def\writefig#1 #2 #3 {\rlap{\kern #1 truecm
\raise #2 truecm \hbox{#3}}}

\title[Cutoff  of  SEP  with   inhomogeneous conductances]{Cutoff of the  simple exclusion process   with  inhomogeneous conductances}

\author[Shangjie Yang]{Shangjie Yang} \address{Shangjie Yang \hfill\break
 Instituto de Matemática e Estatística da Universidade de São Paulo \hfill\break
  Rua do Matão 1010,  CEP 05508-090, São Paulo, Brasil.}\email{shjyang@ime.usp.br}
\keywords{Markov chains,   simple exclusion process, inhomogeneous conductances,     mixing time,  cutoff.\hspace{0.5cm} \textit{AMS subject classification}: 60K37; 60J27}

\begin{document}

\maketitle

\begin{abstract} In this paper, we study the  mixing time of  the simple exclusion process  with $k$ particles in the line segment $\lint 1, N \rint$ with conductances $c^{(N)}(x, x+1)_{1\le x<N}$ where $c^{(N)}(x, x+1)>0$ is the rate of swapping the contents of the two sites $x$ and $ x+1$. 
Writing $r^{(N)}(x,  x+1) := 1/c^{(N)}(x,  x+1)$, under the assumption
\begin{equation*}\label{eq32}
 \limsup_{N\to \infty}\, \frac{1}{N}\sup_{1< m \le N}\, \left|  \sum_{x=2}^m r^{(N)}(x-1, x)- (m-1) \right|\;=\;0\,,
\end{equation*} 
and some further assumptions on  $r^{(N)}(x, x+1)_{x \in \mathbb{N} }$ and $k$,  we prove that around time   $(1+o(1)) (2 \pi^2)^{-1} N^2 \log k$, the total variation distance to equilibrium of  the simple exclusion process drops abruptly from $1$ to $0$.
\end{abstract}

\tableofcontents

\section{Introduction}
\subsection{Overview}
The simple exclusion process is one of  simplest interacting particle systems from the perspective of probability and statistical mechanics, which 
 reasonably describes the dynamics of a low density gas.  We refer to \cite[Chapter VIII.6]{liggett2012interacting} for more  introductory and historical  information. 
Its dynamics has been studied from  different viewpoints such as 
  Hydrodynamic limits  \cite{LandimHydrodynamicsbk, Rost81, KOV89, rez91}, 
  Spectral gap \cite{DSC93, Quastel92},  
  $\log$-Sobolev inequalites \cite{Yau97},  Mixing Time \cite{benjamini2005mixing, LPWMCMT,  Morris06, labbe2016cutoff, labbe2018mixing} and so on, where the list of references is far from being exhaustive.

\medskip

All the works mentioned above are about the simple exclusion process in  a \textit{homogeneous} environment. However, when there is impurity in the environment, the pattern of relaxation can change drastically, for which we point to 
\cite{FGS16, FN17} and references therein. Only quite recently, the interest in the \textit{disordered} setup has been flourishing for interacting particle systems (cf.\ \cite{faggionato2008random, FRS19,Lacoin2021aseprandom,  schmid2019mixing}), 
where the jump
rates of particles vary spatially and are random. 
In the sequel, we are mainly concerned with \textit{Mixing time}: given a finite state space continuous-time irreducible Markov chain, how long does it take to reach equilibrium starting from any initial state in total variation distance?
Our goal is to understand how the \textit{disordered} setup affects mixing time,  comparing with the  \textit{homogeneous} setup.

\subsection{The simple exclusion process in random conductance}
Given a finite connected graph $G=(V, E)$ with $V$ being the vertice set and $E$ being the edge set, we place a Poisson clock with rate $c(e)>0$ at each edge $e \in E$, and at each time a Poisson clock rings at an edge $e$, we swap the contents of the two vertices of the edge $e$. Initially, if at each site we place a  particle and all the particles receive a different label, say, from  $1$ to $\vert V\vert$, this model is referred to as the interchange process. If we view particles labeled from $1$ to $k$ as indistinguishable particles and view particles labeled from $k+1$ to $\vert V \vert$ as empty sites, this process is called the simple exclusion process. In particular, when $k=1$, it is the random conductance model. Aldous' spectral gap conjecture states that the spectral gaps of the interchange process and of the random conductance model are the same, confirmed by Caputo, Ligget and Richthammer   \cite{caputo2010aldous}, and thus the spectral gaps of the three processes are the same. Concerning the mixing time of the simple exclusion process,  Oliveira \cite{oliveira2013mixing} shows that it is bounded above by that of the random conductance model up to a $\log \vert V\vert$ factor. 

From now on, our discussion focuses  on the underlying graph being a line segment with $N$ sites (see Figure \ref{fig:SEPrandcond} for  graphical explanation). When the conductances are \textit{homogeneously} equal to one,  in \cite{wilson2004mixing} Wilson uses harmonic analysis to provide a lower bound on the mixing time as $\frac{1+o(1)}{2\pi^2} N^2 \log k$ and an upper bound as $\frac{1+o(1)}{\pi^2} N^2 \log k$,  and conjectures that the lower bound is sharp, which is confirmed by Lacoin \cite{lacoin2016mixing} basing on the eigenfunctions of the Laplace operator and  censoring inequality established by Peres and Winkler  \cite{peres2013can}. When we are in the \textit{disordered} setup: taking $(c(x, x+1))_{1\le x<N }$ IID with $\bbE[1/c(x, x+1)]=1$ and assuming some furthermore assumption on $(c(x, x+1))_{1\le x<N }$ and $k$, we exploit the information about the eigenvalues and eigenfunctions proved in \cite{Yang2024SpectralLight} to show that the $\gep-$mixing time is still $\frac{1+o(1)}{2\pi^2} N^2 \log k$, which implies cutoff. A next natural step is to ask: if we replace a line segment by a circle, can we extend the results in \cite{Lacoin2016profile, Lacoin2017diffusivecutoff} for the \textit{disordered} setup?

\subsection{A short review of related models and results}
\subsubsection{Random walk in random environment}\label{subsec:RWRE}
When the disorder is IID in the sites of the lattice $\bbZ$ and there is only one particle, this is the classical frame work of random walk in random environment where the particle at site $x$ jumps to site $x+1$ at rate $\go_x \in (0, 1)$ and jumps to site $x-1$ at rate $1-\go_x$, for which we refer to  \cite{sznitman2004topics, zeitouni2004part} for references. Concerning the  asymptotic behaviour of the random walk,  Solomon \cite{solomon1975random} showed that if $\bbE \left[   \log \frac{1-\go_x}{\go_x}\right] \neq 0$, the random walk is transient; while if $\bbE \left[   \log \frac{1-\go_x}{\go_x}\right] = 0$, 
the random walk is recurrent. Furthermore, in such a setting, it is also possible that the random walk is transient with sublinear speed (cf.\ \cite{kesten1975limit}), which is impossible for the random walk in \textit{homogeneous} environment.  Moreover, concerning the time required for the random walk restricted in a finite line segment of length $N$ to reach equilibrium, in the transient environment case
Gantert and Kochler  \cite{gantert2012cutoff} showed that it is related to the depth of the deepest trap (see \cite[Eq.\ (58)]{Lacoin2021aseprandom} for the definition of trap) and may be much larger than $N^2 \log N$.

\subsubsection{The simple exclusion process with site  randomness}
In the same setting as Subsection \ref{subsec:RWRE}, we place $k$ particles performing independently continuous-time random walk restricted to a line segment  of length $N$ with  the  \textit{exclusion rule}: any jump to an occupied site is canceled.
We refer to Figure \ref{fig:asep} for a graphical explanation. 

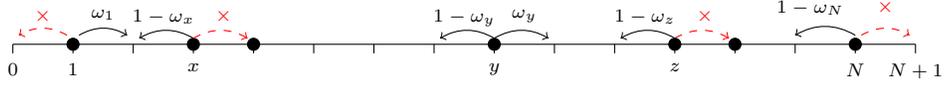
\begin{figure}[h]
 \centering
   \begin{tikzpicture}[scale=.4,font=\tiny]
     \draw (23,-1) -- (53,-1);
     
     \foreach \x in {23,25, 27,...,51,53} {\draw (\x,-1.3) -- (\x,-1);}
     \draw[fill] (51,-1) circle [radius=0.2];   
     \node[below] at (25,-1.3) {$1$};
     \node[below] at (51,-1.3) {$N$};
     \node[below] at (23,-1.3) {$0$};
     \node[below] at (53,-1.3) {$N+1$};

       \draw  (24.8, -0.7)  edge[bend right,->, dashed,red](23.2, -0.7);
       \node[below, red] at (24, 0.5) {$\times$};
       \draw  (25.2, -0.7)  edge[bend left,->](26.8, -0.7);
       \node[below] at (26, 0.5) {$\go_1$};

       \draw[fill] (25,-1) circle [radius=0.2];  
       \node[below] at (29, -1.3) {$x$};
       \node[below] at (28, 0.5) {$1-\go_x$};
       \node[below,red] at (30, 0.5) {$\times$};    

       \draw (51.2, -0.7) edge[bend left,->, dashed,red] (52.8, -0.7);
       \draw (51, -0.7) edge[bend right,->] (49, -0.7);
       \node[below] at (49.5,0.8) {$1-\go_{N}$};
       \node[below,red] at (52,.8) {$\times$};

       \draw[fill] (29,-1) circle [radius=0.2];
   
       \draw (29, -0.8) edge[bend right,->] (27.2, -0.8);
       \draw (29, -0.8) edge[bend left,dashed,->, red] (30.8, -0.8);
       \draw[fill] (31,-1) circle [radius=0.2]; 
       
       \draw (39, -0.8) edge[bend left,->] (40.8, -0.8);   
       \draw (39, -0.8) edge[bend right,->] (37.2, -0.8);      
       \node[below] at (39,-1.3) {$y$};
       \node[below] at (40,0.5) {$\go_y$};
       \node[below] at (38,0.5) {$1-\go_y$};

       \draw[fill] (39,-1) circle [radius=0.2];
         
       \draw (45, -0.8) edge[bend right, ->] (43.2, -0.8);  
       \draw (45, -0.8) edge[bend left,->,dashed,red] (46.8, -0.8);      
       \node[below, red] at (46,0.5) {$\times$};
       \node[below] at (44,0.5) {$1-\go_z$};
       \node[below] at (45,-1.3) {$z$};
   
       \draw[fill] (45,-1) circle [radius=0.2];   
       \draw[fill] (47,-1) circle [radius=0.2];       
 \end{tikzpicture}
   \caption{ A graphical representation of the simple exclusion process in the segment $\lint 1, N \rint$ and environment $\go=(\go_x)_{x \in \bbZ}$: a bold circle represents a particle, and the number above every arrow represents the jump rate while a red $"\times"$ represents a nonadmissible jump.} \label{fig:asep} 
 \end{figure}

 In the transient environment case,  concerning the relaxation pattern there exists a new mechanism peculiar  to this system: particle flow limitation (cf.\ \cite{schmid2019mixing, Lacoin2021aseprandom})  where the particle flow   corresponds to the
inverse of the time that a particle needs to cross half of the deepest trap, for which we refer to \cite[Figure 7]{Lacoin2021aseprandom} for an intuitive argument.  For the recurrent random environment case, we refer to \cite[Section 2.4]{Lacoin2021aseprandom} for  comment and conjecture.

\section{Model and result}\label{sec:1}
\subsection{Models}
\noindent Let $\left(c^{(N)}(i,i+1)\right)_{1 \le i<N}$ be a sequence of strictly positive numbers. For $x<y$ real numbers, we define  
$\lint x,y\rint:=[x,y]\cap \bbZ$.
  We consider
the simple exclusion process with $k$ particles on the segment  $\lint 1, N\rint$ 
with its state space
\begin{equation*}
\gO_{N,k}\;\colonequals\; \left\{ \xi: \lint 1, N \rint \to \{0,1\} \  \bigg| \ \sum_{i=1}^N \xi(i)=k\right\}\,,
\end{equation*}
and with its generator defined by ($f: \gO_{N,k} \mapsto \bbR$)
\begin{equation}\label{generator:SEP}
\left(\cL_{N,k}f\right) (\xi) \; \colonequals \; \sum_{i=1}^{N-1}c^{(N)}(i,i+1) \left[ f(\xi\circ \tau_{i,i+1})-  f(\xi)\right]\,,
\end{equation}
where $\tau_{i,j}$ is the transposition of the two elements $i$ and $j$. Here $1$s represent particles and $0$s represent empty sites.  We refer to Figure \ref{fig:SEPrandcond} for a graphical explanation. 
Due to the symmetry between particles and empty sites, in the paper we always assume the number of particles satisfying $k_N \le N/2$.
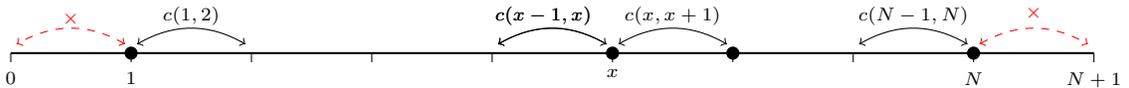
\begin{figure}[H]
 \centering
   \begin{tikzpicture}[scale=.4,font=\tiny]
     \draw[thick] (21,-1) -- (57,-1);
     \foreach \x in {21, 25,...,57} {\draw (\x,-1.3) -- (\x,-1);}
     \node[below] at (25,-1.3) {$1$};
     \node[below] at (53,-1.3) {$N$};
     \node[below] at (21,-1.3) {$0$};
     \node[below] at (57,-1.3) {$N+1$};
       \draw  (24.8, -0.7)  edge[bend right,<->, dashed,red](21.2, -0.7);
       \node[below, red] at (23, 0.7) {$\times$};
       \draw  (25.2, -0.7)  edge[bend left,<->](28.8, -0.7);
       \node[below] at (27, 0.9) {$c(1,2)$};
       \draw[fill] (25,-1) circle [radius=0.2];
       \draw  (40.8, -0.7)  edge[bend right,<->](37.2, -0.7);
       \node[below] at (38.7, 0.9) {$c(x-1,x)$};
       \draw[fill] (41,-1) circle [radius=0.2];
       \node[below] at (41, -1.2) {$x$};
       \draw[fill] (45,-1) circle [radius=0.2];
    \draw  (44.8, -0.7)  edge[bend right, <->](41.2,-0.7);
    \node[below] at (43, 0.9) {$c(x, x+1)$};
        \draw  (40.8, -0.7)  edge[bend right,->](37.2, -0.7);
       \node[below] at (38.7, 0.9) {$c(x-1,x)$};
       \draw[fill] (53,-1) circle [radius=0.2];
  \draw  (52.8, -0.7)  edge[bend right,<->](49.2,-0.7);
    \node[below] at (51, 0.9) {$c(N-1, N)$};
      \draw  (53.2, -0.7)  edge[bend left,<->, red,dashed](56.8,-0.7);
\node[below,red] at (55, 0.9) {$\times$};
 \end{tikzpicture}
 \caption{A graphical explanation for the simple exclusion process in conductances $(c^{(N)}(x, x+1))_{1 \le x<N}$:  at edge $\{x, x+1 \}$ there is a Poisson clock with rate $c^{(N)}(x, x+1)>0$ for all $1\le x<N$, and we swap the contents of the two sites $x, \, x+1$ when the  clock on the edge $\{x, x+1\}$ rings. Note that there is no Poisson clock on the edges $\{0, 1\}$ and $\{N, N+1\}$ (marked "${\red \times}$" in the figure), so that the number of particle is conservative.}\label{fig:SEPrandcond}
   \end{figure}
The uniform probability measure 
$\mu_{N,k}$
on $\gO_{N,k}$,  defined by $\mu_{N,k}(\xi)=1/ \vert \gO_{N,k} \vert$ for all $\xi \in \gO_{N,k}$,  satisfies the detailed balance condition w.r.t.\  $\cL_{N,k}$, and thus it is the unique invariant probability measure. Let 
 $(\eta_t^{\xi})_{t \ge 0}$ 
 denote the c\`{a}dl\`{a}g trajectory of the Markov chain associated with the generator $\cL_{N,k}$ and starting with the  configuration $\xi$, and let 
 $P_t^{\xi}$  denote its marginal distribution  at time instant $t$. Moreover, for $t \ge 0$,  we define the distance to equilibrium from the worst initial state to be
\begin{equation}\label{TVdist}
d_{N,k}(t) \colonequals \max_{\xi \in \gO_{N,k}} \Vert P_t^{\xi}-\mu_{N,k} \Vert_{\TV}
\end{equation}
where 
$\Vert \nu_1-\nu_2 \Vert_{\TV} \colonequals \max_{A \subset \gO_{N,k}} \left( \nu_1(A)-\nu_2(A) \right)\,$ 
    is the
total variation distance between two probability measures $\nu_1, \nu_2$ on $\gO_{N,k}$. For $\gep \in (0,1)$, we define the $\gep$-mixing time to be
\begin{equation}\label{mixtime:kpart}
t_{\Mix}^{N,k}(\gep) \colonequals \inf \left\{ t \ge 0:\,  d_{N,k}(t) \le \gep \right\}\, .
\end{equation}
Moreover, 
we say that this sequence of Markov chains exhibits a cutoff if for all $\epsilon\in (0, 1)$,
\begin{equation}\label{def:cutoff}
\lim_{N \to \infty} \frac{t_{\Mix}^{N, k}(\epsilon)}{t_{\Mix}^{N, k}(1-\epsilon)}\;=\;1\,.
\end{equation}
We refer to  the seminal paper \cite{diaconis1996cutoff} for a survey about cutoff phenomenon. 
We mention \cite{LPWMCMT,Yang2021thesis} for an introduction to the topic of mixing time.

Furthermore,
the rate of relaxation to equilibrium is characterized by the spectral gap, denoted by $ \gap_{N,k}$,  of the generator $ \cL_{N,k}$  which is
 the minimal strictly positive eigenvalue of $ -\cL_{N,k}$.
To characterize it, we define the Dirichlet form associated with the dynamic given by ($f,g: \gO_{N,k} \mapsto \bbR$) 
\begin{equation*}
\mathcal{E}_{N,k}(f) \colonequals -\langle f, \mathcal{L}_{N,k} f\rangle_{\mu_{N,k}}=\frac{1}{2}\sum_{\xi \in \gO_{N,k}}\sum_{x=1}^{N-1} \mu_{N,k}(\xi)c^{(N)}(x,x+1) \left[f(\xi \circ \tau_{x,x+1})-f(\xi) \right]^2\,,
\end{equation*}
where $\langle f, g \rangle_{{\mu_{N,k}}} \colonequals \sum_{\xi \in \Omega_N} \mu_{N,k}(\xi)f(\xi)g(\xi)$ is the usual inner product in $L^2(\gO_{N,k}, \mu_{N,k})$. Moreover, the spectral gap $\gap_{N,k}$ is given by 
\begin{equation*}\label{def:gap}
\gap_{N,k}\; \colonequals\;\,\, \inf_{f \ : \ \Var_{\mu_{N,k}}(f)>0}\,\,\frac{\mathcal{E}_{N,k}(f)}{\Var_{\mu_{N,k}}(f) }
\end{equation*}  
where $\Var_{\mu_{N,k}}(f)\colonequals \langle f, f\rangle_{\mu_{N,k}}-\langle f, \ind \rangle^2_{\mu_{N,k}}$. 
Moreover, the relation between the mixing time  and the spectral gap is
as follows
(cf.\ \cite[Theorem 3.4]{Yang2021thesis}): for all $\gep \in (0, 1)$,
\begin{equation}\label{relation:gapmixtime}
 \frac{1}{\gap_{N,k}} \log \frac{1}{2 \gep} \; \le\; t_{\Mix}^{N,k}(\gep) \;\le\; \frac{1}{\gap_{N,k}} \log \frac{1}{2 \gep \mu_{\min}}
\end{equation}
where $\mu_{\min} \colonequals \min_{\xi \in \gO_{N,k}} \mu_{N,k}(\xi)$, and the distance to equilibrium is related with the spectral gap by (cf.\ \cite[Theorem 3.4]{Yang2021thesis}):
\begin{equation}\label{speed:gap}
\lim_{t \to \infty} \frac{1}{t} \log d_{N,k}(t)\;=\;-\gap_{N,k}\,.
\end{equation}
In this paper, we are interested in the question: how does the mixing time depend on the conductance sequence $(c^{(N)}(x,x+1))_{1 \le x <N}$? In the sequel, we always set
\begin{equation*}
 r^{(N)}(x-1,x)\;\colonequals\; 1/c^{(N)}(x-1,x),\quad  \forall\, x \in \lint 1, N-1\rint\,,
\end{equation*}
 which is referred to as resistance. For simplicity of notations, we drop the superscript "$(N)$" when there is no confusion in the context. In this paper, we always assume
\begin{equation}\label{LLN}
 \limsup_{N\to \infty}\, \frac{1}{N}\sup_{2\le m  \le N}\, \left| (r^{(N)}(1,m)- (m-1) \right|\;=\;0\,.
\end{equation}

We restate the results concerning the spectral gap and principle eigenfunction in \cite{Yang2024SpectralLight} which combine Proposition 1.1 and  Theorem 1.3 of \cite{Yang2024SpectralLight} in the following. 

\begin{theorema}\label{th:gapshapeder} 
Let $g_i$ be the $i$th eigenfunction of $\cL_{N, 1}$ with $g_i(1) \colonequals 1$, whose corresponding eigenvalue is denoted by $-\gl_i$, i.e. $\cL_{N, 1} g_i= -\gl_i g_i$, where 
\begin{equation}\label{eigv:distinct}
0\;=\;\gl_0\;<\;\gl_1\;<\;\gl_2\;<\gl_3\;<\;\cdots\;<\;\gl_{N-1}\,.
\end{equation}
Then $g_1$ is strictly decreasing. 
		If the condition \eqref{LLN} on the resistances holds, for any prefixed constant $K_0$ and  all $1\le i \le K_0$
		we have 
		\begin{equation}\label{gap:asym}
		\lim_{N\to \infty} \frac{N^2 \gl_i }{\pi^2}\;=\;i^2 \, .
		\end{equation}
		Furthermore,  setting 

\begin{equation}\label{fun:h}
h_i(x) \colonequals \cos\left(\frac{i\pi(x   -1/2)}{N}\right), \;\;\; \forall \, x \in \lint 1, N \rint \, ,
\end{equation}			
concerning the shape and (weighted) derivative of the eigenfunctions,	
	 for any prefixed constant $K_0$ and  all $1\le i \le K_0$ we have
\begin{align}
		\lim_{N\to \infty}\; \sup_{x\in \lint 1,\, N\rint}\; &\left| g_i(x) - h_i(x) \right|\;=\;0\,, \label{eigfun:shape}\\	
\lim_{N \to \infty}\; \sup_{x \in \lint 1,\, N \rint}\; &\left \vert  N (c \nabla g_i)(x)- N(\nabla h_i)(x) \right\vert \; = \; 0 \,,	\label{approx:der2eigfuns}
\end{align}
where $(c \nabla f)(x)\colonequals c(x-1, x)[f(x)-f(x-1)]$.
\end{theorema}

\begin{rem}
Note that when the sequence $(r^{(N)}(x-1,x))_{ 2 \le x \le N}$ is IID with common law denoted by $\bbP$ and its expectation $\bbE[r(x, x+1)]=1$, by the strong law of large numbers we have  
\begin{equation*}
\bbP \left( \lim_{N \to \infty}\, \frac{1}{N} \max_{2 \le m \le N}\, \vert r(1,m)-(m-1) \vert=0 \right)\;=\;1\,,
\end{equation*}
and thus the assumption~\eqref{LLN} holds almost surely. 
\end{rem}
Furthermore,  we also need the following assumptions concerning resistances and the number of particles below.
 
\begin{assumption}\label{assum:1}
 There exist  constants $\upsilon\in (0, 1)$ and $C_\bbP>0$  such that 
\begin{equation}\label{upbd:asumresist}
   \max_{1 \le x< N} r(x, x+1) \;\le\; C_\bbP \exp \left( { (\log N)^\upsilon}\right) \,.
  \end{equation} 

\end{assumption}

\begin{rem} The assumption $\upsilon<1$ is almost optimal shown in the one slow bond case studied
in \cite{Yang2024oneslow}.
\end{rem}

\begin{assumption}\label{assum:2}
 There exists a sequence of positive numbers $\left( \bar \Upsilon_N\right)_N >0$  such that 
\begin{equation}\label{assumpweak:lwbdresist}
\min_{1 \le x< N} r(x, x+1) \;\ge\; \bar \Upsilon_N
\end{equation}
where $\lim_{N \to \infty} \bar \Upsilon_N=0$ and  $\lim_{N \to \infty} \bar \Upsilon_N  \log N = \infty$.
\end{assumption}

To state the results about simple exclusion process, we need a further assumption about the number of particles. 
 \begin{assumption}\label{assum:nparticles}
There exists  $\varrho \in (0, 1]$ and $c_{\varrho}>0$ such that  for all $N$ sufficiently large, the number of particles $k_N$ satisfies
\begin{equation*}
c_\varrho N^\varrho \;\le\;  k_N  \;\le\; N/2\,.
\end{equation*} 
 \end{assumption}

\begin{theorem}\label{th:mixexclusion}
Concerning the mixing time of the simple exclusion process associated with generator $\cL_{N,k}$ defined in \eqref{generator:SEP},   under Assumptions \ref{assum:1}, \ref{assum:2}, \ref{assum:nparticles} and  \eqref{LLN}, for all $\gep \in (0,1)$ we have
\begin{equation}\label{theq:mix}
\lim_{N \to \infty}  \frac{ 2 \pi^2 t_{\Mix}^{N,k}(\gep) }{N^2\log k_N } \;=\; 1 \,,
\end{equation}
which implies cutoff.
\end{theorem}

\subsection*{Organization}
Section \ref{sec:prelim} is devoted to preliminaries about the natural partial order on configurations, a graphical construction for Markov chains, and that the spectral gap $\gap_{N, k}$ is independent of $k$. 
  Section \ref{sec:lwbdmixtime} is devoted to the lower bound on the mixing time. 
  Section \ref{sec:upbdmixtime} is devoted to the upper bound on the mixing time.
     Appendix \ref{secapped:cov} is devoted to an upper bound on the covariance of a probability measure used in the lower bound on the mixing time.  Appendix \ref{appsec:heighfun} is devoted to the estimate of the dynamics starting from the maximal height function.  Appendix \ref{app:propmaxconf} is devoted to the proof that the total variation distance between the equilibrium and  the distribution of the Markov chain starting from the maximal path (height function).  

\subsection*{Acknowledgments}
 S.Y.\ is very grateful to  Hubert Lacoin for suggesting this problem and  thanks  Tertuliano Franco and Hubert Lacoin for enlightening and insightful discussions. 
  Moreover, he also thanks sincerely Gideon Amir,  Chenlin Gu, Gady Kozma, and Leonardo Rolla for helpful discussions. S.Y.\ is supported by FAPSP 2023/12652-4.  Partial of this work was done during his stay in Bar-Ilan University supported by  Israel Science Foundation grants 1327/19
and 957/20.

\section{Preliminaries}\label{sec:prelim}
\subsection{Partial order}

For $\xi \in \gO_{N,k}$, we define a height function $h^\xi$ associated with $\xi$ by  $h^\xi(0)\colonequals 0$ and for all   $x \in \lint 1, N \rint$
\begin{equation}\label{heighfun}
h^{\xi}(x) \;\colonequals\; \sum_{y=1}^x  \xi(y) - \frac{k}{N} x  \,.
\end{equation}
The term $k x/N$ is subtracted so that $h^\xi(x)$ has zero mean under equilibrium. 
Furthermore, we identify the image set $\{h^\xi: \xi \in \gO_{N, k}\}$ with $\gO_{N, k}$, as it brings no confusion.
We define a partial order 
$"\le"$ 
on $\gO_{N,k} \times \gO_{N,k}$ as follows: for $\xi, \xi' \in \gO_{N,k}$,
\begin{equation}\label{partorder:gONk}
\left( \xi \;\le\; \xi' \right) \quad \Leftrightarrow 
\quad  \left( \forall\, x \in \lint 1, N \rint\,,\; h^\xi(x)\;\le\; h^{\xi'}(x)\right)\,.
\end{equation}
The maximal and minimal configurations of $\gO_{N, k}$ are respectively
\begin{equation*}
\wedge=\wedge_{N,k} \colonequals \ind_{\{1 \le x \le k\}}\,, \quad \vee=\vee_{N,k} \colonequals \ind_{\{N-k < x \le N\}}\,.
\end{equation*}
The notations of $\wedge$ and $\vee$ are due to the shapes of the maximal and minimal height functions respectively. With an abuse of notation,  $\cL_{N, k}$  also denotes the generator of the Markov chain with state space as  the set of heigh functions.

\subsection{A graphical construction}\label{subsec:graphconstr}

\begin{figure}[h]
\begin{tikzpicture}[scale=0.7]

 \draw[thick,->] (0,0) -- (13,0) node[anchor=north west]{$x$-axis};
\draw[thick,->] (0,-6) -- (0,6) node[anchor=south east] {$y$-axis};   

\draw[black, thick](0,0)--(6,6)--(12,0);
\draw[black, thick](0,0)--(6,-6)--(12,0);
\draw[gray](0,0)--(6,-6)--(12,0);
\draw[gray](1,-1)--(7,5);
\draw[gray](2,-2)--(8,4);
\draw[gray](3,-3)--(9,3);
\draw[gray](4,-4)--(10,2);
\draw[gray](5,-5)--(11,1);
\draw[gray](6,-6)--(12,0);

\draw[gray](1,1)--(7,-5);
\draw[gray](2,2)--(8,-4);
\draw[gray](3,3)--(9,-3);
\draw[gray](4,4)--(10,-2);
\draw[gray](5,5)--(11,-1);
\draw[gray](6,6)--(12,0);

\draw[blue, ultra thick](0,0)--(1,1)--(2,0)--(3,-1)--(4,0)--(5,-1)--(6,0)--(7,1)--(8,0)--(9,1)--(10,2)--(11,1)--(12,0);

\draw[blue, ultra thick, dashed](4,0)--(5,1)--(6, 0);

\draw  (5.1, 0.1)  edge[bend left,<->](5.9, 0.1);

 \foreach \x in {1,2,...,12} {\draw[thick] (\x,0) -- (\x,-0.3);}
 
 \foreach \x in {1,4,6,7,9, 10}
{\draw[fill] (\x,0) circle [radius=0.13];                   }

\node[below] at (1, 0){$1$};
\node[below] at (2, 0){$2$};
\node[below] at (5, 0){$5$};
\node[below] at (12, 0){$12$};
\end{tikzpicture}
\caption { A graphical construction: a solid dot represents a particle and the corresponding height function is given by \eqref{heighfun} with $k=6,\, N=12$. The correspondence between configurations and height functions is given in \eqref{heighfun}. In the figure, swapping the contents of the sites $5$ and $6$ in the configuration corresponds to flipping the corner of the height function at site $5$. }\label{fig:spinsystem}
\end{figure}
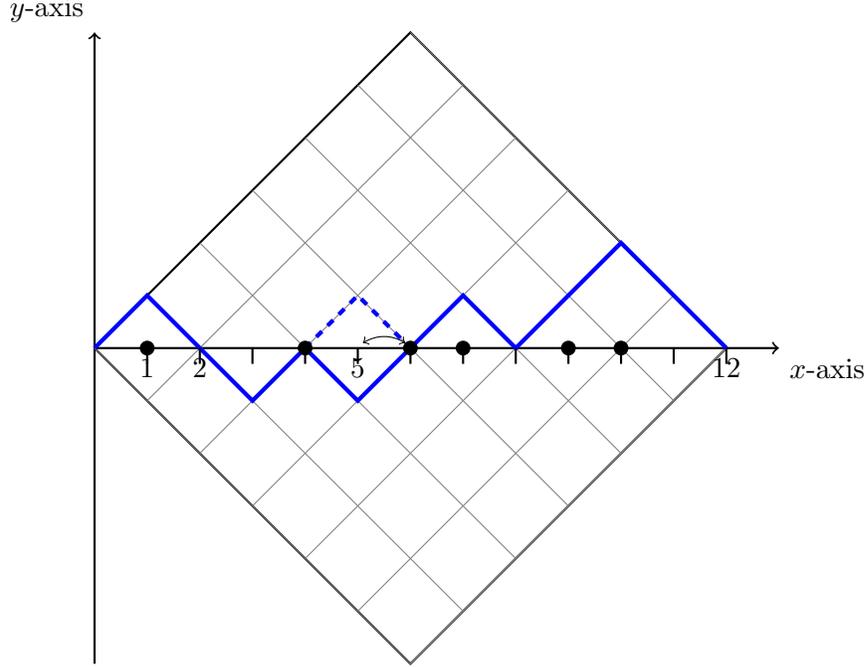
Following \cite[Section 8.1]{lacoin2016mixing},
now we provide a graphical construction for  the dynamics which allows to construct all the trajectories  $( \eta_t^{\xi})_{t \ge 0}$ starting  from all the initial state $\xi \in \gO_{N, k}$
and their associated  height function dynamics $(h_t^{\xi})_{t \ge 0}$  simultaneously, and which is a grand coupling and preserves the monotonicity.

We set the exponential clocks  in the centers of the parallelograms formed by all the possible corners and their counterparts (cf.\ Figure \ref{fig:spinsystem}).
Let
\begin{equation}\label{def:Thetaspins}
\begin{aligned}
\Theta\;\colonequals\;
\left\{(x, z) :  x\in \lint 1,\, N-1 \rint \text{ and }  z\in \{\max\left(0, x-N+k \right)-k x/N, \min \left(x, k \right)-k x/N\}  \right\}\,
\end{aligned}
\end{equation}
be the set of the centers mentioned above. 
Given a realization of $c(x,x+1)_{1\le x<N}$, for each $(x, z) \in \Theta$ we associate two IID Poisson clock processes $ \cT_{(x,z)}^{\uparrow}=(\cT_{(x,z)}^{\uparrow}(n))_{n \ge 0}$ 
and $ \cT_{(x,z)}^{\downarrow}=(\cT_{(x,z)}^{\downarrow}(n))_{n \ge 0}$ 
 where $\cT_{(x,z)}^{\downarrow}(0) =\cT_{(x,z)}^{\uparrow}(0)\colonequals  0$ and for $n \ge 1$, 
\begin{equation*}
\left(\cT_{(x,z)}^{\uparrow}(n)-\cT_{(x,z)}^{\uparrow}(n-1)\right)_{n \ge 1} \quad \text{ and } \quad
\left(\cT_{(x,z)}^{\downarrow}(n)-\cT_{(x,z)}^{\downarrow}(n-1)\right)_{n \ge 1}
\end{equation*}
are two fields of IID exponential distributions with mean $1/c(x,x+1)$.
Let $\bP$ denote the joint law of  the Poisson clock processes, and let $\bE$  denote the corresponding expectation.

\medskip
 
Given a realization of $\left(\cT_{(x,z)}^{\uparrow},\cT_{(x,z)}^{\downarrow} \right)_{(x,z)\in \Theta}$ and   $\xi \in \gO_{N,k}$ (or $h^\xi$), we construct \textemdash in a deterministic fashion  \textemdash $(\eta_t^{\xi})_{t \ge 0}\,,$ $(h^\xi_t)_{t \ge 0}$ the trajectories of the Markov chains starting with $\xi \in \gO_{N,k}$ and its corresponding height function $h^\xi$
as follows.

When the clock process $\cT^{\uparrow}_{(x, z)}$ rings at time $t=\cT_{(x, z)}^{\uparrow}(n)$  for $n\ge 1$  we update $h_{t_-}^\xi$  and  $\eta_{t_-}^{\xi}$  as follows:

\begin{itemize}
\item If $h_{t_-}^{\xi}(x)=z-\frac{1}{2}$ and $h_{t_-}^{\xi}$
has a local minimum at $x$  (equivalent to $\eta_{t_-}^\xi(x)=0$ and $\eta_{t_-}^\xi(x+1)=1$), let $h_{t}^\xi(x)=z+\frac{1}{2}$, $\eta_t^\xi(x)=1$ and  $\eta_t^\xi(x+1)=0$,  and the other coordinates remain  unchanged;
\item If  the conditions above is  not satisfied, we do nothing.
\end{itemize}

When the clock process $\cT^{\downarrow}_{(x, z)}$ rings at time $t=\cT^{\downarrow}_{(x, z)}(n)$ for $n\geq 1$,
 we update $h_{t_-}^\xi$ and  $\eta_{t_-}^{\xi}$ as follows:
\begin{itemize}
\item If $h_{t_-}^\xi(x)=z+\frac{1}{2}$ and $h_{t_-}^\xi$ has a local maximum at $x$ (equivalent to $\eta_{t_-}^\xi(x)=1$ and $\eta_{t_-}^\xi(x+1)=0$), let $h_{t}^\xi(x)=z-\frac{1}{2}$, $\eta_t^\xi(x)=0$ and  $\eta_t^\xi(x+1)=1$,  and the other coordinates remain  unchanged;
\item If  the conditions above is  not satisfied, we do nothing.
\end{itemize}

We also introduce a natural filtration $(\cF_t)_{t\ge 0}$ in this probability space,  for $(x,z) \in \Theta$
defining 
\begin{equation}\label{defizero}
\begin{gathered}
n_0((x, z),t):=\max\{ n\ge 1 \ : \  T_{(x, z)}^{\uparrow}(n)\le t\}\,,\\
m_0((x, z),t):=\max\{ n\ge 1 \ : \ T_{(x, z)}^{\downarrow}(n)\le t\}\,,
\end{gathered}
\end{equation}
with the convention $\max \emptyset=0$ and setting
\begin{equation}\label{filtrations}
 \cF_t \colonequals  \sigma\left(  T_{(x, z)}^{\uparrow}(n),\, T_{(x, z)}^{\downarrow}(m): \, n\le n_0((x, z), t),\, m\le m_0((x, z), t)\,, (x, z) \in \Theta \right)\,.
\end{equation}
By the graphic construction above, the readers can check that the following holds (cf.\ \cite[Proposition 3.1]{lacoin2016mixing} for details). 
\begin{proposition}\label{prop:attravtive}
For $\xi, \xi' \in \gO_{N,k}$, we have
\begin{equation*}
\begin{aligned}
\left( \xi \le \xi'\right)\; \quad  &\Rightarrow \quad \; \bP\left[ \forall\, t \ge 0, \; \eta_t^{\xi} \le \eta_t^{\xi'} \right]\;=\;1\,,\\
\left( h^\xi \le h^{\xi'}\right)\; \quad &\Rightarrow  \quad\; \bP\left[ \forall\, t \ge 0, \; h_t^{\xi} \le h_t^{\xi'} \right]\;=\;1\,.
\end{aligned}
\end{equation*}

\end{proposition}

Note that the graphic construction
above provides a grand coupling. We define  the coalescing time of the two dynamics starting with extremal configurations to be
\begin{equation*}
T \;\colonequals\;\inf \left\{ t \ge 0: \; \eta_t^{\wedge}=\eta_t^{\vee} \right\}\,.
\end{equation*}
By Proposition \ref{prop:attravtive},  for all $t \ge T$ and for all $\xi \in \gO_{N,k}$, we have
\begin{equation*}
\eta_t^{\wedge}\;=\;\eta_t^{\xi}\;=\;\eta_t^{\vee}\,.
\end{equation*}
Therefore, by triangle inequality and \cite[Proposition 4.7]{LPWMCMT} we have
\begin{equation}\label{dist-coupling}
d_{N,k}(t) \le\,  \max_{\xi \in \gO_{N,k}} \sum_{\xi'\in \gO_{N,k}} \mu_{N,k}(\xi') \left\Vert P_t^{\xi}-P_t^{\xi'} \right\Vert_{\TV} \;\le\;
 \max_{\xi,\, \xi' \in \gO_{N,k}} \left\Vert P_t^{\xi}-P_t^{\xi'} \right\Vert_{\TV}  \le \bP \left[ T>t \right]\,.
\end{equation}
The estimation on the coalesce time is the goal of Section \ref{sec:upbdmixtime}.

\subsection{The spectral gap is independent of $k$}

\begin{proposition}\label{prop:eigfunc}
Assuming $c(x, x+1)>0$ for all $1\le x<N$,  concerning the spectral gap of the simple exclusion process we have
\begin{equation}
 \gap_{N, k}\;=\; \gap_{N, 1}\, \mbox{  } \quad \forall\, k \in \lint 1, N-1 \rint\,.
\end{equation}

\end{proposition}

\begin{proof}
 Proposition~\ref{prop:eigfunc} is  a corollary of  \cite[Theorem 1.1]{caputo2010aldous}. For the sake of completeness, we provide a self-contained proof in our setting. Recall that $g_i$ is the $i$-th eigenfunction for $\cL_{N,1}$, that is, the exclusion process with one particle.

	As in \cite[Equation (3.5)]{Hermon_Salez}, we define
	\begin{equation}\label{eigfuns:innprod}
	\tf_i(\xi) \colonequals \sum_{x=1}^N \xi(x)g_i(x)\,, \quad \forall\,\xi \in \gO_{N,k} \text{ and } i \in \lint 1, N-1\rint\,.
	\end{equation}
	We claim now that
	\begin{equation}\label{eigfunclaim:innprod}
	(\cL_{N,k} \tf_i)(\xi)=-\gl_i \tf_i(\xi)\,, \quad \forall\,\xi \in \gO_{N,k} \text{ and } i \in \lint 1, N-1\rint\,.
	\end{equation}
	For all $x \in \lint 1, N \rint$, let $H_x: \gO_{N,k} \mapsto \{0,1 \}$ be given by $H_x(\xi)=\xi(x)$. Observe that
	\begin{equation}\label{gen-comp1}
	\begin{cases}
	(\cL_{N,k} H_1)(\xi)
	=c(1,2)\left( \xi(2)-\xi(1)\right)\,,\\
	(\cL_{N,k} H_x)(\xi)=c(x,x+1) \left(\xi(x+1)-\xi(x)  \right)+c(x-1,x) \left( \xi(x-1)-\xi(x) \right)\,, \quad \forall\, x \in \lint 2, N-1\rint\,,\\
	(\cL_{N,k}H_N)(\xi)=c(N-1,N) \left(\xi(N-1)-\xi(N) \right)\,,
	\end{cases}
	\end{equation}
	and for all $x \in \lint 1, N \rint$
	\begin{equation}\label{gen-comp2}
	(\cL_{N, 1}g_i)(x)\;=\;c(x, x+1)\left[    g_i(x+1)-g_i(x)\right]-c(x-1, x)\left[    g_i(x)-g_i(x-1)\right]\;=\; -\gl_i g_i(x)\,. 
	\end{equation}
	Then by \eqref{gen-comp1}, \eqref{gen-comp2} and reorganizing the terms, we obtain
	\begin{equation}\label{eigfun:exclusionpro}
	(\cL_{N,k}\tf_i)(\xi)\;=\;\sum_{x=1}^N (\cL_{N,k} H_x)(\xi) g_i(x)
	\;=\;-\gl_i \sum_{x=1}^N \xi(x)g_i(x) \;=\;-\gl_i \tf_i(\xi)\,.
	\end{equation}
	Therefore, $(\tf_i)_{1 \le i \le N-1}$ are eigenfunctions of   $\cL_{N,k}$ with corresponding eigenvalues $(-\gl_i)_{i=1}^{N-1}$, and then 
\begin{equation}\label{upbd:spectralgap-excl}
\Gap_{N,k} \;\le\; \gl_1 \;=\;\Gap_{N,1}\,, \quad \forall \,  k \in  \lint 1, N-1\rint\,.
\end{equation}

	Now we move to show that $\Gap_{N,k} \ge \Gap_{N,1}$. 
	Writing $\tf (\xi)= \tf_1(\xi)= \sum_{x=1}^N \xi(x) g_1(x)$ and
$
\delta_{\min} \;\colonequals \; \min_{1\le x<N} g_1(x)-g_1(x+1)\;>0\,,
$	
	 we have
	\begin{equation*}
	\begin{gathered}
	\tf(\xi) \;\le\; \tf(\xi')\,, \quad \forall\, \xi \le \xi' \,,\\
	\delta_{\min}\;=\; \min_{\xi \le \xi', \hspace{0.05cm} \xi \neq \xi'} \tf(\xi')- \tf(\xi)\;>\;0\,.
	\end{gathered}
	\end{equation*}
	Moreover, recalling Proposition~ \ref{prop:attravtive} and \eqref{dist-coupling}, we have
	\begin{align*}
	d_{N,k}(t)\le \bP \left( \eta_t^{\wedge_{N,k}} \neq \eta_t^{\vee_{N,k}}  \right) \;&\le\; \frac{1}{\delta_{\min}} \bE \left[ \tf\left(\eta^{\wedge_{N,k}}_t\right)- \tf\left(\eta_t^{\vee_{N,k}}\right)\right]\;\\
	&=
	\;\frac{1}{\delta_{\min}} e^{-\gl_1 t} \left[ \tf\left(\wedge_{N,k}\right)-\tf\left(\vee_{N,k}\right) \right]\,,
	\end{align*}
	where  we used Markov inequality for the inequality and $\cL_{N,k}\tf=-\gl_1 \tf$ for  the  equality.
	Then by \eqref{speed:gap}, we have 
	\begin{equation}\label{lwbd:spectralgap-excl}
	-\Gap_{N,k}\;=\; \lim_{t \to \infty} \frac{1}{t}\log d_{N,k}(t)\;
	 \le \; \limsup_{t \to \infty} \frac{1}{t}  \log \left( \frac{1}{\delta_{\min}} e^{-\gl_1 t} \left[ \tf(\wedge_{N,k})-\tf(\vee_{N,k}) \right] \right)\;=\;-\Gap_{N,1}\,.
	\end{equation}
Combining \eqref{upbd:spectralgap-excl} and \eqref{lwbd:spectralgap-excl},  we obtain 
	\begin{equation}\label{gaps1k}
\Gap_{N,k} \;=\;\Gap_{N,1}\,, \quad \forall \,  k \in  \lint 1, N-1\rint\,.
	\end{equation}
\end{proof}

By Proposition \ref{prop:eigfunc}, we simply  write $\gap_N=\gap_{N, k}$, as the spectral gap does not depend on $k$. Furthermore, 
the combination of 
Theorem~\ref{th:gapshapeder}  and Proposition \ref{prop:eigfunc} implies  that  under the assumption \eqref{LLN} the spectral gap of the simple exclusion process satisfies
\begin{equation}
\lim_{N \to \infty}\, N^2 \gap_{N, k} \;=\; \pi^2\,.
\end{equation}

\section{The lower bound on the mixing time}\label{sec:lwbdmixtime}
In this section, our aim is to show the following proposition. 

\begin{proposition} \label{prop:mixlow}
If \eqref{LLN} holds and there exists $\gga>0$ such that  $ (\log N)^{1+\gga} \le k_N \le N/2$, for all $\gep \in (0, 1)$ for all $N$ sufficiently large we have
\begin{equation}\label{lwbd:mixexclu}
t_{\Mix}^{N,k}(\gep) \;\ge\; \frac{1}{2\gl_1} \log k_N-\frac{ C(\gep)}{\gl_1}\,.
\end{equation}

\end{proposition}

\subsection{Preliminaries for the lower bound on mixing time}\label{subsec:lwbd-mixing}
Our strategy is to combine the second moment method with the martingale approach. Note that  the function
\begin{equation}\label{eigfun:inner-principle}
\tf(\xi)=\tf_1(\xi)=\sum_{x=1}^N \xi(x) g_1(x)
\end{equation}
 is an eigenfunction satisfying $\cL_{N,k} \tf=-\gap_N \cdot \tf$, which is proved in \eqref{eigfunclaim:innprod}. 
 Let $(\eta_t^\nu)_{t \ge 0}$ be the dynamics constructed by the graphical construction with initial state sampled from a probability measure $\nu$ on $\gO_{N,k}$. Here $\bP$ (or $\bE$), with an abuse of notation, denotes the joint law of $\nu$  and the random source in the graphical construction.
For any positive number $\ell_N(\gep) >0$, 
\begin{equation}\label{lwbd:mixphil}
\begin{aligned}
d_{N,k}(t_0) & \;\ge\; \bP \left[ \tf\left(\eta_{t_0}^{\nu}\right) \ge \ell_N(\gep)\right]
-\bP \left[ \tf\left(\eta_{t_0}^{\mu}\right) \ge \ell_N(\gep) \right]\\
&\;=\; 
\bP \left[  \tf\left(\eta_{t_0}^{\nu}\right)- \bE\left[\tf\left(\eta_{t_0}^{\nu}\right) \right]  \ge  \ell_N(\gep)-\bE\left[\tf(\eta_{t_0}^{\nu}) \right]  \right]-
\bP \left[  \tf\left(\eta_{t_0}^{\mu}\right)- \bE\left[\tf(\eta_{t_0}^{\mu}) \right]  \ge \ell_N(\gep) \right]
\end{aligned}
\end{equation}
where we have used
\begin{equation}\label{mean:stat-eigfun}
\bE\left[\tf(\eta_{t_0}^{\mu}) \right] \;=\;\mu_{N,k}\left(\tf\right)\;=\;\frac{k}{N} \sum_{x=1}^N  g_1(x) \;=\; 0\,,
\end{equation}
as  $\mu_{N,k}$ is uniform on $\gO_{N,k}$ and  $\langle g_1, \ind \rangle_{\mu_{N,k}}=0$.

 In order to estimate $\bE[\tf(\eta_{t_0}^{\nu})]$,  
and $\tf(\eta_{t_0}^{\nu})-\bE[\tf(\eta_{t_0}^{\nu})]$, we use the martingale approach. For fixed $t_0>0$, we define
\begin{equation*}
F(t, \xi) \colonequals e^{\gl_1(t-t_0)} \tf(\xi)\,, \quad \forall \;  \xi \in \gO_{N,k} \,,
\end{equation*}
and set
\begin{equation}\label{def:martin}
M_t \;\colonequals\; F(t,\eta_t^{\nu})-F(0,\eta_0^{\nu})-\int_0^t \left( \partial_s+\cL_{N,k} \right)F(s, \eta_s^{\nu}) \; \dd s \,.
\end{equation}
Note that $(M_t)_{t \ge 0}$ is a Dynkin martingale w.r.t.\ the   filtration $(\cF_t)_{t \ge 0}$ defined in \eqref{filtrations} (cf.\ \cite[Lemma 5.1 in Appendix 1]{LandimHydrodynamicsbk}).
Moreover, as $\cL_{N,k} \tf=-\gap_{N,k} \cdot \tf$, we have
\begin{equation*}
\left( \partial_t+\cL_{N,k} \right)F(t, \xi)\;=\;\gl_1 F(t,\xi)-\gl_1 F(t,\xi)\;=\;0\,,
\end{equation*}
and thus $M_t=F(t,\eta_t^{\nu})-F(0,\eta_0^{\nu})$, which implies
\begin{equation}\label{mart:meanmaxicon}
\bE \left[\tf\left(\eta_{t_0}^{\nu}\right) \right]\;=\;\bE \left[F\left(t_0, \eta_{t_0}^{\nu}\right) \right]\;=\;\bE \left[ F(0, \eta_{0}^{\nu})\right]\; =\; e^{-\gl_1 t_0} \bE \left[\tf \left(\eta_{0}^{\nu} \right) \right]\,.
\end{equation}
Now we discuss $\bE\left[\tf(\eta_{t_0}^{\nu})-\bE\left[\tf(\eta_{t_0}^{\nu})\right]\right]^2$ according to whether $\nu$ is degenerated as follows:

\begin{enumerate}

\item
If $\nu$ concentrates at one configuration, then
\begin{equation}\label{degvar:premierdemi}
 \bE\left[\tf(\eta_{t_0}^{\nu})-\bE\left[\tf(\eta_{t_0}^{\nu})\right]\right]^2\;=\;
\bE \left[ F(t_0, \eta_{t_0}^\nu)-F(0, \eta_0^\nu)  \right]^2 \;=\;
 \bE \left[ M_{t_0}^2\right]\,,
\end{equation}
where we have used \eqref{mart:meanmaxicon} and $\bE \left[F(0, \eta_0^\nu) \right]=F(0, \eta_0^\nu)$ in the first equality.

\item
 If $\nu$ is non-degenerated, we have 
\begin{equation}\label{nondegvar:premierdemi}
\begin{aligned}
 \bE\left[\tf(\eta_{t_0}^{\nu})-\bE\left[\tf(\eta_{t_0}^{\nu})\right]\right]^2\;&=\;
  \bE\left[F(t_0, \eta_{t_0}^{\nu})- F(0, \eta_{0}^{\nu})+F(0, \eta_{0}^{\nu})-
  \bE\left[F(0, \eta_0^{\nu})\right]\right]^2\\
  \;& \le \; 2\bE \left[M_{t_0}^2 \right]+ 2\bE \left[ F(0, \eta_{0}^{\nu})-
  \bE\left[F(0, \eta_0^{\nu}) \right]\right]^2 \,.
  \end{aligned}
\end{equation}
 \end{enumerate}
We estimate $\bE[ M_t^2]$ by controlling the martingale bracket $\langle M. \rangle$, which is such that the process $\big(M^2_t- \langle M. \rangle_t\big)_{t\geq 0}$ is a martingale w.r.t.\ the   filtration $(\cF_t)_{t \ge 0}$. 
Moreover,  a transition  
across the edge $\{ n,n+1\}$ at time $t$
 changes the value of $M_t$ in absolute value by  $e^{\gl_1(t-t_0)}  \vert g(n+1)-g(n)\vert$. By 
 \cite[Lemma 5.1 in Appendix 1]{LandimHydrodynamicsbk}, writing
 \begin{equation}\label{change-rate}
 \overline \eta^{\nu}_t(x,x+1) \colonequals  \eta^{\nu}_t(x)\left(1-\eta^{\nu}_t(x+1) \right)+ \eta^{\nu}_t(x+1)\left(1-\eta^{\nu}_t(x) \right)
 \end{equation}
  we have
\begin{equation*}
\partial_t \langle M. \rangle_t \;=\;   e^{2\gl_1(t-t_0)} \sum_{x=1}^{N-1}  \overline \eta^{\nu}_t(x,x+1) r(x,x+1)  \left[ c(x,x+1)(g(x)-g(x+1))\right]^2\,.
\end{equation*}
By Theorem \ref{th:gapshapeder}--\eqref{approx:der2eigfuns}, for all $N$ sufficiently large, we have
\begin{equation}\label{increment}
c(n,n+1)  \vert g(n+1)-g(n) \vert \;\le\; \vert h(n+1)-h(n) \vert+ \frac{\pi}{N}\ \;\le\;\frac{ 2\pi}{N}\,,
\end{equation}
and then 
\begin{equation}\label{upbd:martbracket}
\begin{aligned}
\partial_t \langle M. \rangle_t 
&\;\le\; \frac{4 \pi^2}{N^2}
e^{2\gl_1(t-t_0)}  \sum_{x=1}^{N-1} \overline \eta^{\nu}_t(x,x+1) r(x, x+1) \,.
\end{aligned}
\end{equation}
Therefore, 
\begin{equation}\label{upbd:2ndmomentMart}
\begin{aligned}
\bE\left[M_{t_0}^2\right]\;=\;&\bE\left [\langle M. \rangle_{t_0}\right]
\;\le\;& \int_0^{t_0} \frac{4 \pi^2}{N^2} e^{2\gl_1(t-t_0)}  \sum_{x=1}^{N-1} \bE \left[ \overline \eta^{\nu}_t(x,x+1)\right] r(x,x+1)   \dd t\,.
\end{aligned}
\end{equation}

Concerning the last term in the r.h.s.\ of \eqref{lwbd:mixphil}, we have \eqref{mean:stat-eigfun} and
\begin{equation*}\label{var:stat-eigfun}
\bE \left[ \tf(\eta_{t_0}^{\mu})^2\right]
\;=\;\sum_{x=1}^N \frac{k}{N} g(x)^2+ \sum_{x=1}^N \sum_{y: y \neq x} \frac{k(k-1)}{N(N-1)} g(x)g(y)\;=\;\frac{k(N-k)}{N(N-1)} \sum_{x=1}^N g(x)^2\,,
\end{equation*}
where we have used $\sum_{x=1}^N g(x)=0$ in the last equality. 
Furthermore, by $k \le N/2$ and Theorem \ref{th:gapshapeder}--\eqref{eigfun:shape}, for all $N$ sufficiently large we have
\begin{equation}\label{upbd3:2ndmomentMart}
\bE\left[ \tf(\eta_{t_0}^{\mu})^2 \right]
\;=\; \frac{k(N-k)}{N(N-1)} \sum_{x=1}^N g(x)^2 \;\le\; \frac{k(N-k)}{N(N-1)} \cdot 2N \;\le \; 2k\,.
\end{equation}

We choose $t_0$ relatively small but large so that
  $\bE \left[ \tf(\eta_{t_0}^{\nu}) \right]$ is large, and we will show that both \eqref{upbd:2ndmomentMart}  and   $\bE \left[ F(0, \eta_{0}^{\nu})-
  \bE\left[F(0, \eta_0^{\nu}) \right]\right]^2$ in the r.h.s.\ of \eqref{nondegvar:premierdemi} are small.
Then we choose $\ell_N(\gep)$ close to $\bE \left[\tf( \eta_{t_0}^{\nu}) \right]$ so that 
the r.h.s.\ of \eqref{lwbd:mixphil} can be bounded from below by $1-\gep$. We carry out the upper bounds on  \eqref{upbd:2ndmomentMart}  and   $\bE \left[ F(0, \eta_{0}^{\nu})-
  \bE\left[F(0, \eta_0^{\nu}) \right]\right]^2$, a lower bound on \eqref{mart:meanmaxicon} according $N/64\le k \le N/2$ or $k<N/64$ in the following subsection.

\subsection{Proof of proposition \ref{prop:mixlow}}

   We deal with \eqref{lwbd:mixexclu}, based on  \eqref{mart:meanmaxicon}, \eqref{upbd:2ndmomentMart} and \eqref{upbd3:2ndmomentMart}. 

Now we move to deal with  \eqref{mart:meanmaxicon} and \eqref{upbd:2ndmomentMart} according to $ k \in [N/64,\, N/2]$ and $ k < N/64$. 

\begin{proof}[Proof for the case $ N/64\le k \le  N/2$]
We first deal with the case $k \in [N/64,\, N/2]$ by taking $\nu=\delta_\wedge$ (Dirac probability measure centered at $\wedge_{N, k}$).
 By Theorem \ref{th:gapshapeder} and \eqref{fun:h}, for all $\delta>0$  sufficiently small  there exists $\vartheta=\vartheta(\delta) \in (0, 1/6)$ small  such that for all $N$ sufficiently large we have
\begin{equation}\label{diff:eiffunrandder}
\begin{aligned}
&\max_{x \in \lint 1,\, N \rint} \;\vert g(x)-h(x) \vert\;<\; \delta\,,\\
&\min_{x \in \lint 1, \, N/2 \rint}\;   h(x) \;> \; 0\,, \\ 
&\min_{x \in \lint 1, \, (\frac{1}{2}-\vartheta)N \rint}\; h(x) \;\ge\; 3 \delta\,.
\end{aligned}
\end{equation}
 By \eqref{diff:eiffunrandder}
we have
\begin{equation*}
\begin{gathered}
\min_{x \in \lint 1,\, N/2 \rint}\;   g(x)  \; \ge\; - \delta\,, \\
\min_{x \in \lint 1,\, (\frac{1}{2}-\vartheta)N \rint}\; g(x) \;\ge\; 2 \delta\,.
\end{gathered}
\end{equation*}
Then by $k \le N/2$ and the monotonicity of $g_1$,  we have
\begin{equation*}
\tf(\wedge_{N,k})\;=\;\sum_{x=1}^k g_1(x)
\;\ge\; \sum_{x=1}^{ \lfloor k/2 \rfloor } g_1(x) + \sum_{x= \lfloor k/2 \rfloor+1}^{\lfloor (1-2 \vartheta)k \rfloor}  2\delta -\sum_{x=\lfloor (1-2 \vartheta)k \rfloor+1}^{k} \delta
\;\ge \;
\left( \frac{\sqrt{2}}{2}-\delta \right) \left\lfloor \frac{k}{2} \right\rfloor
\;\ge \; \frac{k}{5}
\end{equation*}
where we have used $\vartheta<1/6$ and $\delta>0$ is chosen to be sufficiently small such that the last inequality holds.
Taking $t_0=\frac{1}{2\gl_1} \log k- \frac{C(\gep)}{\gl_1}$  in \eqref{mart:meanmaxicon}, 
we have
\begin{equation}\label{2ndmom:maxconf}
\bE \left[ \tf(\eta_{t_0}^{\wedge})\right]\;=\;e^{-\gl_1 t_0} \tf(\wedge_{N,k}) 
\;\ge\; e^{-\gl_1 t_0} \frac{k}{5} 
\;=\; \frac{1}{5}e^{C(\gep)} \sqrt{k} 
\;\ge\; 8 \sqrt{ \frac{N}{\gep}}
\end{equation}
where we have used $k \ge N/64$ and $C(\gep)>0$ is chosen to be large enough such that the last inequality holds.
Now we turn to upper bounds on the r.h.s.\ of \eqref{upbd:2ndmomentMart}.
 Using  $\overline \eta^\wedge_t(x, x+1)\le 1$  and
$r(1,N) \le 2N$  by \eqref{LLN} we have

\begin{equation}\label{upbd2:2ndmomentMart}
\bE\left[M_{t_0}^2\right]\;=\;\bE\left [\langle M. \rangle_{t_0}\right] \;\le\; \int_0^{t_0}  e^{2\gl_1(t-t_0)} \frac{8\pi^2}{N} \dd t
\;\le\;\frac{4 \pi^2}{N \gl_1} \; \le \;
 8 N \,,
\end{equation}
where  we have used Theorem --\eqref{gap:asym} in the last inequality. 

In \eqref{lwbd:mixphil}, we take  $\ell_N(\gep) \colonequals 4 \sqrt{N/\gep}$, and by \eqref{degvar:premierdemi},  \eqref{2ndmom:maxconf}, \eqref{upbd2:2ndmomentMart} and \eqref{upbd3:2ndmomentMart}  to  obtain
\begin{equation}\label{lwbd:tvdistLLNBigk}
\begin{aligned}
d_{N,k}(t_0) 
& \;\ge\; 
\bP \left[  \tf\left(\eta_{t_0}^{\wedge}\right)- \bE\left[\tf\left(\eta_{t_0}^{\wedge}\right) \right]  \ge  \ell_N(\gep)-\bE\left[\tf(\eta_{t_0}^{\wedge}) \right]  \right]-
\bP \left[  \tf\left(\eta_{t_0}^{\mu}\right)- \bE\left[\tf(\eta_{t_0}^{\mu}) \right]  \ge \ell_N(\gep) \right]\\
&\;\ge\; 
\bP \left[ \vert \tf(\eta_{t_0}^{\wedge})- \bE\left[\tf(\eta_{t_0}^{\wedge}) \right] \vert \le 4 \sqrt{\frac{ N}{\gep}} \right]-
\bP \left[ \vert \tf(\eta_{t_0}^{\mu})- \bE\left[\tf(\eta_{t_0}^{\mu}) \right] \vert \ge 4 \sqrt{ \frac{N}{\gep}} \right]\\
&\;\ge \; 1-\frac{\gep}{2}-\frac{\gep}{2}\,,
\end{aligned}
\end{equation}
which allows to  conclude the proof for the case $k \in [N/64,\,  N/2]$.
\end{proof}

\begin{proof}[Proof for the case $k <N/64$]
Now we move to the case $k <N/64$. A nondegenerate probability measure $\nu$ that we plug in \eqref{nondegvar:premierdemi} is as follows:
  we first sample a configuration $\xi \in \gO_{N, 2k}$ according to 
 $\mu_{N, 2k}$ which is the uniform distribution on $\gO_{N, 2k}$,  and then the first $k$ particles are kept (counting from left to right) and the other particles are projected to empty sites.  
 By \eqref{upbd:2ndmomentMart} and 
$\overline \eta^{\nu}_t(x,x+1) \le \eta^{\nu}_t(x)+\eta^{\nu}_t(x+1)$, we have

 \begin{equation}\label{upbd:martquad}
 \begin{aligned}
\bE\left[M_{t_0}^2\right]\;=\;\bE\left [\langle M. \rangle_{t_0}\right] \;&\le\; \int_0^{t_0} e^{2\gl_1(t-t_0)} \frac{4 \pi^2}{N^2} \sum_{x=1}^{N-1}  \bE\left[\eta_t^\nu(x)+\eta_t^\nu(x+1)\right]  r(x,x+1)  \dd t\\
&\le\; 
\int_0^{t_0} e^{2\gl_1(t-t_0)} \frac{4 \pi^2}{N^2} \sum_{x=1}^{N-1}  \frac{4k}{N}  r(x,x+1)  \dd t \;\le \; 32 k\,,
\end{aligned}
\end{equation}
 where we have used Theorem \ref{th:gapshapeder}--\eqref{gap:asym}, \eqref{LLN}, and $\eta_t^{\nu}(x) \le  \eta_t^{\mu_{N, 2k}}(x)$ by coupling the two processes using the graphical construction.

 In order to provide a lower bound on the last term of \eqref{mart:meanmaxicon},
 given $\xi \in \gO_{N, 2k}$, we define  $\bar \xi \ : \  \lint 1, 2k\rint \mapsto \lint 1, N\rint$ as the  positions of the particles of $\xi$ from left to right:
\begin{equation}\label{partiposition}
 \{\bar \xi(i)=x\}  \quad  \Longleftrightarrow  \quad  \left\{\xi(x)=1 \text{ and } \sum_{y=1}^x \xi(y)=i \right\}.
\end{equation}
Setting $\chi \colonequals 17/16$ and
\begin{equation*}
\cA \colonequals  \left\{\xi \in \gO_{N, 2k}: \; \bar \xi\left( \left\lfloor\frac{k}{2}\right\rfloor\right) \le \frac{N \chi}{4}\,, \; \bar \xi\left(\left\lfloor\frac{3k}{4}\right\rfloor\right) \le \frac{3N \chi}{8}\,,  \; \bar \xi\left(k\right) \le \frac{N \chi}{2} \right\}\,,
\end{equation*}
by the monotone decreasing property of $g$ stated in Theorem~\ref{th:gapshapeder} we have 

\begin{equation}\label{lwbd:initial-heignf}
\begin{aligned}
&\bE \left[ \tf(\eta_0^\nu)\right]\;=\; \sum_{x=1}^N \nu\left( \xi(x) \right) g(x) \; \ge \;  \mu_{N, 2k}\left( \cA^\complement \right) k g(N)\\
 + &\mu_{N, 2k} \left( \cA\right)
\left[  \left\lfloor\frac{k}{2}\right\rfloor \cdot g\left( \left\lfloor \frac{N\chi}{4} \right\rfloor \right) + \left(\left\lfloor\frac{3k}{4}\right\rfloor -\left\lfloor\frac{k}{2}\right\rfloor\right)\cdot g\left(\left\lfloor \frac{3N\chi}{8} \right\rfloor \right)+ 
\left(k -\left\lfloor\frac{3k}{4}\right\rfloor\right)\cdot g\left(\left\lfloor \frac{N\chi}{2} \right\rfloor  \right) \right]\,.
\end{aligned}
\end{equation} 
By Theorem \ref{th:gapshapeder}--\eqref{eigfun:shape}, for any given $\delta>0$, for all $N$ sufficiently large we have 
\begin{equation}\label{bd:g4p}
\begin{aligned}
g(N) \;&\ge\; -(1+\delta)\,, 
\quad \quad  \quad \quad \quad 
g\left( \left \lfloor \frac{N \chi}{4}\right\rfloor \right) \;\ge\; \cos\left(\frac{\chi}{4}\pi \right)-\delta\,,\\
g\left(\left\lfloor \frac{3N \chi}{8}\right\rfloor \right) &\ge \; \cos\left(\frac{3 \chi}{8}\pi \right)-\delta\,,
\quad \quad
g\left(\left\lfloor \frac{N \chi}{2} \right\rfloor \right) \ge \; \cos\left(\frac{\chi}{2}\pi \right)-\delta\,,\\
\end{aligned}
\end{equation}
and  
\begin{equation}\label{upbd:smallprobmu2k}
\mu_{N, 2k}\left(\cA^\complement\right) \le 
\mu_{N, 2k} \left( \bar \xi\left( \left\lfloor\frac{k}{2}\right\rfloor\right) > \frac{N\chi}{4} \right)
+
\mu_{N, 2k} \left( \bar \xi\left(\left\lfloor\frac{3k}{4}\right\rfloor\right) > \frac{3N\chi}{8}  \right)
+
\mu_{N, 2k} \left( \bar \xi\left( k \right) > \frac{N\chi}{2}  \right)\,.
\end{equation}
Concerning the first term in the r.h.s.\ of \eqref{upbd:smallprobmu2k}, we have
\begin{align}
\mu_{N, 2k} \left( \bar \xi\left( \left\lfloor\frac{k}{2}\right\rfloor\right) > \frac{N\chi}{4} \right) \;&=\;  {N\choose 2k}^{-1} \;\sum_{x=0}^{\left\lfloor\frac{k}{2}\right\rfloor-1}\; { \left \lfloor \frac{N \chi}{4}  \right\rfloor \choose x} { N-\left \lfloor \frac{N\chi}{4} \right\rfloor  \choose 2k-x} \,,\label{upbd:1stprob}
\end{align}
Note that the following map with $\left \lfloor \frac{N\chi}{4}  \right\rfloor \le n \le \left \lfloor \frac{N\chi}{2}  \right\rfloor$ and  $x \in \lint 0,\, k \rint$,
\begin{equation}\label{binomial:map}
x \mapsto {n \choose x} {N-n \choose 2k -x}\,
\end{equation}
is increasing if and only if 
\begin{equation*}
{n \choose x+1} {N-n \choose 2k-(x+1)}  \left[{n \choose x} {N-n \choose 2k-x)}\right]^{-1} \;\ge\; 1
\end{equation*}
which is equivalent to 
\begin{equation}\label{increase:range}
x \;\le\; \frac{2k(n+1)-(N-n)-1}{N+2}
\;\lesssim\; \frac{2 k n}{N}\,,
\end{equation}
where the last inequality follows from $\liminf_{N \to \infty} k_N = \infty$ and $\left \lfloor \frac{N\chi}{4}  \right\rfloor \le n \le \left \lfloor \frac{N\chi}{2}  \right\rfloor$. 
In order to estimate binomial coefficients, we define a function
\begin{equation*}
 H_2(y)\; \colonequals\; -y \log_2 y -(1-y) \log_2 (1-y)\,, \quad y \in [0, \, 1]\,,
\end{equation*}
 which satisfies
 \begin{equation*}
 \begin{gathered}
 H'_2(y) \;=\; -(\log 2)^{-1}\left[\log y -\log (1-y) \right]\,,\\
 H''_2(y) \;=\; -(\log 2)^{-1} \frac{1}{y(1-y)} \; \le \; - \frac{4}{\log 2} \,.
 \end{gathered}
 \end{equation*}
Therefore, the function $\tilde H_2(y) \colonequals H_2(y)+\frac{2}{\log 2} y^2$ with $y \in [0, 1]$ is concave, that is, for $\theta, x, y  \in [0, 1]$,
\begin{equation*}
\theta \tilde H_2(x)+(1-\theta) \tilde H_2(y) \le \tilde H_2(\theta x+ (1-\theta)y)\,,
\end{equation*}
from which we obtain
\begin{equation}\label{convex:diff}
\theta H_2(x)+(1-\theta)H(y)-H(\theta x+(1-\theta)y) \;\le\; - \frac{2}{\log 2} \cdot  \theta (1-\theta) (x-y)^2\,.
\end{equation}
Moreover, by \cite[Lemma 7 in Chapter 10]{MacWilliams1977error}  we have
\begin{equation}\label{bd:binomial}
\frac{2^{N H_2(m/N)}}{\sqrt{8 \pi m \left(1-\frac{m}{N} \right) }} \;\le\;  {N \choose m} \;\le\; \frac{2^{N H_2(m/N)}}{\sqrt{2 \pi m \left(1-\frac{m}{N} \right) }}\,,\;\; \forall\, m \in \lint 1,\, N-1\rint\,.
\end{equation}

In order to treat \eqref{upbd:1stprob}, in view of \eqref{bd:binomial} we define
\begin{equation}\label{1stbn:IN}
\begin{aligned}
I_N \; &\colonequals\; 2k H_2 \left( \frac{ \left \lfloor \frac{k}{2} \right \rfloor}{2k} \right) +(N-2k)H_2 \left( \frac{ \left \lfloor \frac{N \chi}{4} \right \rfloor -
\left \lfloor \frac{k}{2} \right \rfloor}{N-2k}  \right)-NH_2 \left(\frac{\left \lfloor \frac{N \chi}{4} \right \rfloor}{N} \right)\\
\; &\le \;
-\frac{2}{\log 2} N \cdot \frac{2k}{N} \left( 1- \frac{2k}{N}\right) \left( \frac{\left \lfloor \frac{k}{2} \right \rfloor}{2k}-  \frac{\left \lfloor \frac{N \chi}{4} \right \rfloor- \left \lfloor \frac{k}{2} \right \rfloor }{N-2k} \right)^2 \;\le\; -c(\chi) k\,,
\end{aligned}
\end{equation}
where we have used \eqref{convex:diff} in the first inequality, and  in the last inequality have used $\chi=17/16$, $k \le N/64$ and $\liminf_{N \to \infty} k_N= \infty$.
With $n=\left \lfloor \frac{N\chi}{4}  \right\rfloor$ the map in \eqref{binomial:map} is increasing for $x \in \lint 0, \left \lfloor \frac{k}{2}  \right\rfloor\rint$.  Then by \eqref{bd:binomial} and \eqref{1stbn:IN}, concerning \eqref{upbd:1stprob}   we have
\begin{equation}\label{1st:partpos}
\begin{aligned}
\mu_{N, 2k} \left( \bar \xi\left( \left\lfloor\frac{k}{2}\right\rfloor\right) > \frac{N\chi}{4} \right) \; &\le \; \frac{k}{2}
{N\choose 2k}^{-1}  { \left \lfloor \frac{N\chi}{4}  \right\rfloor \choose \left\lfloor\frac{k}{2}\right\rfloor } { N-\left \lfloor \frac{N\chi}{4} \right\rfloor  \choose 2k-\left\lfloor\frac{k}{2}\right\rfloor }\\
\;&=\; 
\frac{k}{2} {2k \choose \left\lfloor\frac{k}{2}\right\rfloor }
 {N-2k \choose \left\lfloor\frac{N\chi}{4}\right\rfloor-\left\lfloor\frac{k}{2}\right\rfloor } {N \choose \left \lfloor \frac{N \chi}{4} \right\rfloor }^{-1}\\
 \;&\le\;
 \frac{\frac{k}{2} \sqrt{\left \lfloor \frac{N \chi}{4} \right\rfloor \left(1-\frac{\left \lfloor \frac{N \chi}{4} \right\rfloor}{N} \right)} \cdot 2^{I_{N,1}} }{\sqrt{\left \lfloor \frac{k}{2} \right\rfloor \left(1- \frac{\left \lfloor \frac{k}{2} \right\rfloor}{2k} \right) \cdot \left(
\left \lfloor \frac{N\chi}{4} \right\rfloor - \left \lfloor \frac{k}{2} \right\rfloor\right) \left(1-\frac{\left \lfloor \frac{N\chi}{4} \right\rfloor - \left \lfloor \frac{k}{2} \right\rfloor}{N-2k} \right)
  }}\\
  \; &\le \; C(\chi) \sqrt{k} 2^{I_N} \;\le\; C(\chi) \sqrt{k} 2^{-c(\chi)k}\,.
\end{aligned}
\end{equation} 
By analogous analysis, concerning the other two terms in the r.h.s.\ of \eqref{upbd:smallprobmu2k} we have
\begin{align}
\mu_{N, 2k} \left( \bar \xi\left( \left\lfloor\frac{3k}{4}\right\rfloor\right) > \frac{3N\chi}{8} \right) 
   \;&\le \; C(\chi) \sqrt{k} 2^{-c(\chi)k}\,,\label{2nd:partpos}\\
   \mu_{N, 2k} \left( \bar \xi\left( k\right) > \frac{N\chi}{2} \right) 
    \;&\le \; C(\chi) \sqrt{k}2^{-c(\chi)k}\,.\label{3rd:partpos}
\end{align}
Combing \eqref{1st:partpos}, \eqref{2nd:partpos}, \eqref{3rd:partpos} and \eqref{bd:g4p}, by \eqref{lwbd:initial-heignf} we have 
\begin{equation}\label{lwbd:expsmallk}
\begin{aligned}
\bE \left[ \tf(\eta_0^\nu)\right] \; 
& \ge \; 
-6 C(\chi)k^{3/2} 2^{-c(\chi)k}+
\left(1-3C(\chi)\sqrt{k} 2^{-c(\chi)k} \right) \bigg[  \left\lfloor\frac{k}{2}\right\rfloor \cdot \left( \cos\left(\frac{\chi \pi}{4} \right)-\delta\right)\\
& \quad  + \left(\left\lfloor\frac{3k}{4}\right\rfloor -\left\lfloor\frac{k}{2}\right\rfloor\right)\cdot \left( \cos\left(\frac{3\chi \pi}{8} \right)-\delta\right)+ 
\left(k -\left\lfloor\frac{3k}{4}\right\rfloor\right)\cdot \left( \cos\left(\frac{\chi \pi }{2} \right)-\delta  \right)\bigg]\\
& \ge \;  -6 C(\chi)k^{3/2} 2^{-c(\chi)k}+
\left(1-3C(\chi)\sqrt{k} 2^{-c(\chi)k} \right) \left\lfloor\frac{k}{2}\right\rfloor \cdot \left( \cos\left(\frac{\chi \pi}{4} \right)-\delta\right)\\
& \ge \; c(\chi) k\,,
\end{aligned}
\end{equation}
where we have used $\cos\left(\frac{3\chi \pi}{8} \right)-\delta>- \left(\cos\left(\frac{\chi \pi }{2} \right)-\delta \right)$ in the second last inequality.

Concerning the second term in the r.h.s.\ of  \eqref{nondegvar:premierdemi},  
using $t_0=\frac{1}{2\gl_1} \log k- \frac{C(\gep)}{\gl_1}$, 
$
e^{-2 \gl_1 t_0}\;=\;  \frac{1}{k}e^{2C(\gep)}
$ and  $\max_x \vert g(x) \vert \le 2$ by Theorem  \ref{th:gapshapeder}--\eqref{eigfun:shape}, for all $k$ sufficiently large we have

\begin{equation}\label{upbd:2ndmomeignu}
\bE \left[ F(0, \eta_{0}^{\nu})-
  \bE\left[F(0, \eta_0^{\nu}) \right]\right]^2 
  \;\le\; 
  4e^{-2 \gl_1 t_0}\, \sum_{1 \le x, \, y \le N} \, \vert \left[\nu \left( \xi(x) \xi(y) \right) \right)- \nu\left( \xi(x) \right) \nu \left(\xi(y) \right]\vert \;\le\; k\,,
\end{equation}
where we have used 
Lemma \ref{lema:2ndmomLLN} in the last inequality. 
Regarding  
\eqref{nondegvar:premierdemi}, by  \eqref{upbd:martquad} and \eqref{upbd:2ndmomeignu} we have
\begin{equation}\label{upbd:varsmallkLLN}
 \bE\left[\tf(\eta_{t_0}^{\nu})-\bE\left[\tf(\eta_{t_0}^{\nu})\right]\right]^2 \;\le \; 66 k\,.
\end{equation}
We can apply the argument as that in \eqref{lwbd:tvdistLLNBigk} for the case of $k \in \lint (\log N)^{1+\gga}, N/64\rint$. Taking $t_0=\frac{1}{2\gl_1} \log k- \frac{C(\gep)}{\gl_1}$  in \eqref{mart:meanmaxicon},  by \eqref{lwbd:expsmallk} 
we have
\begin{equation}\label{2ndmom:maxconfsmallk}
\bE \left[ \tf(\eta_{t_0}^{\nu})\right]\;=\;e^{-\gl_1 t_0} \bE \left[  \tf(\eta_0^ \nu) \right]
\;\ge\; e^{-\gl_1 t_0} c(\chi) k  \;\ge \;
c(\chi) e^{C(\gep)} \sqrt{k}\;\ge\;
e^{C(\gep)/2} \sqrt{k}
\end{equation}
where the constant $C(\gep)$ is chosen to be sufficiently large such that the last inequality holds. 
In \eqref{lwbd:mixphil}, we take  $\ell_N(\gep) \colonequals e^{C(\gep)/2} \sqrt{k}/2 $  to  obtain
\begin{equation}\label{lwbd:tvdistLLNsmallk}
\begin{aligned}
d_{N,k}(t_0) 
& \;\ge\; 
\bP \left[  \tf\left(\eta_{t_0}^{\nu}\right)- \bE\left[\tf\left(\eta_{t_0}^{\nu}\right) \right]  \ge  \ell_N(\gep)-\bE\left[\tf(\eta_{t_0}^{\nu}) \right]  \right]-
\bP \left[  \tf\left(\eta_{t_0}^{\mu}\right)- \bE\left[\tf(\eta_{t_0}^{\mu}) \right]  \ge \ell_N(\gep) \right]\\
&\;\ge\; 
\bP \left[ \vert \tf(\eta_{t_0}^{\nu})- \bE\left[\tf(\eta_{t_0}^{\nu}) \right] \vert \le  e^{C(\gep)/2} \sqrt{k}/2 \right]-
\bP \left[ \vert \tf(\eta_{t_0}^{\mu})- \bE\left[\tf(\eta_{t_0}^{\mu}) \right] \vert \ge e^{C(\gep)/2} \sqrt{k}/2 \right]\\
&\;\ge \; 1-\frac{\gep}{2}-\frac{\gep}{2}\,,
\end{aligned}
\end{equation}
where we have used  \eqref{upbd3:2ndmomentMart} and \eqref{upbd:varsmallkLLN} with $C(\gep)$ chosen to be sufficiently large such that the last inequality holds. 
This allows to  conclude the proof for the case $k \in  [ (\log N)^{1+\gga},  N/64]$.

\end{proof}

\section{The upper bound on the mixing time}\label{sec:upbdmixtime}

In this section,   the assumption \eqref{LLN} and Assumption \ref{assum:1}, \ref{assum:2} and \ref{assum:nparticles} are always in force without explicit statement. 
Our goal is to prove the following proposition.

\begin{proposition} \label{prop:mixupLLN}
Assuming \eqref{LLN}, Assumption \ref{assum:1}, Assumption \ref{assum:2}  and  Assumption \ref{assum:nparticles}, 
for all $\gep \in (0,1)$, all $\delta>0$ and  for all $N$ sufficiently large we have
\begin{equation}\label{upbd:mixexclu}
t_{\Mix}^{N,k}(\gep) \;\le\; \frac{(1+2\delta)^3}{2 \cdot \Gap_N} \log k_N\,.
\end{equation}
\end{proposition}

Theorem \ref{th:mixexclusion} is a consequence of the combination of Theorem \ref{th:gapshapeder}, Proposition \ref{prop:mixlow} and Proposition \ref{prop:mixupLLN}.

\subsection{Preliminaries for the upper bounds on mixing time}

Recalling \eqref{dist-coupling}, in practice it is more feasible to couple two dynamics with one of them at equilibrium. We define
\begin{equation}\label{def:couptimes}
\begin{gathered}
T_1 \colonequals \inf \left\{t \ge 0: \; h^{\wedge}_t=h_t^{\mu} \right\}\,,\\
T_2 \colonequals \inf \left\{t \ge 0: \; h^{\vee}_t=h_t^{\mu} \right\}\,,
\end{gathered}
\end{equation}
where $(h_t^{\mu})_{t \ge 0}$ is the dynamics with initial configuration sampled from $\mu_{N,k}$ and then following the graphic construction for $t>0$. Moreover, we know that
\begin{equation*}
T\;=\;\max(T_1,T_2)\,.
\end{equation*}
We adapt the approach in \cite[Section 7]{labbe2018mixing} to deal with $T_1, T_2$. As 
\begin{equation*}
\bP[T>t] \;\le \;\bP[T_1>t]+\bP[T_2>t]\,,
\end{equation*}
 by symmetry it is sufficient to provide an upper bound on $\bP[T_1>t]$.

\subsection{Construction of a supermartingale} Inspired by \cite[Theorem 2]{wilson2004mixing}, 
we start by embedding the segment $\lint 1, N \rint$ in a longer  segment $\lint - \left \lfloor \delta N  \right\rfloor, N+ \left \lfloor \delta N  \right\rfloor  \rint$, and place conductance $\left( c(x, x+1)=1 \right)_{x \le 0}$ and $\left( c(x,x+1)=1 \right)_{x \ge N
}$. Set $\bar N \colonequals N+ 2 \left \lfloor \delta N  \right\rfloor+1 $. By Theorem \ref{th:gapshapeder}--\eqref{gap:asym} the spectral gap, denoted by $\bar \gl_1$, of the simple exclusion process in $\lint - \left \lfloor \delta N  \right\rfloor, N+ \left \lfloor \delta N  \right\rfloor  \rint$  satisfies
\begin{equation}\label{estimate:bargl1}
\lim_{N \to \infty} \frac{ \bar \gl_1 \bar N^2}{\pi^2}\;=\; 1\,,
\end{equation}
and by Theorem \ref{th:gapshapeder}--\eqref{eigfun:shape} the eigenfunction $\left(G(x)\right)_{x \in \lint - \left \lfloor \delta N  \right\rfloor, N+ \left \lfloor \delta N  \right\rfloor  \rint}$ with $G(-\lfloor \delta N \rfloor)=1$ corresponding to the  spectral gap $\bar \gl_1$  satisfies

\begin{align}\label{shap:Gshift}
		\lim_{N\to \infty}\; \sup_{x \in \lint - \left \lfloor \delta N  \right\rfloor,\, N+ \left \lfloor \delta N  \right\rfloor  \rint}\; &\left| G(x) - \cos\left(\frac{\pi(x  +\left \lfloor \delta N  \right\rfloor+ 1/2)}{\bar N}\right) \right|\;=\;0\,.
\end{align}
The reason why in \eqref{shap:Gshift} it is $+1/2$ rather than $-1/2$ is due to a shift of the coordinates. 
As in \eqref{gen-comp2}, we have

\begin{equation*}\label{autofun:G}
	(\Delta^{(c)}G)(x)\;=\;(c   \nabla G)(x+1)-(c \nabla G)(x)\;=\; -\bar \gl_1 G(x)\,, \quad \forall\, x \in \lint \left\lfloor-\delta N \right\rfloor,  N+\left\lfloor-\delta N \right\rfloor \rint\,.
	\end{equation*}
Moreover, by Theorem~\ref{th:gapshapeder},  $G$ is strictly monotone decreasing,  and set   $x \in \lint 1, N \rint$
\begin{equation}\label{difer:barG}
\bar G(x) \; \colonequals\; G(x)-G(x+1)\;>\;0\,.
\end{equation}
For $\xi \in \gO_{N,k}$, recalling \eqref{heighfun} 
  we define
\begin{equation}\label{def:funbF}
\bF(\xi) \colonequals \sum_{x=1}^{N-1} h^\xi(x) \bar G(x)\,.
\end{equation}
For $\xi,\, \xi'\in \gO_{N,k}$ with $\xi \le \xi'$, since $h^\xi(x)\le h^{\xi'}(x)$ and $\bar G(x)>0$,  we have 
\begin{equation*}
\bF(\xi) \;\le\; \bF(\xi')\,.
\end{equation*}
  Furthermore, if  $\xi \le \xi'$ with $\xi \neq \xi'$,  we have
 $\bF(\xi) \;<\; \bF(\xi')\,.$
Using  $h^\xi(0)= h^\xi(N)=0$ and for $x \in \lint 1, N-1\rint$  
 \begin{equation}\label{gen:height}
 \begin{aligned}
 \left(\cL_{N, k} h^\xi \right)(x)\;&=\;
 c(x, x+1)\left[\xi(x+1)-\xi(x) \right]\\
&=\;
 c(x, x+1) \left[ \left(h^\xi(x+1)-h^{\xi}(x)\right)- \left(h^\xi(x)-h^{\xi}(x-1)\right)\right]\,,
 \end{aligned}
 \end{equation}
 we obtain 
\begin{equation}
\begin{aligned}
\left( \cL_{N,k} \bF \right)(\xi) \;&=\;
\sum_{x=1}^{N-1} \bar G(x) \left(\cL_{N, k} h^\xi \right)(x)\\
&=\;
\sum_{x=1}^{N-1} h^\xi(x) \left[ c(x+1, x+2)\bar G(x+1)-c(x, x+1)\bar G(x) \right]\\
& \quad 
- \sum_{x=1}^{N-1} h^\xi(x) \left[ c(x, x+1)\bar G(x)-c(x-1, x)\bar G(x-1) \right]\\
& \quad - h^\xi(1) c(0, 1)\bar G(0)- h^\xi(N-1) c(N, N+1)\bar G(N)\\
&=\;
\sum_{x=1}^{N-1} h^\xi(x) \bar \gl_1\left[  G(x+1)-G(x) \right]\\
& \quad - h^\xi(1) c(0, 1)\bar G(0)- h^\xi(N-1) c(N, N+1)\bar G(N)\\
&=\;
-\bar \gl_1 \bF(\xi) - h^\xi(1) c(0, 1)\bar G(0)- h^\xi(N-1) c(N, N+1)\bar G(N)
\end{aligned}
\end{equation}
where we have used in the second last equality
\begin{equation*}
\begin{aligned}
 &c(x, x+1)\bar G(x)-c(x-1, x)\bar G(x-1)\\
 \;=\;&
 -\left\{c(x, x+1) \left[ G(x+1)-G(x) \right] + c(x-1, x) \left[ G(x-1)-G(x) \right] \right\}\\
 \;=\;& \bar \gl_1 G(x)\,.
 \end{aligned}
\end{equation*}
Note that for $\xi, \xi' \in \gO_{N,k}$ with $\xi \le \xi'$ in partial order sense \eqref{partorder:gONk}, we have
\begin{equation*}
h^\xi(1) \;\le\; h^{\xi'}(1), \quad \quad h^\xi(N-1) \;\le\; h^{ \xi'}(N-1)
\end{equation*}
 and  thus
\begin{equation}\label{gen:twostates}
\begin{aligned}
&\left( \cL_{N, k} \bF\right)(\xi')-\left( \cL_{N, k} \bF\right)(\xi)\\
\;&=\;
-\bar \gl_1 \left[\bF(\xi')- \bF(\xi) \right] \\
& \quad-\left[h^{\xi'}(1)- h^\xi(1) \right] c(0, 1)\bar G(0)- \left[h^{\xi'}(N-1)- h^\xi(N-1) \right] c(N, N+1)\bar G(N)\\
& \le \;
-\bar \gl_1 \left[\bF(\xi')- \bF(\xi) \right]\,.
\end{aligned}
\end{equation}
To record the effect of the  minimal change of  jumps on the function $\bF$, we define
\begin{equation}\label{def:bardelta}
\bar \delta_{\min}\, \colonequals\, \min_{2 \le x \le N+1} \, \vert (c\nabla G)(x)\vert\,, \quad \quad \bar \delta_{\max} \,\colonequals\, \max_{2 \le x \le N+1} \, \vert (c\nabla G)(x)\vert\,.
\end{equation}
Recalling \eqref{def:funbF},
we define
\begin{equation}\label{def:Atfun}
A_t \;\colonequals \; \bar \delta_{\min}^{-1} \left[ \bF(\eta_t^{\wedge})-\bF(\eta_t^{\mu})\right]\,.
\end{equation}
Note that $A_t=0$ if and only if $t \ge T_1$ where $T_1$ is   defined in \eqref{def:couptimes}.
By \cite[Lemma 5.1 in Appendix 1]{LandimHydrodynamicsbk}, defining
\begin{equation*}
M_t \colonequals A_t-A_0-\int_0^t \cL_{N,k} A_s \dd s
\end{equation*}
we know that  $\left(M_t, (\cF_t)_{t \ge 0}, \bP \right)$
is a Dynkin martingale.  By \eqref{gen:twostates} we have
\begin{equation}\label{genact:A}
\begin{aligned}
\cL_{N,k} A_t\;&=\; \bar \delta_{\min}^{-1} \left[ (\cL_{N,k} \bF)(\eta_t^\wedge)- (\cL_{N,k} \bF)(\eta_t^\mu) \right]\\
 \;&\le\;  -\bar \delta_{\min}^{-1} \bar \gl_1 \left[\bF(\eta^\wedge_t)-\bF(\eta^\mu_t) \right]\\
 \;&=\;- \bar \gl_1 A_t \;\le\; 0\,.
\end{aligned}
\end{equation}
 Therefore, $ \left( (A_t)_{t \ge 0},\, (\cF_t)_{t \ge 0},\,\bP \right)$ is a supermartingal. With an abuse of notation, let $\langle A. \rangle$ denote the predictable bracket associated with the martingale $\left(M_t, (\cF_t)_{t \ge 0}, \bP \right)$. Moreover, we have the following estimate on \eqref{def:bardelta}.

\begin{lemma}\label{lema:eigfunincrement}
Under the assumption \eqref{LLN}, $\left( c(x,x+1)=1 \right)_{x \le 0 
}$ and  $\left( c(x,x+1)=1 \right)_{ x \ge N
}$, for all $\delta>0$,
  if  $N$ is sufficiently large,   there exists a constant $C(\delta)>0$ such that
\begin{align}
 \frac{1}{C(\delta)N}\;&\le\; \bar \delta_{\min} \;\le\; \frac{C(\delta)}{N}\,,
 \label{eigcond:minincre}\\
 \frac{1}{C(\delta)N}\;&\le\; \bar \delta_{\max} \;\le\; \frac{C(\delta)}{N}\,. \label{eigcond:maxincre}
\end{align}

\end{lemma}

\begin{proof}
 By 
 \begin{equation*}
 c(x, x+1)\left[G(x+1)-G(x) \right]\;=\; c(x-1, x)\left[G(x)-G(x-1) \right]-\overline \gl_1 G(x)
 \end{equation*}

 and the fact that $G$ is strictly monotone decreasing due to Theorem \ref{th:gapshapeder}, we know that 

\begin{equation}\label{eigcond:minincrem2}
\begin{aligned}
\bar \delta_{\min}\;&=\; \bar \gl_1 \min_{1 \le x < N} \, \sum_{y=- \left\lfloor \delta N \right\rfloor}^x G(y)\\
\;&=\; \bar \gl_1 \min \left( \sum_{y=- \left\lfloor \delta N \right\rfloor}^1 G(y),\,\, \sum_{y=- \left\lfloor \delta N \right\rfloor}^{N-1} G(y) \right)\\
&=\;
\bar \gl_1 \min \left( \sum_{y=- \left\lfloor \delta N \right\rfloor}^1 G(y),\,\, -\sum_{y=N}^{\bar N} G(y) \right)
\end{aligned}
\end{equation}
where  we have used $\sum_{x=- \left\lfloor \delta N \right\rfloor}^{\bar N} G(x)=0$  in the last equality.
Then combing \eqref{eigcond:minincrem2} with \eqref{shap:Gshift} and \eqref{estimate:bargl1}, we obtain \eqref{eigcond:minincre}. 

Since 
\begin{equation*}
\bar \delta_{\min} \;\le \;  \bar \delta_{\max} \;= \;
\bar \gl_1 \max_{1 \le x < N} \, \sum_{y=- \left\lfloor \delta N \right\rfloor}^x G(y)\,,
\end{equation*}
we obtain \eqref{eigcond:maxincre} by   \eqref{estimate:bargl1}, \eqref{shap:Gshift} and \eqref{eigcond:minincre}.
\end{proof}

 \subsection{Set the plan for the proof of Proposition \ref{prop:mixupLLN}}
For $\delta>0$,  set 
\begin{equation}\label{def:tdelta}
t_{\delta}\;\colonequals\;\frac{1+\delta}{2 \bar \gl_1}\log k\,.
\end{equation} 
Recalling \eqref{def:couptimes}, our goal is to show that for all $\gep>0$ and $\delta>0$, for all $N$ sufficiently large we
 \begin{equation}\label{upbd:coupleT1}
 \bP\left[ T_1> t_{\delta}\right]\;\le\; \gep/2\,,
\end{equation}  
 which allows to obtain Proposition \ref{prop:mixupLLN}.
 Taking $\psi \in (0,  \delta/8)$ 
 and  $K\colonequals\lceil 1/(2\psi) \rceil$, 
 we define a sequence of successive stopping times $(\Tau_i)_{i=2}^{K}$ by
\begin{equation*}
\Tau_2 \;\colonequals\; \inf \left\{ t\geq t_{\delta/2} : \,  A_t \;\le\; \bar \delta_{\min}^{-1}  k^{\frac{1}{2}-2\psi}  \right\}\,,
\end{equation*}
and for  $3\leq i \leq K$,
\begin{equation*}
 \Tau_i \;\colonequals\; \inf \left\{ t\geq \Tau_{i-1}:\, A_t\;\le\, 
\bar \delta_{\min}^{-1}  k^{\frac{1}{2}-i\psi}  \right\} \,,
\end{equation*}
where $k$ is the number of particles. 
 For consistency of notations, we set $\Tau_{\infty}\colonequals\max \big(T_1, t_{\delta/2}\big)$.
 In the remainder, we always assume $T_1 \ge t_{\delta/2}$, otherwise \eqref{upbd:coupleT1} holds directly. The task in the remaining of this section is the following proposition.

\begin{proposition}\label{prop:consecutstoptime} 
 Given $\delta>0$, if $\psi \in (0, \delta/8)$ and  $K=\lceil 1/(2\psi) \rceil>1/(2\psi)$, we have
\begin{equation*}
\lim_{N \to \infty} \bP \left[ \{\Tau_2=t_{\delta/2}\} \cap  \left(  \bigcap_{i=3}^{K} \left\{\Delta \Tau_i \leq 2^{-i} (t_\delta-t_{\delta/2}) \right\}\right)  \cap \left\{\Tau_{\infty}-\Tau_K \le \bar \Upsilon_N^{-1}  \bar\delta_{\min}^{-2}\right\} \right]\;=\;1\,, 
\end{equation*}
where $\Delta \Tau_i\colonequals\Tau_i -\Tau_{i-1}$ for $3 \leq i \leq K$ and $\bar \Upsilon_N$ is defined in Assumption \ref{assum:2}.
\end{proposition}
\noindent
If Proposition \ref{prop:consecutstoptime} holds,  for  all $N$ sufficiently large we have
\begin{equation*}
T_1\;=\;\Tau_{\infty}\;\le\; t_{\delta/2}+\sum_{i=3}^{K}2^{-i}(t_{\delta}-t_{\delta/2})+\bar \Upsilon_N^{-1}  \bar\delta_{\min}^{-2}
\;\le\;
 t_\delta\,,
\end{equation*}
where we have used Lemma \ref{lema:eigfunincrement},    \eqref{estimate:bargl1} and  Assumptions \ref{assum:2} and \ref{assum:nparticles}    in the last  inequality. 
Then  we conclude the proof  of  Proposition \ref{prop:mixupLLN}.
 
 \subsection{The proof for $\cT_2=t_{\delta/2}$.}
The goal of this subsection if the following lemma.

\begin{lemma}\label{lema:1ststoptime} 
For all  $\epsilon>0$, all sufficiently small $\delta>0$ and $0<\psi<\delta/8$,  if  $N$ is sufficiently large, we have  
\begin{equation*}
\bP \left[\Tau_2>t_{\delta/2}\right]\;\le\; \epsilon/2\,.
\end{equation*}
\end{lemma}

For the proof of the lemma above, we first provide an upper bound on $\bE\left[A_0 \right]$.
Recalling \eqref{heighfun}, we have
 \begin{equation}\label{height:mean}
 \bE \left[ h_t^\mu(x) \right]\;=\; 0\,, \quad \forall \, t \ge 0,\, \forall\, x \in \lint 1, N \rint\,,
 \end{equation}
 and
 \begin{equation*}
 h_0^\wedge (x)\;=\;
 \begin{cases}
 \left(1- \frac{k}{N} \right) x & \text{ if } x \le k\,,\\
 k-\frac{k}{N}x & \text{ if } x > k\,,
 \end{cases}
 \end{equation*}
 and then
 \begin{equation}\label{upbd:inicioF}
 \begin{aligned}
 \bF(\eta_0^\wedge)\;&=\;\sum_{x=1}^k  \left(1- \frac{k}{N} \right) x \left[ G(x)-G(x+1) \right]
 +
 \sum_{x=k+1}^N  \left(k-\frac{k}{N}x \right) \left[ G(x)-G(x+1) \right]\,.
 \end{aligned}
 \end{equation}
 Moreover, using $-2 \le G(x) \le 1$ by \eqref{shap:Gshift} and the fact that $G$ is strictly monotone decreasing, 
  we have
 \begin{equation*}
 \begin{aligned}
 &\sum_{x=1}^k  \left(1- \frac{k}{N} \right) x \left[ G(x)-G(x+1) \right]\\
  \;=&\;
 \left(1- \frac{k}{N} \right) \left[ G(1)+ \sum_{x=2}^k  \left[xG(x)-(x-1)G(x) \right]-kG(k+1) \right]\\
  \;=&\;
  \left(1- \frac{k}{N} \right)
  \sum_{x=1}^k  \left(G(x)-G(k+1) \right) \\
   \;\le& \; 3 k\,,
 \end{aligned}
 \end{equation*}
and  
 \begin{equation*}
 \begin{aligned}
 &\sum_{x=k+1}^N  \left(k-\frac{k}{N}x \right) \left[ G(x)-G(x+1) \right]\\
 =\;&
 k \left[G(k+1)-G(N+1) \right]-\frac{k}{N} \left[ (k+1)G(k+1)+ \sum_{x=k+2}^N  G(x)-NG(N+1) \right]\\
 \le\;&
 3k-\frac{k}{N} \left[(k+1)\left(G(k+1)-G(N+1) \right)+\sum_{x=k+2}^N  \left( G(x)-G(N)\right) \right]\\
 \le\;& 
 3k\,.
 \end{aligned}
 \end{equation*}
Therefore, we have for all $N$ sufficiently large
\begin{equation}\label{upbd:areat0}
\bF(\eta_0^\wedge)\;\le\; 6k\,.
\end{equation}

\begin{proof}[Proof of Lemma \ref{lema:1ststoptime}]
By \eqref{genact:A} we have
\begin{equation*}\label{der:area} 
\frac{d}{\dd t}\bE\left[  A_t\right]\;=\;\bE \left[ \cL_{N,k}A_t \right]\;\le\;-\bar\gl_1 \bE\left[ A_t \right]\,,
\end{equation*}
and thus
\begin{equation}\label{ineq:derarea}
\bE \left[ A_t\right]\;\le\; e^{-\bar\gl_1 t} \bE \left[ A_0\right]\;\le\;6 k \cdot e^{-\bar\gl_1 t} \bar \delta_{\min}^{-1} 
\end{equation}
where the last inequality is by  \eqref{height:mean} and \eqref{upbd:areat0}.
By  Markov's inequality and \eqref{ineq:derarea}, we have
\begin{equation}
\bP \left[\cT_2>t_{\delta/2} \right]
\;=\;
\bP \left[A_{t_{\delta/2}} > \bar \delta_{\min}^{-1} k^{\frac{1}{2}-2\psi} \right]
\;\le\; \bar \delta_{\min}
 k^{-\frac{1}{2}+2\psi} \bE \left[ A_{t_{\delta/2}} \right]
\;\le\; 6k^{-(\delta/4-2\psi)}\,.
\end{equation}
We conclude the proof by  $\liminf_{N \to \infty}k_N=\infty$. 
 \end{proof}

 \subsection{The estimation of $\langle  A.\rangle_{\cT_i}-\langle  A.\rangle_{\cT_{i-1}}$}
 Our goal in  this subsection is to prove that for all $i\in \lint 3, K \rint$, $\langle A. \rangle_{\Tau_i}-\langle A. \rangle_{\Tau_{i-1}}$ is small. 
 Write
\begin{equation}
\begin{aligned}
\Delta_i \langle A \rangle &\;\colonequals\;\langle A. \rangle_{\Tau_i}- \langle A. \rangle_{\Tau_{i-1}}, \; \forall \,i\, \in \lint 3, K \rint\,,\\
\Delta_{\infty} \langle A \rangle &\; \colonequals\; \langle A. \rangle_{\Tau_{\infty}}- \langle A. \rangle_{\Tau_K}\,.
\end{aligned}
\end{equation}

\begin{proposition}[{ \cite[Proposition 29]{labbe2018mixing}}] \label{prop:surpmart}
Let $(\mathbf{M}_t)_{t\geq 0}$ be a pure-jump supermartingale with bounded  jump rates and jump amplitudes, and $\mathbf{M}_0\leq a$ almost surely where $a \ge 1$. Let $\langle \mathbf{M}. \rangle$, with an abuse of notation, denote the predictable bracket associated with the martingale $\overline{\mathbf{M}}_t=\mathbf{M}_t+I_t$ where $I$ is the compensator of $\mathbf{M}$. 
Given $b \in \mathbb{R}$ and $b\leq a$, we set $$ \tau_b \; \colonequals\;\inf\{t \geq 0 :\, \mathbf{M}_t\leq b\}\,.$$

\noindent
 If the
amplitudes of the jumps of $(\mathbf{M}_t)_{t\geq 0}$ are bounded above by $a-b$, for
any $u\geq 0$, we have
\begin{equation}
\bP\left[\langle \mathbf{M}.\rangle_{\tau_b}\geq (a-b)^2u \right]\;\le\; 8u^{-1/2}\,.
\end{equation}

\end{proposition}

 Now we apply Proposition \ref{prop:surpmart} to prove that the probability of the event 
\begin{equation*}
\mathcal{A}_N\;\colonequals\;\left\{ \forall\, i\in \lint 3, K \rint, \; \Delta_i \langle A \rangle \;\le\; \bar \delta_{\min}^{-2} k^{1-2(i-1)\psi+\frac{1}{2}\psi}\right\} \bigcap \left\{ \Delta_{\infty} \langle A \rangle\;\le\; \bar  \delta_{\min}^{-2} \right\}
\end{equation*}
 is almost one, which is the following lemma.
 
\begin{lemma} \label{lema:ANfullmass} 
For $K=\lceil 1/(2\psi) \rceil>1/(2\psi)$, we have
\begin{equation}
\lim_{N \to \infty} \bP \left[  \cA_N \right]\;=\;1\,.
\end{equation}
\end{lemma}
\begin{proof}
We prove that $\mathcal{A}_N^{\complement}$ has almost zero mass. Note that  amplitudes of the jumps of $(A_t)_{t \ge 0}$ are bounded above by

\begin{equation}\label{bd:jumpamplitude}
\begin{aligned}
\bar \delta_{\min}^{-1} \max_{1\le x \le N} \left[ G(x)-G(x+1)\right]
\;&=\; \bar \delta_{\min}^{-1} \max_{1\le x \le N} r(x,x+1) c(x,x+1)  \left[ G(x)-G(x+1)\right]\\
\;&\le \;
\bar \delta_{\min}^{-1} \max_{1\le x <N} r(x,x+1) \bar \delta_{\max} \;  =\; o(N)\,,
\end{aligned}
\end{equation}
where we have used Lemma \ref{lema:eigfunincrement} and the assumption \eqref{LLN} in the last equality.

For $3 \le i \le K $, 
we apply Proposition \ref{prop:surpmart} to $(A_{t+\Tau_{i-1}})_{t\geq 0}$ with $a_i \colonequals \bar \delta_{\min}^{-1} k^{\frac{1}{2}-(i-1)\psi}$, $b_i \colonequals \bar \delta_{\min}^{-1} k^{\frac{1}{2}-i\psi}$ and $u_i \colonequals k^{2\psi+\frac{1}{2} \psi}(k^{\psi}-1)^{-2}$, where $a_i,b_i,u_i$ are chosen to satisfy
\begin{equation*}
 \left(\bar \delta_{\min}^{-1} k^{\frac{1}{2}-(i-1)\psi}-\bar \delta_{\min}^{-1}k^{\frac{1}{2}-i\psi} \right)^2u_i \;=\;\bar \delta_{\min}^{-2}k^{1-2(i-1)\psi+\frac{1}{2}\psi}\,.
\end{equation*}
Note that 
\begin{equation*}
a_i-b_i\; = \;
\bar \delta_{\min}^{-1} k^{\frac{1}{2}-(i-1)\psi} -\bar  \delta_{\min}^{-1} k^{\frac{1}{2}-i\psi} 
\;=\;
\bar \delta_{\min}^{-1} k^{\frac{1}{2}-(i-1)\psi} \left(1-k^{-\psi} \right)
\;\ge\;
\bar \delta_{\min}^{-1} \;\ge\; C(\delta)^{-1}N \,,
\end{equation*}
where  we have used Lemma  \ref{lema:eigfunincrement} in the last inequality. 
Thus, the assumption in 
Proposition \ref{prop:surpmart} is satisfied.
Then  we obtain for every $ i \in \lint 3, K\rint$,  
\begin{equation}\label{applicat-surmarting}
\bP \left[\Delta_i \langle A \rangle \;\ge\; \left(\bar \delta_{\min}^{-1} k^{\frac{1}{2}-(i-1)\psi}-\bar \delta_{\min}^{-1} k^{\frac{1}{2}-i\psi} \right)^2 u_i \right]
\;\le\; 8 u_i^{-\frac{1}{2}}, 
\end{equation}
which vanishes as $k$ tends to infinity.

Taking $a_{\infty} \colonequals \bar \delta_{\min}^{-1} k^{\frac{1}{2}-K\psi}$,  $b_{\infty} \colonequals 0$ and $u_{\infty} \colonequals k^{2K\psi-1}$, 
the assumption in Proposition \ref{prop:surpmart} is satisfied by \eqref{bd:jumpamplitude} and 
\begin{multline*}
a_{\infty}-b_{\infty}\;=\;
\bar \delta_{\min}^{-1} k^{\frac{1}{2}-K\psi}
\;=\; \bar
\delta_{\min}^{-1} k^{\frac{1}{2}-\left\lceil \frac{1}{2 \psi} \right\rceil  \psi}\\
\;\ge\;
\bar \delta_{\min}^{-1}k^{-\psi}\; \ge \; \bar \delta_{\min}^{-1} N^{-\psi} 
\; \ge\; \bar \delta_{\min}^{-1}  \bar \delta_{\max}  \max_{1\le x<N} r(x,x+1)
\end{multline*}
where  we have used
 Assumption \ref{assum:2},   $\psi< 1$ and  Lemma \ref{lema:eigfunincrement}  in the last inequality.
Applying Proposition \ref{prop:surpmart} to $(A_{t+\Tau_K})_{t\geq 0}$ 
we obtain
\begin{equation}\label{Ainfty}
\bP \left[\Delta_{\infty} \langle A \rangle \ge 
\bar \delta_{\min}^{-2}
\right] \le 8 u_{\infty}^{-\frac{1}{2}}\,,
\end{equation}
where $a_{\infty}, b_{\infty}, u_{\infty}$ are chosen to satisfy 
\begin{equation*}
(a_{\infty}-b_{\infty})^2u_{\infty}
\;=\;
 \bar \delta_{\min}^{-2} k^{1-2K\psi}
 k^{2K\psi-1}
\; = \;\bar \delta_{\min}^{-2}\,. 
\end{equation*}
Moreover, the r.h.s.\ in \eqref{Ainfty}
 tends to zero as $k$ tends to infinity due to $K=\lceil 1/(2\psi) \rceil>1/(2\psi)$. 
Since $K$ is a constant, we obtain 
\begin{equation*}
\lim_{N \to \infty} \bP \left[ \mathcal{A}_N^{\complement} \right]\;=\;0\,.
\end{equation*}
\end{proof}

\subsection{Comparing  $\Tau_i-\Tau_{i-1}$ with  $\Delta_i \langle A \rangle$.}

\begin{figure}[h]
\centering
  \begin{tikzpicture}[scale=.4,font=\tiny]
   \draw (25,4) -- (25,-1) -- (52,-1);
   \draw[color=black, line width=0.45mm](25,-1)--(26,0); \draw[color=black, line width=0.45mm](29,1)--(30,0)--(31,-1)--(33,1);
    \draw[color=black, line width=0.45mm](41,1)--(42,0);
   \draw[color=black, line width=0.45mm](49,1) -- (50,0)--(51, -1);
    \draw[color=blue] (26,0) -- (27,-1) -- (28,0) -- (29,1);

      \draw[color=blue]  (33,1)-- (34,0) -- (35,-1) -- (36,0) -- (37,-1) -- (38,-2) -- (39,-1) -- (40,0) -- (41,1);

      \draw[color=blue] (42,0) -- (43,-1) -- (44,0) -- (45,-1) -- (46,0)--(47,-1)--(48,0)--(49,1);
   \draw[color=red](26,0)--(27,1)--(28,2)--(29,1);     
   \draw[color=red](33,1)--(37,5)--(41,1);
    \draw[color=red](42,0)--(44,2)--(46,0)-- (47,1) -- (48,2) -- (49,1); 
     \node[below] at (26.2,-1.3) {$a_1$};
      \node[below] at (29.2,-1.3) {$b_1$};
     \node[below] at (33.2,-1.3) {$a_2$};
     \node[below] at (49.2,-1.3) {$b_2$};
    \foreach \x in {25,...,51} {\draw (\x,-1.3) -- (\x,-1);}
    \draw[fill] (25,-1) circle [radius=0.1];
    \draw[fill] (51,-1) circle [radius=0.1];   
    \node[below] at (25,-1.3) {$0$};
    \node[below] at (51,-1.3) {$N$};
    \node[red,above] at (39, 4){$h_t^{\wedge}$};
    \node[blue,above] at (38, -1.2){$h_t^{\mu}$};
    \draw[thick,->] (25,-1) -- (25,4) node[anchor=north west]{y};
    \draw[thick,->] (25,-1) -- (52,-1) node[anchor=north west]{x};
  \end{tikzpicture}
  \caption{\label{fig:lowerboundforderivativeofA} In this figure, $h_t^{\wedge}$ consists of the red line segments and black thick line segments, while $h_t^{\mu}$ consists of the blue line segments and black thick line segments. Moreover, $\mathcal{B}_t= \lint  a_1, b_1\rint \cup \lint a_2, b_2\rint$, $\#(\mathcal{D}_t \cap \lint a_1, b_1\rint)=3$, and $\#(\mathcal{D}_t \cap \lint a_2, b_2\rint)=13$. In $\lint a_2, b_2\rint$, the monotone segments of $h_t^{\mu}$ are $\lint a_2, a_2+1\rint$, $\lint a_2+1, a_2+3\rint$, $\lint a_2+3, a_2+5\rint$, and so on as shown in the figure.}
\end{figure}
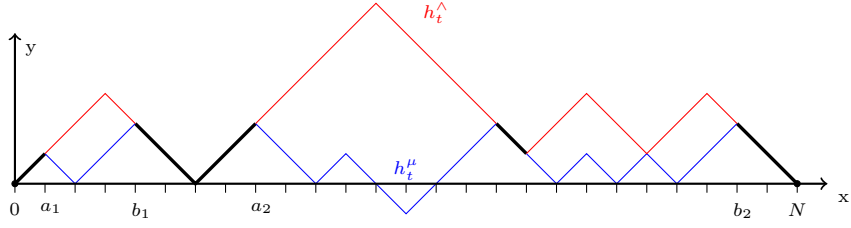

 For the generator  $\bbL_{N, k}$ corresponding to the graphical construction in Subsection \ref{subsec:graphconstr},  for $\xi, \xi' \in \gO_{N,k}$ with  $\xi \le \xi'$ and $$\bbF(\xi, \xi') \colonequals \bar \delta_{\min}^{-1} \left[ \bF(\xi')-\bF(\xi)\right]$$ we have  
 \begin{equation}\label{squarefeild}
 \begin{aligned}
&\left( \bbL_{N,k} {\bbF}^2 \right)\left(\xi, \xi'\right)- 2 \left( \bbF \bbL_{N, k} \bbF \right)\left(\xi, \xi'\right)\\
\;=\;&
\sum_{x=1}^{N-1} c(x,x+1) \left[  \bbF\left(\xi\circ \tau_{x,x+1}, \xi' \circ \tau_{x,x+1}\right)- \bbF\left(\xi,\xi'\right) \right]^2\\
\;=\;&
\sum_{x=1}^{N-1} r(x,x+1) \left[ 
c(x,x+1) \left[  \bbF\left(\xi \circ \tau_{x,x+1}, \xi' \circ \tau_{x,x+1}\right)- \bbF\left(\xi,\xi'\right) \right] \right]^2\\
\ge\; &
\sum_{x= 1}^{N-1} r(x,x+1) \ind_{\{ \bbF\left(\xi \circ \tau_{x,x+1}, \xi' \circ \tau_{x,x+1}\right) \neq  \bbF\left(\xi,\xi'\right) \}}\,,
\end{aligned}
 \end{equation}
 where we have used the definition of $\bar \delta_{\min}$ in  \eqref{def:bardelta} in the last inequality.
 By Assumption \eqref{assum:2} and \eqref{squarefeild}, we have for all $t< \Tau_{\infty}$,
\begin{equation*}
\partial_t \langle A. \rangle_t \;\ge\; \bar \Upsilon_N\,,
\end{equation*}
and thus
\begin{equation*}
\Delta_{\infty} \langle A \rangle\;=\;\int_{\Tau_K}^{\Tau_{\infty}} \partial_t \langle A. \rangle \dd t \;\ge\; \int_{\Tau_K}^{\Tau_{\infty}} \bar \Upsilon_N \dd t \;=\; \bar \Upsilon_N \left(  \Tau_{\infty}-\Tau_K \right)\,.
\end{equation*}
Thus, if the event $\mathcal{A}_N$ holds,  we have 
\begin{equation*}
\Tau_{\infty}-\Tau_K\;\le\; \bar \Upsilon_N^{-1} \Delta_{\infty}\langle  A \rangle\; \le\;  \bar \Upsilon_N^{-1}  \bar\delta_{\min}^{-2} \;=o(t_{\delta}-t_{\delta/2})\,,
\end{equation*}
where we have used Assumption \ref{assum:2} and Assumption \ref{assum:nparticles} in the last equality.

Now we estimate the  increment $\Tau_i-\Tau_{i-1}$ for $ 3 \leq i\leq K$ by comparing with 
\begin{equation*}
\langle A. \rangle_{\Tau_i}-\langle A. \rangle_{\Tau_{i-1}}\;=\;\Delta_i \langle A. \rangle\,.
\end{equation*}
We first provide a lower bound on $\partial_t \langle A. \rangle$, based on  (a) 
the maximal contribution among all the coordinates $x\in \lint 1, N\rint$ in the definition of $A_t$; 
  (b) the amount of admissible jumps in $h_t^{\mu}$ or $h_t^{\wedge}$ that can change the value of $A_t$. Concerning (a),  
set
\begin{equation}\label{def:H}
H(t)\colonequals\max_{x \in \lint 1,\, N \rint}  \bar G(x)\left[h_t^{\wedge}(x)- h_t^{\mu}(x)\right]\,.
\end{equation}
Concerning (b), 
we define
\begin{equation*}
\mathcal{B}_t\;\colonequals\;\left\{  x\in \lint 1, N-1 \rint :\, \exists y \in \lint x-1, x+1 \rint, \, h_t^{\wedge}(y) \neq  h_t^{\mu}(y) \right\}\,,
\end{equation*}
 and 
 \begin{equation*}
 \mathcal{D}_t \;\colonequals\; \left\{x\in \mathcal{B}_t:\, 
x \text{ is a local extremum of } h_t^\mu 
  \right\}\,.
 \end{equation*} 
 By \eqref{squarefeild} and the fact that each point in $\cD_t$ corresponds to an admissible jump,  we have
\begin{equation*}
\partial_t \langle A.\rangle_t \;\ge\; \bar \Upsilon_N \# \mathcal{D}_t\,,
\end{equation*}
where we have used Assumption \ref{assum:2}. 
Let $\lint a, b \rint$ denote the horizontal coordinates of a maximal connected component of $\mathcal{B}_t$, for which we refer to Figure \ref{fig:lowerboundforderivativeofA} for illustration. Since $h_t^{\mu}$ can  not be monotone in the entire domain $\lint a, b \rint$, we have
\begin{equation}\label{upbd:trAbracket}
\# (\mathcal{D}_t \cap  \lint a, b \rint) \;\ge\, 1\,.
\end{equation}
In order to provide a larger lower bound on the r.h.s.\ above when  $b-a$ is large, 
 we need an upper bound on the maximal monotone segment of $h_t^{\mu}$.  
For $\xi \in \gO_{N,k}$, we define
\begin{equation}\label{def:Q12}
\begin{aligned}
Q_1(h^\xi)\;&\colonequals\;
\max \left\{n\geq 1: \, \exists i \in \lint 0, N-n\rint,\, \forall\, x \in \lint i+1, i+n \rint, h^\xi(x)-h^\xi(x-1)=1- \frac{k}{N} \right\}\,,\\
Q_2(h^\xi)\;& \colonequals\;\max\left\{n\geq 1;,\, \exists i \in \lint 0, N-n \rint,\, \forall x \in \lint  i+1, i+n\rint, h^\xi(x)-h^\xi(x-1)=-\frac{k}{N}\right\}\,,
\end{aligned}
\end{equation}
and
\begin{equation}\label{maxmontonesegm}
 Q(h^\xi)\;\colonequals\;\max \left(Q_1(h^\xi), Q_2(h^\xi)\right)\,.
\end{equation}

\begin{lemma} \label{lema:lwbdderAt}
 Assuming \eqref{assumpweak:lwbdresist}, we have
\begin{equation*}
\partial_t \langle A. \rangle \; \ge\; \bar \Upsilon_N \max   \left( 1,\,     \frac{A_t \bar\delta_{\min}}{3H(t)Q\left(h_t^\mu\right)} \right)\,,
\end{equation*}
where $\bar \delta_{\min}$ and $H(t)$ are defined in \eqref{def:bardelta} and  \eqref{def:H} respectively.
\end{lemma}

\begin{proof}
In $\mathcal{B}_t$, we decompose the path $h^{\mu}_t$ into consecutive maximal  monotone segments. Note that in $\mathcal{B}_t$ every two consecutive components correspond to one admissible jump, which is a point in $\mathcal{D}_t$. As any maximal monotone component is at most of length  $Q(h_t^{\mu})$ defined in \eqref{maxmontonesegm}, 
we obtain
\begin{equation} \label{lwbd:flipcornerNum}
\# (\mathcal{D}_t \cap \lint a, b \rint) \;\ge\; \frac{1}{2} \left\lfloor \frac{b-a}{Q(h_t^{\mu})} \right\rfloor \;\ge\; \frac{1}{3} \frac{b-a}{Q(h_t^{\mu})}\,.
\end{equation}
Moreover,  we observe that 
\begin{equation} \label{upbd:areamaxhight}
\bar\delta_{\min}^{-1}\sum_{x=a}^{b} \left( h_t^\wedge(x)- h_t^{\mu}(x)\right)   \bar G(x)
 \;\le\;\bar \delta_{\min}^{-1} (b-a)    H(t)\,.
\end{equation}

Summing up all such intervals $\lint a, b \rint$ and using  \eqref{lwbd:flipcornerNum} and \eqref{upbd:areamaxhight}, we obtain
\begin{equation*}
A_t\;\le\; 3\bar\delta_{\min}^{-1} H(t) Q(\eta_t^{\mu}) \#\mathcal{D}_t\,.
\end{equation*} 
Therefore, we have 
\begin{equation*}
\partial_t \langle A. \rangle \; \ge\; \bar \Upsilon_N \# \mathcal{D}_t   \; \ge\;  \bar \Upsilon_N \frac{A_t \bar\delta_{\min}}{3H(t)Q\left(h_t^\mu\right)}\,.
\end{equation*}
This yields the desired result, combining with \eqref{upbd:trAbracket}.
\end{proof}

 For the lower bound on  $\partial_t \langle A. \rangle $ stated in Lemma \ref{lema:lwbdderAt},   we provide an upper bound on  $Q(h_t^{\mu})$  in the following lemma. We recall  $$t_{\delta}=(1+\delta)\frac{1}{2 \bar \gl_1}\log k\,.$$
 \begin{lemma}\label{lema:upbdQmu} 
Assuming $k \le N/2$, for every $\gep>0$, for all $N$ sufficiently large  we have for all $t \ge 0$,
\begin{equation}\label{upbd:flipcorners}
 \bP \left[ Q(h_t^{\mu})> N k^{-1} (\log N)^{1+\gep}\right]\;\le\; 2^{-(\log N)^{1+\gep/2}}. 
\end{equation}

\end{lemma}

\begin{proof} 
By the relation between $\xi$ and $h^\xi$ given in \eqref{heighfun},
in \eqref{def:Q12}  we use the notation $Q_1(\xi)$ by replacing  $h^\xi(x)-h^\xi(x-1)=1-k/N$ with $\xi(x)-\xi(x-1)=1$ and similarly for $Q_2(\xi)$.  We have

\begin{equation}\label{upbd:Q12equi}
\begin{aligned}
&\bP \left[ Q(h_t^{\mu})> N k^{-1} (\log N)^{1+\gep}\right]\\
=\; &\mu_{N, k} \left(Q(\xi)> N k^{-1} (\log N)^{1+\gep} \right)\\
\le\;&
\mu_{N, k} \left(Q_1(\xi)> N k^{-1} (\log N)^{1+\gep} \right)
+
\mu_{N, k} \left(Q_2(\xi)> N k^{-1} (\log N)^{1+\gep} \right)\,.
\end{aligned}
\end{equation}
We cut the line segment $\lint 1, N \rint$ into segments of length $\lfloor (\log N)^{1+\gep}/2 \rfloor$  as follows
\begin{equation*}
\left( 
\lint (i-1) \cdot \lfloor (\log N)^{1+\gep}/2 \rfloor+1,\, i \lfloor (\log N)^{1+\gep}/2 \rfloor \rint \right)_{1\le i \le \lfloor N/\lfloor (\log N)^{1+\gep}/2 \rfloor \rfloor }
\end{equation*}
except that the length of the last segment may be less than $\lfloor (\log N)^{1+\gep}/2 \rfloor$. 

Concerning the first term in the r.h.s.\ of \eqref{upbd:Q12equi}, if $Q_1(\xi)> (\log N)^{1+\gep}$, there must exists some $i_0 \in  \lint 1,\, \lfloor N/\lfloor (\log N)^{1+\gep}/2 \rfloor \rfloor  $ such that 
\begin{equation*}
 \xi(x)=1\,, \quad \forall\, x \in \lint (i_0-1)\cdot \lfloor (\log N)^{1+\gep}/2 \rfloor+1,\, i_0 \lfloor (\log N)^{1+\gep}/2 \rfloor \rint\,.
\end{equation*}
Therefore,  we have
\begin{equation}\label{upbd:probQ111}
\begin{aligned}
&\mu_{N, k} \left(\xi: Q_1(\xi) \ge (\log N)^{1+\gep} \right)\\
\; \le\;&
\frac{N}{ \lfloor (\log N)^{1+\gep}/2\rfloor} \mu_{N, k} \left(\xi:\,  \forall x \in \lint 1,\,  \lfloor (\log N)^{1+\gep}/2 \rfloor\rint, \, \xi(x)=1  \right)\\
=\;&
\frac{N}{ \lfloor (\log N)^{1+\gep}/2\rfloor} {N- \lfloor (\log N)^{1+\gep}/2 \rfloor \choose k-  \lfloor (\log N)^{1+\gep}/2 \rfloor } {N \choose k }^{-1}\\
=\;&
\frac{N}{ \lfloor (\log N)^{1+\gep}/2\rfloor} \frac{k \cdot (k-1)\cdots- \left(k- \lfloor (\log N)^{1+\gep}/2 \rfloor+1\right)}{N \cdot \left( N-1\right)  \cdots \left(N- \lfloor (\log N)^{1+\gep}/2 \rfloor+1\right)}\\
\le\;&
\frac{N}{ \lfloor (\log N)^{1+\gep}/2\rfloor} 2^{-\lfloor (\log N)^{1+\gep}/2\rfloor}\,,
\end{aligned}
\end{equation}
where the last inequality uses $k \le N/2$ and the following inequality: for all $0 < c \le b \le a$,
\begin{equation}\label{ineq:fractional}
\frac{b-c}{a-c} \; \le\;
\frac{b}{a}\,.
\end{equation}
Concerning the second term in the r.h.s.\ of \eqref{upbd:Q12equi},  by the same argument as for \eqref{upbd:probQ111}, we have
\begin{equation}\label{upbd:equilQ2}
\begin{aligned}
&\mu_{N,k}\left(\xi:\, Q_2(\xi) \ge  N k^{-1}(\log N)^{1+\gep} \right)\\
\;\le&\;
\frac{N}{ \lfloor N k^{-1} (\log N)^{1+\gep}/2\rfloor} \mu_{N, k} \left(\xi:\, \forall x \in \lint 1, \left\lfloor Nk^{-1} (\log N)^{1+\gep}/2 \right\rfloor\rint,\,  \xi(x)=0 \right)\\
\;=&\;
\frac{N}{ \lfloor Nk^{-1} (\log N)^{1+\gep}/2\rfloor}  {N- \left\lfloor Nk^{-1}(\log N)^{1+\gep}/2 \right\rfloor \choose k} {N  \choose k}^{-1}\\
\le&\;
\frac{4 k}{(\log N)^{1+\gep}} 
\cdot
\frac{(N-k)(N-k-1) \cdots (N-k-\left\lfloor Nk^{-1}(\log N)^{1+\gep}/2 \right\rfloor+1 )}{N(N-1)\cdots \left(N-\left\lfloor Nk^{-1}(\log N)^{1+\gep}/2 \right\rfloor+1 \right)}\\
\le&\;
\frac{4 k}{(\log N)^{1+\gep}}  \left(1- \frac{k}{N} \right)^{\left\lfloor Nk^{-1}(\log N)^{1+\gep}/2\right\rfloor}\\
\le&\;
\frac{4 k}{(\log N)^{1+\gep}} \cdot \exp  \left(-\frac{k}{N}  \left\lfloor Nk^{-1}(\log N)^{1+\gep}/2\right\rfloor \right)\\
\le&\;
\frac{4 k}{(\log N)^{1+\gep}} \exp  \left(-\frac{(\log N)^{1+\gep}}{3}   \right)
\end{aligned}
\end{equation}
where   we have used  \eqref{ineq:fractional} in the third last inequality. 
 Combining \eqref{upbd:probQ111} 
and \eqref{upbd:equilQ2}, we conclude the proof. 
\end{proof}

\begin{proposition} \label{prop:tdmaxconf}
For all $\gep>0$ and $t \ge t_{\delta/2}$, if $N$ is sufficiently large,  we have

\begin{equation}
\Vert P_t^\wedge- \mu_{N,k} \Vert_{\TV} \;\le\; \gep\,.
\end{equation}
\end{proposition}
Based on the informarion about $h_t^\wedge$ stated in Appendix \ref{appsec:heighfun}, we follows  the  same strategy in \cite[Section 8.2]{lacoin2016mixing} to prove Proposition \ref{prop:tdmaxconf}, which is postponed in Appendix \ref{app:propmaxconf}.

\begin{lemma}  \label{lema:highestH}
Recalling \eqref{def:H},  for any $\gep\in (0, 1/2)$
we have
\begin{equation}
\lim_{N \to \infty}\, \sup_{t \in [t_{\delta/2},\, t_{\delta}]}\,
 \bP \left[ H(t) \ge k^{1/2+\gep}N^{-1} e^{(\log N)^\upsilon } \right]\;=\;0\,,
\end{equation}
where $\upsilon<1$ is the constant stated in Assumption \ref{assum:1}.
\end{lemma}

\begin{proof}

Recalling \eqref{def:bardelta}, we have 
\begin{equation}
\begin{aligned}
H(t)
\;=\;&
\max_{x \in \lint 1,\, N \rint}   \left[ G(x)-G(x+1)\right] \cdot c(x, x+1) r(x, x+1) \left[h_t^{\wedge}(x)- h_t^{\mu}(x)\right]\\
\;\le\;&
\max_{x \in \lint 1,\, N \rint}   \left[ G(x)-G(x+1)\right]  c(x, x+1) \cdot \max_{x \in \lint 1,\, N \rint}  
 r(x, x+1) \left[h_t^{\wedge}(x)- h_t^{\mu}(x)\right]\\
 \;\le\;&
\frac{C(\delta)}{N} \max_{x \in \lint 1,\, N \rint}  
 r(x, x+1) \left[h_t^{\wedge}(x)- h_t^{\mu}(x)\right]\\
 \;\le\;&
\frac{C(\delta)}{N} \max_{x \in \lint 1,\, N \rint}  
 r(x, x+1) \max_{x \in \lint 1,\, N \rint}  \left[h_t^{\wedge}(x)- h_t^{\mu}(x)\right]\\
  \;\le\;&
  \frac{C(\delta)}{N} C_{\bbP} e^{(\log N)^\upsilon }  \max_{x \in \lint 1,\, N \rint}  \left[h_t^{\wedge}(x)- h_t^{\mu}(x)\right]\,.\\
\le\;&  \frac{C(\delta)}{N} C_{\bbP} e^{(\log N)^\upsilon } \left[
\max_{x \in \lint 1,\, N \rint}  h_t^{\wedge}(x)-\min_{x \in \lint 1,\, N \rint} h_t^{\mu}(x)  \right]\,,
\end{aligned}
\end{equation}
where we have used Lemma \ref{lema:eigfunincrement} in the second inequality and used  Assumption  \ref{assum:1} in the second last inequality.
Therefore, 
 it is sufficient to  prove    
\begin{equation*}\label{HLpath:estimate}
\begin{aligned}
 &\lim_{N \to \infty}\, \sup_{t \in [t_{\delta/2},\, t_{\delta}]}\,
 \bP \left[ \max_{x \in \lint 1,\, N \rint}  h_t^{\wedge}(x) \ge \frac{1}{2C(\delta) C_\bbP} k^{1/2+\gep} \right]\;=\;0\,,
 \\
& \lim_{N \to \infty}\, \sup_{t \in [t_{\delta/2},\, t_{\delta}]}\,
 \bP \left[ \min_{x \in \lint 1,\, N \rint}  h_t^{\mu}(x) \le - \frac{1}{2C(\delta) C_\bbP} k^{1/2+\gep}\right]\;=\;0\,.
 \end{aligned}
\end{equation*}
By Proposition \ref{prop:tdmaxconf}, we just need to show
\begin{equation}\label{HLpath:22estimate}
\begin{aligned}
 &\lim_{N \to \infty}\,
 \bP \left[ \max_{x \in \lint 1,\, N \rint}  h_t^{\mu}(x) \ge \frac{1}{2C(\delta) C_\bbP} k^{1/2+\gep} \right]\;=\;0\,,
 \\
& \lim_{N \to \infty}\, 
 \bP \left[ \min_{x \in \lint 1,\, N \rint}  h_t^{\mu}(x) \le - \frac{1}{2C(\delta) C_\bbP} k^{1/2+\gep}\right]\;=\;0\,,
 \end{aligned}
\end{equation}
since for all $t \ge t_{\delta/2}$
\begin{equation*}
\bP \left[ \max_{x \in \lint 1,\, N \rint}  h_t^{\wedge}(x) \ge \frac{1}{2C(\delta) C_\bbP} k^{1/2+\gep} \right]\;\le\;
\bP \left[ \max_{x \in \lint 1,\, N \rint}  h_t^{\mu}(x) \ge \frac{1}{2C(\delta) C_\bbP} k^{1/2+\gep} \right]+\Vert P_t^\wedge- \mu_{N,k} \Vert_{\TV}\,.
\end{equation*}

By similarity, we only write down the detals for the first event in \eqref{HLpath:22estimate}.
Let $(X_i)_{1 \le i \le N}$ be IID with common distribution as
\begin{equation*}
\bP \left( X_1=1-\frac{k}{N} \right)\;=\;\frac{k}{N}
\;=\;
1-\bP \left( X_1=-\frac{k}{N} \right)\,,
\end{equation*}
and we have
\begin{equation}\label{upbd:hightequi}
\begin{aligned}
&\mu_{N, k} \left(\max_{1 \le x \le N} h^\mu(x) \ge \frac{1}{2C(\delta) C_\bbP} k^{1/2+\gep} \right)\\
\;=\;&
\bP \left[ \max_{1 \le x \le N} \sum_{i=1}^x X_i \ge \frac{1}{2C(\delta)C_\bbP} k^{1/2+\gep} \;\Big|\; \sum_{i=1}^N X_i=0\right]\\
\;\le\;&\frac{1}{ \bP \left[ \sum_{i=1}^N X_i=0\right]}
\bP \left[ \max_{1 \le x \le N} \sum_{i=1}^x X_i \ge \frac{1}{2C(\delta)C_\bbP} k^{1/2+\gep}\right]\,.
\end{aligned}
\end{equation}
Observe that if $\sum_{i=1}^N X_i=0$, there must be $k$ many $X_i$s taking values $1-k/N$ and $N-k$ many $X_i$s taking value $-k/N$. Therefore,  using the  following estimate: for all $n \in \bbN$

\begin{equation*}
 \sqrt{ 2 \pi n} \left( \frac{n}{e} \right)^n e^{\frac{1}{12 n+1}}  \;\le\; n !\;\le\; \sqrt{ 2 \pi n} \left( \frac{n}{e} \right)^n e^{\frac{1}{12 n}},
\end{equation*}
we obtain
\begin{equation}\label{lwbd:endhit0}
\begin{aligned}
 \bP \left[ \sum_{i=1}^N X_i=0\right]
 \;&=\;
 {N \choose k} \left(\frac{k}{N} \right)^k  
 \left(1-\frac{k}{N} \right)^{N-k}\\
 &\ge \;
\sqrt{\frac{N}{2 \pi k (N-k)}} e^{\frac{1}{12N+1}-\frac{1}{12k}-\frac{1}{12(N-k)}}\\
&\ge \;
\frac{1}{2\sqrt{\pi k}}\,,
 \end{aligned}
\end{equation}
where we have used  $k \le N/2$ and the fact that both $k$ and $N$ are sufficiently large.
Moreover, since $(( \sum_{i=1}^x X_i)_x, \bP)$ is a martingale w.r.t.\ its natural filtration, we apply Markov's inequality and $L^p$ maximal inequality (cf. \cite[Theorem 5.4.3]{Durrett})  to obtain 
\begin{equation}\label{upbd:iidhighest}
\begin{aligned}
&\bP \left[ \max_{1 \le x \le N} \sum_{i=1}^x X_i \ge \frac{1}{2C(\delta)C_\bbP} k^{1/2+\gep}\right]\\
\;\le\;&
\left( \frac{1}{2C(\delta)C_\bbP} k^{1/2+\gep}\right)^{-2 \lceil 1/\gep \rceil} \bE \left[ \max_{1 \le x \le N}  \left( \sum_{i=1}^x X_i \right)^{2 \lceil 1/\gep \rceil} \right]\\
\;\le\;&
\left( \frac{1}{2C(\delta)C_\bbP}  k^{1/2+\gep}\right)^{-2 \lceil 1/\gep \rceil}  \left(  \frac{2 \lceil 1/\gep \rceil}{2 \lceil 1/\gep \rceil-1}\right)^{2 \lceil 1/\gep \rceil} \bE \left[  \left( \sum_{i=1}^N X_i \right)^{2 \lceil 1/\gep \rceil} \right]\\
\;\le \;&
\left( \frac{1}{2C(\delta)C_\bbP} k^{1/2+\gep}\right)^{-2 \lceil 1/\gep \rceil}  \left(  \frac{2 \lceil 1/\gep \rceil}{2 \lceil 1/\gep \rceil-1}\right)^{2 \lceil 1/\gep \rceil}  C(\gep) k^{\lceil 1/\gep \rceil}\\
\;\le \;&
C(\gep, \delta, \bbP) k^{-2}
\end{aligned}
\end{equation}
where the second last equality is due to binomial expansion and the IID property of $(X_i)_i$ with $\bE[X_i]=0$ and for all $m \ge 2$,
\begin{equation*}
\begin{aligned}
\vert \bE\left[ X_i^m\right] \vert \;&=\; \left\vert \frac{k}{N} \left(1-\frac{k}{N} \right)^m+ \left(1- \frac{k}{N} \right)\left( -\frac{k}{N}\right)^m \right\vert \\
&\le \;\frac{k}{N} \left( 1-\frac{k}{N}\right) \left[ \left( 1-\frac{k}{N}\right)^{m-1}+ \left( \frac{k}{N}\right) ^{m-1}\right]\\
&\le\; \frac{2k}{N}\,.
\end{aligned}
\end{equation*}
Combining \eqref{lwbd:endhit0} and \eqref{upbd:iidhighest},
we obtain an upper bound on the r.h.s.\ of \eqref{upbd:hightequi} as 
\begin{equation}\label{inprob:highpath}
\mu_{N, k} \left(\max_{1 \le x \le N} h^\mu(x) \ge \frac{1}{2C(\delta) C_\bbP} k^{1/2+\gep} \right)\;\le\; 2\sqrt{\pi} C(\gep, \delta, \bbP) k^{-3/2}\,.
\end{equation}

\end{proof}

 With Lemma \ref{lema:highestH} at hands, we are ready to prove Proposition \ref{prop:consecutstoptime}.

\begin{proof}[Proof of Proposition \ref{prop:consecutstoptime}] Take $\gep' \in (0, \psi/2)$.
 We define the event  $\mathcal{H}_N$ where  there are a lot of admissible jumps in $h_t^{\mu}$ during the time interval $[t_{\delta/2},  t_{\delta}]$ 
\begin{equation*}
\cH_N \colonequals \left\{ 
\int_{t_{\delta/2}}^{t_{\delta}}  
\ind_{ \{H(t) \;\le\; k^{1/2+\gep'}N^{-1} e^{(\log N)^\upsilon } \} \bigcap  \{Q(h_t^{\mu}) \;\le\; N k^{-1}(\log N)^{1+ \gep'}\}}  
 \dd t  \;\ge\;  (t_{\delta}-t_{\delta/2}) \left( 1-2^{-(K+1)}  \right)
   \right \}\,.
\end{equation*}
First, we show that $\cH_N$ holds with high probability. We have

\begin{equation}
\begin{aligned}\label{eventHfullmass}
\bP\left[\cH_N^{\complement}\right] 
\;&=\;
\bP\left[\int_{t_{\delta/2}}^{t_{\delta}} \ind_{  \{H(t) \;>\; k^{1/2+\gep'}N^{-1} e^{(\log N)^\upsilon } 
\} \bigcup
\{Q(h^{\mu}_t)> N k^{-1} (\log N)^{1+ \gep'}\}} \dd t \ge 2^{-(K+1)} (t_{\delta}-t_{\delta/2})\right]   \\
&\le\;
\bP\left[\int_{t_{\delta/2}}^{t_{\delta}} \ind_{
\{Q(h^{\mu}_t)> N k^{-1} (\log N)^{1+ \gep'}\}} \dd t \ge   (t_\delta-t_{\delta/2})2^{-(K+2)} \right]\\
& \quad +
\bP\left[\int_{t_{\delta/2}}^{t_{\delta}} \ind_{  \{H(t) \;>\;k^{1/2+\gep'}N^{-1}e^{(\log N)^\upsilon } \} } \dd t \ge  (t_\delta-t_{\delta/2}) 2^{-(K+2)} \right]\,.
\end{aligned}
\end{equation}
Concerning the first term in the r.h.s.\ of \eqref{eventHfullmass}, 
by Markov's inequality
we have
\begin{equation}
\begin{aligned}
&\bP\left[\int_{t_{\delta/2}}^{t_{\delta}} \ind_{
\{Q(h^{\mu}_t)> N k^{-1} (\log N)^{1+ \gep'}\}} \dd t \ge   (t_\delta-t_{\delta/2}) 2^{-(K+2)} \right]\\
\; 
\le\; &
(t_\delta-t_{\delta/2})^{-1} 2^{K+2}  \int_{t_{\delta/2}}^{t_{\delta}}  \bE\left[ \ind_{
\{Q(h^{\mu}_t)> N k^{-1} (\log N)^{1+ \gep'}\}}\right] \dd t \\
\le\;&
 2^{K+2} \bP \left[Q(h_t^\mu)> Nk^{-1} (\log N)^{1+\gep'} \right]\,,
\end{aligned}
\end{equation}
which tends to zero as $N$ goes to infinity by Lemma  \ref{lema:upbdQmu}. Similarly, about the second term in the r.h.s.\ of \eqref{eventHfullmass}, 
\begin{equation}
\begin{aligned}
&\;\bP\left[\int_{t_{\delta/2}}^{t_{\delta}} \ind_{  \{H(t) \;>\;k^{1/2+\gep'}N^{-1} e^{(\log N)^\upsilon } \} } \dd t \ge  (t_\delta-t_{\delta/2}) 2^{-(K+2)} \right] \\
\le\;&
 (t_\delta-t_{\delta/2})^{-1} 2^{K+2} \int_{t_{\delta/2}}^{t_{\delta}}  \bE \left[ \ind_{  \{H(t) \;>\;k^{1/2+\gep'}N^{-1} e^{(\log N)^\upsilon } \} }  \right] \dd t\\
\le\;&
  2^{K+2}  \sup_{t \in [t_{\delta/2},\, t_{\delta}]} \,
 \bP \left[ \{H(t) \;>\;k^{1/2+\gep'}N^{-1} e^{(\log N)^\upsilon } \} \right]
\end{aligned}
\end{equation}
which also tends to zero as $N$ goes to infinity by Lemma \ref{lema:highestH}.

From now on, we assume the event $\cA_N \cap \cH_N \cap\{\Tau_2=t_{\delta/2}\}$.  Based on \eqref{eventHfullmass}, Lemma \ref{lema:1ststoptime} and Lemma \ref{lema:ANfullmass}, we have
\begin{equation*}
\lim_{N \to \infty} \bP\left[ \cA_N \cap \cH_N \cap\{\Tau_2=t_{\delta/2}\} \right]\;=\;1\,.
\end{equation*}
 By induction, we show that $\Delta \Tau_j=\Tau_j-\Tau_{j-1}\leq 2^{-j} (t_{\delta}-t_{\delta/2})$ for all $j \in \lint 3, K \rint$. We argue by contradiction: let  $i_0$ be the smallest integer satisfying
\begin{equation*}
\Delta \Tau_{i_0} \;>\; 2^{-i_0} (t_{\delta}-t_{\delta/2})\,.
\end{equation*}
We know that 
\begin{equation}\label{combine1}
\Delta_{i_0}\langle  A \rangle \;\ge\; \int_{\Tau_{i_0-1}}^{\Tau_{i_0-1}+2^{-i_0} (t_{\delta}-t_{\delta/2})} \partial_t \langle A. \rangle \ind_{ \{H(t) \;\le\; k^{1/2+\gep'}N^{-1} e^{(\log N)^\upsilon } \} \bigcap \{ Q(h^{\mu}_t)\le N k^{-1} (\log N)^{1+ \gep'} \}} \dd t\,. 
\end{equation}
According to Lemma \ref{lema:lwbdderAt}, Lemma \ref{lema:upbdQmu} and Lemma \ref{lema:highestH}, we have a lower bound for $\partial_t\langle A.\rangle$ when the indicator function equals to $1$. 
That bound is
\begin{equation}\label{combine2}
\begin{aligned}
\partial_t\langle A.\rangle
\;&\ge\; 
  \bar \Upsilon_N \frac{A_t \bar\delta_{\min}}{3H(t)Q\left(h_t^\mu\right)}\\
  \; &\ge\; 
  \bar \Upsilon_N \frac{A_t \bar\delta_{\min}}{3 k^{1/2+\gep'}N^{-1} e^{(\log N)^\upsilon }  N k^{-1} (\log N)^{1+ \gep'}}\\
&=\;
 \bar \Upsilon_N \frac{A_t \bar\delta_{\min} k^{1/2-\gep'}}{3 e^{(\log N)^\upsilon }   (\log N)^{(1+\gep')}}\,.
 \end{aligned}
\end{equation}
Since $\Tau_2=t_{\delta/2}$ and $\Delta \Tau_j=\Tau_j-\Tau_{j-1}\leq 2^{-j} (t_{\delta}-t_{\delta/2})$ for $j<i_0$,   we know that
\begin{equation*}
\Tau_{i_0-1}\;\le\; t_{\delta/2}+  \left( t_{\delta}-t_{\delta/2}\right) \sum_{j=3}^{i_0-1}2^{-j}
\end{equation*}
and then 

\begin{equation*}
\Tau_{i_0-1}+2^{-i_0} \left(t_{\delta}-t_{\delta/2} \right)
\;\le\; t_{\delta/2}+ \left( t_{\delta}-t_{\delta/2}\right) \sum_{j=3}^{i_0}2^{-j}\;<\;t_\delta\,.
\end{equation*}
Moreover, when  the assumption $\cH_N$ holds, the indicator function 
\begin{equation*}
\ind_{ \{H(t) \;\le\; k^{1/2+\gep'}N^{-1} e^{(\log N)^\upsilon } \} \bigcap  \{Q(h_t^{\mu})\le N k^{-1} (\log N)^{1+\gep'} \} }
\end{equation*}
is equal to $1$ on a set, which is of Lebesgue measure at least
\begin{equation}\label{combine3}
 (2^{-i_0}-2^{-(K+1)}) \left(t_{\delta}-t_{\delta/2}  \right) \;\ge\; 2^{-(K+1)}\left(t_{\delta}-t_{\delta/2}  \right)\,. 
\end{equation}
 
Combining \eqref{combine1}, \eqref{combine2} and \eqref{combine3}, we obtain
\begin{equation}\label{contradiction1}
\begin{aligned}
\Delta_{i_0} \langle  A \rangle  \;&\ge\; 2^{-(K+1)} (t_\delta-t_{\delta/2})  \bar \Upsilon_N \frac{A_t \bar\delta_{\min} k^{1/2-\gep'}}{3 e^{(\log N)^\upsilon }  (\log N)^{1+\gep'}}\\
  \;&\geq\; 
  2^{-(K+1)}(t_\delta-t_{\delta/2})  \bar \Upsilon_N \frac{k^{1-i_0\psi-\gep'}}{3 e^{(\log N)^\upsilon }  (\log N)^{1+\gep'}} \,,  
  \end{aligned}
\end{equation}
where the last inequality uses the fact that $A_t > \bar \delta_{\min}^{-1}  k^{\frac{1}{2}-i_0\psi} $, 
for $t<\Tau_{i_0}$.
In addition, since we are in $\cA_N$, we know that 
\begin{equation}\label{contradiction2}
\Delta_{i_0}\langle A\rangle \;\le\;  \bar \delta_{\min}^{-2} k^{1-2(i_0-1)\psi+\frac{1}{2}\psi}\,. 
\end{equation}
However, as $i_0\geq 3$, we have
$$1-i_0 \psi - \gep'\;>\; 1-2(i_0-1)\psi +\frac{1}{2}\psi.$$ Therefore, under  Assumption \ref{assum:2} and Assumption \ref{assum:nparticles},  there is a contradiction between \eqref{contradiction1} and \eqref{contradiction2}, as long as $N$ is large enough.
\end{proof}

\appendix
\section{Upper bound on the covariance of the  probability measure $\nu$}\label{secapped:cov}

In this section, our goal is to provide an upper bound on the covariance  of the probability measure $\nu$ defined above \eqref{upbd:martquad}, which is the following. 

\begin{lemma}\label{lema:2ndmomLLN}
 For any $\gga>0$,  if $ (\log N)^{1+\gga} \le k \le N/64 $,  for all $\delta \in (0, \frac{\gga}{2(1+\gga)})$  and  for all $N$ sufficiently large  we have
\begin{equation}
\sum_{1 \le x, \, y \le N} \, \left\vert  \nu \left( \xi(x) \xi(y) \right) \right)- \nu\left( \xi(x) \right) \nu \left(\xi(y) \right \vert\;\le\; 2^{12} k^{2-\delta}\,.
\end{equation}
\end{lemma}

\begin{proof}

Concerning the diagonal terms, 
note that 
\begin{equation*}
 \sum_{x=1}^N \, \left[\nu \left( \xi(x) \right) - \left(\nu\left( \xi(x) \right) \right)^2 \right]   
 \; \le \; \sum_{x=1}^N \, \nu \left( \xi(x) \right)
 \;\le\; k\,.
\end{equation*}
Turning to off-diagonal terms, we have 
\begin{equation*}
\begin{aligned}
&\sum_{1 \le x< y \le N} \, \vert  \left[\nu \left( \xi(x) \xi(y)  \right)- \nu\left( \xi(x) \right) \nu (\xi(y) 
\right] \vert  \\
= \; & \sum_{1 \le x< y \le N} \,  \left \vert \mu_{N, 2k} \left( \xi(x) \xi(y) \ind_{\{ \bar \xi(k) \ge y\}} \right) - \mu_{N,2k}\left( \xi(x) \ind_{\{ \bar \xi(k) \ge x\}}\right) \mu_{N,2k} \left(\xi(y)\ind_{\{ \bar \xi(k) \ge y\}}\right) \right \vert\,.
\end{aligned}
\end{equation*}

To estimate the correlation, we  first prove that the position of $k^{th}$ particle is well concentrated around the midpoint:
 (recall $\bar \xi (\cdot)$ defined in \eqref{partiposition})
\begin{equation}\label{deviate:kthpart}
\mu_{N, 2k} \left( \left\vert \bar \xi(k)-\frac{N}{2} \right \vert \ge  \frac{N}{k^\delta}\right) \;=\; 
\mu_{N, 2k} \left(  \bar \xi(k)  \ge  \frac{N}{2}+ \frac{N}{k^\delta}\right)
+
\mu_{N, 2k} \left(  \bar \xi(k)  \le  \frac{N}{2}- \frac{N}{k^\delta}\right)\,.
\end{equation}
Concerning the first term in the r.h.s.\ of \eqref{deviate:kthpart} and writing  $j' \colonequals \left\lceil  \frac{N}{2}+ \frac{N}{k^\delta} \right\rceil$,  we  have 
\begin{equation}\label{upbd:kthpart222right}
\begin{aligned}
\mu_{N, 2k} \left(  \bar \xi(k)  \ge  \frac{N}{2}+ \frac{N}{k^\delta}\right)
\;&=\; \sum_{j= j'}^{N-k} 
\mu_{N, 2k} \left(  \bar \xi(k)  =j\right)\\
&=\;
\sum_{j= j'}^{N-k} {j-1 \choose k-1}{N-j \choose k} {N \choose 2k}^{-1} \\
& \le \;
\frac{N}{2}  {j'-1 \choose k-1}{N-j' \choose k} {N \choose 2k}^{-1}\,,
\end{aligned}
\end{equation}
where we have  used that the map 
\begin{equation*}
j \mapsto {j-1 \choose k-1}{N-j \choose k}
\end{equation*}
is decreasing for $j \in \lint \left\lceil  \frac{N}{2}+ \frac{N}{k^\delta} \right\rceil,\, N-k\rint $ due to that 
\begin{equation*}
{j \choose k-1}{N-j-1 \choose k} {j-1 \choose k-1}^{-1}{N-j\choose k}^{-1} \le 1
\end{equation*}
if and only if
\begin{equation*}
2kj-kN+N-j \;\ge\; 0\,.
\end{equation*}
 Using  \eqref{convex:diff} and \eqref{bd:binomial}, 
we now provide an upper bound on the r.h.s.\ of \eqref{upbd:kthpart222right} by:
 \begin{equation*}
 \begin{aligned}
 \frac{N}{2}  {j'-1 \choose k-1} {N-j' \choose k} {N \choose 2k}^{-1}
\;&=\;  
 \frac{N}{2} \cdot \frac{k}{j'} {j' \choose k} {N-j' \choose k} {N \choose 2k}^{-1}\\
 \;&=\;  
 \frac{Nk}{2j'}
 {2k \choose k} {N-2k \choose j'-k}{N \choose j'}^{-1} 
 \\
&\le\; \frac{Nk}{2j'} \frac{2^{2k \cdot H_2\left(\frac{k}{2k}\right)}}{\sqrt{2 \pi k \left(1-\frac{k}{2k} \right)}}
 \frac{2^{(N-2k)H_2\left( \frac{j'-k}{N-2k}\right)}}{\sqrt{2 \pi (j'-k) \left(1-\frac{j'-k}{N-2k} \right)}}
 \frac{\sqrt{8 \pi j' \left(1-\frac{j'}{N} \right)}}{2^{N\cdot H_2\left(\frac{j'}{N}\right)}}\\
 &\le \; 
 4 \sqrt{k} 2^{N \left( \frac{2k}{N} H_2\left( \frac{k}{2k}\right)+\frac{N-2k}{N} H_2 \left(\frac{j'-k}{N-2k} \right)-H_2\left(\frac{j'}{N} \right) \right)}\\
 &\le\;
 4 \sqrt{k}  2^{- \frac{2}{\log 2} N \cdot \frac{2k}{N} \cdot \frac{N-2k}{N} \left( \frac{k}{2k}-\frac{j'-k}{N-2k} \right)^2}\\
 & \le\;
 4 \sqrt{k} 2^{-\frac{2}{\log 2}k \left( \frac{1}{2}-\frac{j'-k}{N-2k}\right)^2}\\
& \le\;4 \sqrt{k} 2^{-\frac{2}{\log 2}k^{1-2\delta}}\,,
 \end{aligned}
 \end{equation*}
which tends to zero as $k$ tends to infinity.
Similarly, for the other term in \eqref{deviate:kthpart}, we have 
\begin{equation*}
\begin{aligned}
\mu_{N, 2k} \left(  \bar \xi(k)  \le  \frac{N}{2}- \frac{N}{k^\delta}\right)
\;&\le\; 4 \sqrt{k} 2^{-\frac{2}{\log 2}k^{1-2\delta}}\,,
\end{aligned}
\end{equation*}
and then
\begin{equation}\label{upbd:kthpartotal}
\mu_{N, 2k} \left( \left\vert \bar \xi(k)-\frac{N}{2} \right \vert \ge  \frac{N}{k^\delta}\right)
\; \le\; 
8 \sqrt{k} 2^{-\frac{2}{\log 2}k^{1-2\delta}}\,.
\end{equation}

 For $y \ge \frac{N}{2}+ \frac{N}{k^\delta}$, by \eqref{upbd:kthpartotal} we have

\begin{equation*}
\begin{aligned}
N^2\mu_{N, 2k} \left( \xi(x) \xi(y) \ind_{\{ y \le \bar\xi(k)\}}\right)
\;\le \;
N^2\mu_{N, 2k} \left(  \bar\xi(k) \ge \frac{N}{2}+ \frac{N}{k^\delta} \right)\;\le\;
4N^2 \sqrt{k} 2^{-\frac{2}{\log 2}k^{1-2\delta}}\,,\\
N^2\mu_{N, 2k} \left( \xi(x)  \ind_{\{ x \le \bar\xi(k)\}}\right) \mu_{N, 2k} \left( \xi(y)  \ind_{\{ y \le \bar\xi(k)\}}\right)  \;\le\; 4N^2 \sqrt{k} 2^{-\frac{2}{\log 2}k^{1-2\delta}}\,, 
\end{aligned}
\end{equation*}
which tend to zero as $N \to \infty$. 
Therefore, in the remaining we only need to consider $y <  \frac{N}{2}+ \frac{N}{k^\delta}$.
Furthermore,  note that
 for $j \in \lint k, N-k \rint$, under the conditional probability measure $\mu_{N, 2k} \left( \cdot \,\vert \,\bar \xi(k)=j \right)$, the first $k-1$ particles are uniformly distributed on the segment $\lint 1, j-1 \rint$. 
 Therefore, for $1\le x<y$ we have
 \begin{equation}\label{conditonal:2points}
\begin{aligned}
&\mu_{N, 2k} \left( \xi(x) \xi(y) \ind_{\{ x < y \le \bar \xi(k)\}} \,\vert\, \bar \xi(k)=j  \right)\\
=\;&
\mu_{N, 2k} \left( \xi(x) \xi(y) \ind_{\{ x < y < j\}} \,\vert\, \bar \xi(k)=j  \right)
+
\mu_{N, 2k} \left( \xi(x) \xi(y) \ind_{\{ x < y= j\}} \,\vert\, \bar \xi(k)=j  \right)\\
=\;&  \frac{(k-1)(k-2)}{(j-1)(j-2)} \ind_{\{ x < y < j\}} 
+ 
\frac{k-1}{j-1}
\ind_{\{ x < y= j\}}
\end{aligned}
 \end{equation}
and 
\begin{equation}\label{condtional:1point}
\begin{aligned}
&\mu_{N, 2k} \left( \xi(x)  \ind_{\{ x  \le \bar \xi(k)\}} \,\vert\, \bar \xi(k)=j  \right)\\
=\;&
\mu_{N, 2k} \left( \xi(x)  \ind_{\{ x < j\}} \,\vert\, \bar \xi(k)=j  \right)
+
\mu_{N, 2k} \left( \xi(x)  \ind_{\{ x = j\}} \,\vert\, \bar \xi(k)=j  \right)\\
=\;&  \frac{k-1}{j-1} \ind_{\{ x < j\}} 
+ 
\ind_{\{ x = j\}}\,.
\end{aligned}
\end{equation}
Moreover, for $i, j, \ell \in \lint \left\lfloor \frac{N}{2}- \frac{N}{k^\delta} \right\rfloor,\, \left\lceil  \frac{N}{2}+ \frac{N}{k^\delta} \right\rceil \rint$, for all $N$ sufficiently large we have
\begin{equation}\label{estimate:covcal}
\begin{aligned}
\left \vert \frac{k-1}{j} -\frac{k-1}{N/2} \right \vert\; &\le\;8\frac{k^{1-\delta}}{N}\,,\quad \quad
\left \vert i \cdot j -\frac{N^2}{4} \right \vert\; \le\;2\frac{N^2}{k^\delta}\,,\\
\left \vert \frac{(k-1)(k-2)}{(j-1)(j-2)} -\frac{(k-1)^2}{N^2/4} \right \vert\; &\le\;2^7\frac{k^{2-\delta}}{N^2}\,, \quad \quad
\left \vert \frac{(k-1)^2}{i \cdot j } -\frac{(k-1)^2}{N^2/4} \right \vert\; &\le\;2^7\frac{k^{2-\delta}}{N^2}\,.
\end{aligned}
\end{equation}

The analysis is classified in three cases, listed below by (i) for $1\le i \le 3$.
(1) For $1 \le x<y\le \frac{N}{2}- \frac{N}{k^\delta}$, by \eqref{upbd:kthpartotal}, \eqref{conditonal:2points} and \eqref{estimate:covcal}
 we have 
 \begin{equation}\label{smally:cov}
 \left \vert  \mu_{N, 2k} \left( \xi(x) \xi(y) \ind_{\{ \bar \xi(k) \ge y\}} \right)- \frac{(k-1)^2}{N^2/4}
  \right\vert
  \;\le\; 2^4 \sqrt{k} 2^{-\frac{2}{\log 2} k^{1-2\delta}}+2^7\frac{k^{2-\delta}}{N^2}\,.
 \end{equation}
By \eqref{upbd:kthpartotal},  \eqref{condtional:1point} and \eqref{estimate:covcal}, we have

\begin{align}
\left\vert \mu_{N, 2k} \left( \xi(x)  \ind_{\{  \bar\xi(k) \ge x\}}\right) -\frac{k-1}{N/2} \right\vert\;\le\;
48 \sqrt{k} 2^{-\frac{2}{\log 2} k^{1-2\delta}}+8\frac{k^{1-\delta}}{N}\,,\label{esticov:smallx}\\
\left\vert \mu_{N, 2k} \left( \xi(x)  \ind_{\{ \bar\xi(k) \ge x\}}\right)\mu_{N, 2k} \left( \xi(y)  \ind_{\{ \bar\xi(k) \ge y\}}\right)  -\frac{(k-1)^2}{N^2/4} \right\vert\
\;\le\;
2^9\frac{k^{2-\delta}}{N^2}\,,
\end{align}
which together with \eqref{smally:cov} imply for all $N$ sufficiently large, 
\begin{equation}
\left \vert \mu_{N, 2k} \left( \xi(x) \xi(y) \ind_{\{ \bar \xi(k) \ge y\}} \right)-\mu_{N, 2k} \left( \xi(x)  \ind_{\{ \bar\xi(k) \ge x\}}\right)\mu_{N, 2k} \left( \xi(y)  \ind_{\{ \bar\xi(k) \ge y\}}\right) \right\vert 
\;\le\;
2^{10}\frac{k^{2-\delta}}{N^2}\,.
\end{equation}
This is sufficient for those coordinates $1 \le x< y \le \frac{N}{2}- \frac{N}{k^\delta}$.

(2) Now we head to deal with the case $ y \in \lint \frac{N}{2}- \frac{N}{k^\delta},\, \frac{N}{2}+\ \frac{N}{k^\delta}\rint$ and $x < \frac{N}{2}- \frac{N}{k^\delta}$. 
 By \eqref{upbd:kthpartotal}, \eqref{conditonal:2points}  and \eqref{estimate:covcal} we have 
 \begin{equation}\label{estimate:smallxniceyy}
 \begin{aligned}
&\left\vert  \mu_{N, 2k} \left( \xi(x) \xi(y) \ind_{\{  \bar \xi(k)\ge y>x\}} \right)- \frac{(k-1)^2}{N^2/4}\mu_{N, 2k} \left(y<\bar \xi(k)<\frac{N}{2}+ \frac{N}{k^\delta} \right)-\frac{k-1}{N/2} \mu_{N, 2k} \left(\bar \xi(k)=y \right) \right\vert\\
&\le\;
2^8\frac{k^{2-\delta}}{N^2}+8 \sqrt{k} 2^{-\frac{2}{\log 2}k^{1-2\delta}}\,,
 \end{aligned}
 \end{equation}
 where we have used $\mu_{N, 2k}(\bar \xi(k)=y) \le  \mu_{N, 2k}( \xi(y)=1)=\frac{2k}{N}$. 
Moreover, 
by \eqref{upbd:kthpartotal},  \eqref{condtional:1point}  and \eqref{estimate:covcal}, we have

\begin{equation*}
\begin{aligned}
&\;\left\vert \mu_{N, 2k} \left( \xi(y)  \ind_{\{ y  \le \bar \xi(k)\}}   \right)-  \frac{k-1}{N/2} \mu_{N, 2k}\left(y< \bar \xi(k)<\frac{N}{2}+\frac{N}{k^\delta} \right) -\mu_{N, 2k}\left( \bar \xi(k)=y \right)  \right\vert\\
\;\le&\;8
\frac{k^{1-\delta}}{N}+8 \sqrt{k} 2^{-\frac{2}{\log 2}k^{1-2\delta}}
\end{aligned}
\end{equation*}
which together with \eqref{esticov:smallx} implies
\begin{equation}\label{est:smallxbigy}
\begin{aligned}
\Bigg\vert I_{N, 2k}(x, y) -\frac{(k-1)^2}{N^2/4}\mu_{N, 2k} \left(y<\bar \xi(k)<\frac{N}{2}+\frac{N}{k^\delta}\right) -\frac{k-1}{N/2}\mu_{N, 2k}\left(\bar \xi(k)=y \right)\Bigg\vert\;\le\; 2^9 \frac{k^{2-\delta}}{N^2}
\end{aligned}
\end{equation}
with  $I_{N, 2k}(x, y) \colonequals \mu_{N, 2k} \left( \xi(x)  \ind_{\{  \bar \xi(k)\ge x\}}   \right) \mu_{N, 2k} \left( \xi(y)  \ind_{\{  \bar \xi(k)\ge y\}}   \right)$. Then by 
\eqref{est:smallxbigy} and \eqref{estimate:smallxniceyy}, we obtain for all $N$ sufficiently large

\begin{equation*}
\left\vert \mu_{N, 2k} \left( \xi(x)  \ind_{\{  \bar \xi(k)\ge x\}}   \right) \mu_{N, 2k} \left( \xi(y)  \ind_{\{  \bar \xi(k)\ge y\}}   \right)-\mu_{N, 2k} \left( \xi(x) \xi(y) \ind_{\{  \bar \xi(k)\ge y>x\}} \right) \right\vert\;\le\; 2^{10} \frac{k^{2-\delta}}{N^2}\,,
\end{equation*}
which is sufficient for 
 $ y \in \lint \frac{N}{2}- \frac{N}{k^\delta},\, \frac{N}{2}+\ \frac{N}{k^\delta}\rint$ and  $x < \frac{N}{2}- \frac{N}{k^\delta}$.

(3) Concerning $  \frac{N}{2}-\ \frac{N}{k^\delta} \le x < y \le   \frac{N}{2}+\ \frac{N}{k^\delta}$, we have
\begin{equation}\label{estim:midsec}
\begin{aligned}
&\mu_{N, 2k} \left( \xi(x) \xi(y) \ind_{\{ y \le \bar\xi(k)\}}\right)+ \mu_{N, 2k} \left( \xi(x) \ind_{\{ x \le \bar\xi(k)\}} \right)  \mu_{N, 2k} \left( \xi(y) \ind_{\{ y \le \bar\xi(k)\}}\right)\\
\le\;&
\mu_{N, 2k} \left( \xi(x) \xi(y) \right)+ \mu_{N, 2k} \left( \xi(x)  \right)  \mu_{N, 2k} \left( \xi(y)\right)\\
\le \; &
\frac{k(k-1)}{N(N-1)}+ \frac{k^2}{N^2} \;\le\; \frac{2k^2}{N^2}\,,
\end{aligned}
\end{equation}
and furthermore
\begin{equation}
\# \left\{(x, y):\;  \frac{N}{2}-\ \frac{N}{k^\delta} \le x < y \le   \frac{N}{2}+\ \frac{N}{k^\delta} \right\} \;\le\; \frac{4N^2}{k^{2\delta}}\,,
\end{equation}
which allows to conclude the proof together with  the analysis in (1) and (2).
\end{proof}

\section{Height function}\label{appsec:heighfun}
Recalling \eqref{heighfun} and \eqref{gen:height} and setting
\begin{equation*}
f(t,x) \;\colonequals\; \bE \left[h_t^\wedge(x) \right]\,, \quad  \forall\, t \ge 0\,,\;\; \forall\,  x \in \lint 0, N \rint\,,
\end{equation*}
 for $1 \le x<N $  we have
\begin{equation*}
\begin{aligned}
\frac{\partial}{\partial t} f(t, x) 
\;&= \; 
\bE \left[(\cL h_t^\wedge)(x) \right]\\
\;&=\;
c(x,x+1) \bE\left[\left(h_t^\wedge(x+1)-h_t^\wedge(x)\right)-\left(h_t^\wedge(x)-h_t^\wedge(x-1)\right) \right]\\
\;&=\; c(x, x+1) \Delta_x f
\end{aligned}
\end{equation*}
where 
\begin{equation*}
\Delta_x f \;\colonequals\; f(t, x+1)+f(t, x-1)-2f(t,x)\,.
\end{equation*}
In this section, our goal is to  solve the following equation:

\begin{equation}\label{eq:randheat}
\begin{cases}
\frac{\partial}{\partial t} f(t, x) \;=\; c(x,x+1)  \Delta_x f, &\forall\, x \in \{1, \cdots, N-1 \}\,,\\
f(t, 0)\;=\;f(t, N)\;=\;0\,,  &\forall\, t \ge 0\,,\\
f(0, x)\;=\;h_0^\wedge(x), &\forall\, x \in \{ 0, N \}\,.
\end{cases}
\end{equation}
Note that the matrix corresponding to the operator $c(x, x+1)\Delta_x$ is the following $N-1$ by $N-1$ matrix $\hat A$ given by
{\footnotesize
\begin{equation*}
\begin{bmatrix}
    -2c(1, 2) & c(1,2) & 0  & 0 & \dots  & 0 & 0& 0\\
    c(2,3) & -2c(2,3) & c(2,3) & 0 & \dots  & 0 & 0& 0\\
    0 & c(3, 4) & -2c(3,4) & c(3,4) & \dots  & 0 & 0& 0\\
    \vdots & \vdots & \vdots& \vdots & \vdots& \vdots    & \ddots  \\
     0 & 0 &0 & 0 & \dots  & c(N-2,N-1)& -2c(N-2,N-1) & c(N-2,N-1)\\
    0 & 0 &0 & 0 & \dots  & 0& c(N-1,N) & -2c(N-1,N)
\end{bmatrix}\,,
\end{equation*}
} i.e.
\begin{equation}\label{def:hatA}
\begin{cases}
\hat A(1, 1)=-2c(1, 2),\; \hat A(1, 2)=c(1, 2), \\
\hat A(x, x-1)=c(x,x+1),\,
\hat A(x, x)=-2c(x,x+1),\,
\hat A(x, x+1)=c(x,x+1),\; 1<x<N-1\,,\\
\hat A(N-2, N-1)=c(N-1, N),\;
\hat A(N-1, N-1)=-2c(N-1, N),\\
\end{cases}
\end{equation}
and all other  unmentioned entries are zero. 
In order to apply the method in \cite[Lemma 12.2]{LPWMCMT} to obtain an orthogonal basis,  we first define a probability measure $\nu$ on $\{1, 2 \cdots, N-1\}$ satisfying
\begin{equation}\label{detailconds}
 \nu(x)c(x, x+1)\;=\; \nu(x+1)c(x+1, x+2)\,, \quad \forall\, x \in \{1, \cdots, N-2 \}\,,
\end{equation}
from which we obtain 
\begin{equation*}
 \nu(x)
\;=\;
 \nu(1) \frac{c(1, 2)}{c(x,x+1)}\,,
\end{equation*}
and then 
\begin{equation}\label{def:measnu}
\nu(x)\;=\; \frac{r(x, x+1)}{\sum_{i=1}^{N-1} r(i, i+1)}\,.
\end{equation}
 Our next step is to prove the existence of  an orthonormal basis $(\fg_i)_{1 \le i \le N-1}$ with corresponding eigenvalue $ -\gk_i$ of the operator $c(x,x+1)\Delta_x$, i.e.
\begin{equation}\label{hatA:eigfunval}
\begin{cases}
 c(x,x+1)\Delta_x \fg_i\;=\;-\gk_i \fg_i (x), \text{ for } x \in \{1, 2, \cdots, N-1 \}\,,\\
 \langle \fg_i, \fg_j\rangle_\nu\;=\; \delta_{i, j}\,.
\end{cases}
\end{equation}
Assuming the existence of such a basis, we know that 
 $\{e^{-\gk_i t} \fg_i\}_{1 \le i \le N-1}$ is a basis of the  operator in \eqref{eq:randheat}.
Hence, we have
\begin{equation}\label{sol:heightfun}
f(t, x)\;=\; \bE[h_t^\wedge(x)]\;=\;\sum_{i=1}^{N-1} \langle h_0^\wedge, \fg_i \rangle_\nu \cdot  e^{-\gk_i t} \fg_i(x)\,.
\end{equation}

\begin{lemma} Such a basis  $\{ \fg_i\}_{1 \le i \le N-1}$ exists.
\end{lemma}

\begin{proof}
We follow the approach  in \cite[Lemma 12.2]{LPWMCMT} to provide a proof for the  sake of completeness. Recalling \eqref{def:hatA}, 
we define another 
matrix $\hat B$ by
\begin{equation*}\label{def:hatB}
 \hat B(x, y)\; \colonequals \; \nu(x)^{1/2} \hat A(x, y) \nu (y)^{-1/2}, \quad \forall\, x, y \in \{1, \cdots, N-1 \}\,.
\end{equation*}
By \eqref{def:hatA} and \eqref{detailconds}, we know that 
$\hat B(x, y)=\hat B(y, x)$, and thus  $\hat B$ is diagonalizable. 
Let $$(\hat g_i)_{1\le i \le N-1}:\, \{1, \cdots, N-1 \} \mapsto \bbR $$ be the orthonormal eigenfunctions of $\hat B$ with 
 corresponding eigenvalues $(-\gk_i)_{1 \le i \le N-1}$ satisfying 
 \begin{equation}\label{eig:randlaplace}
 -\gk_1 \;\ge\; -\gk_2 \;\ge\; \cdots \;\ge\; -\gk_{N-1}\,,
 \end{equation}
 i.e. 
 \begin{equation*}
 \begin{cases}
 \hat B \hat g_i \;=\; -\gk_i \hat g_i\,, &\text{ for } i \in \{1, 2, \cdots, N-1 \},\\
 \langle\hat g_i, \hat g_j\rangle \; =\; \delta_{i, j}\,, &\text{ for } i, j \in \{1, 2, \cdots, N-1 \}\,,
 \end{cases}
 \end{equation*}
where $\langle \cdot , \cdot \rangle$ is the usual inner product in $\bbR^{N-1}$. 
Moreover, 
writing $D \colonequals  \mathrm{diag}(\hat \nu(1), \cdots, \hat \nu(N-1))$, we have
\begin{equation*}
\hat B=  D^{1/2}\hat A D^{-1/2}
\end{equation*}
and thus $\hat B \hat g_i= D^{1/2}\hat A D^{-1/2} \hat g_i=-\gk_i \hat g_i$, which implies 
\begin{equation*}
\hat A D^{-1/2}\hat g_i \;=\; -\gk_i D^{-1/2} \hat g_i\,.
\end{equation*}
Viewing $\hat g_i$ as a vector in $\bbR^{N-1}$, we have 
\begin{equation*}
\delta_{i, j}\;=\; 
\langle \hat g_i, \hat g_j \rangle
\;=\; 
\hat g_i^T \hat g_j
\;=\;
  (D^{-1/2}\hat g_i )^T D D^{-1/2} \hat g_j
  \;=\;
  \langle D^{-1/2}\hat g_i , D^{-1/2}\hat g_j \rangle_{\nu}\,,
\end{equation*}
where $\hat g_i^T$ denotes the transpose of the vector $\hat g_i$. Therefore $$\fg_i \colonequals D^{-1/2} \hat g_i$$ is an eigenfunction of $\hat A$ with eigenvalue $-\gk_i$ for all $i \in \{1, 2, \cdots N-1 \}$ and satisfies \eqref{hatA:eigfunval}.
\end{proof}

\begin{lemma}\label{lema:appendeigvlw}
 Concerning the eigenvalues in \eqref{eig:randlaplace},  we have
\begin{equation}
0\;>\;-\gk_1\;>\;-\gk_2\;>\; \cdots \;>\;-\gk_{N-1}\,.
\end{equation}
\end{lemma}

\begin{proof}
For convention, we define $\fg_i(0)=\fg_i(N) \colonequals 0$ for all $1 \le i \le N-1$. Since
 for all $1 \le x\le N-1$,  
\begin{equation}\label{eigfun:hatAform}
c(x,x+1) \left[ \fg_i(x-1)+ \fg_i(x+1)-2\fg_i(x) \right]\;=\;-\gk_i \fg_i(x)
\end{equation}
 we have
 \begin{equation}\label{hatg:iterate}
 \fg_i(x+1)\;=\;\left(2-\frac{\gk_i}{c(x,x+1)} \right) \fg_i(x)-\fg_i(x-1)\,.
 \end{equation}
By \eqref{hatg:iterate} we know that $\fg_i(1)\neq 0$. Otherwise $\fg_i \equiv 0$, a contradiction.
Without lost of generality, we   always assume $\fg_i(1)>0$.
Furthermore, if there exists $x \in \{1, 2, \cdots, N-1, N \}$ such that $\fg_i(x-1)=\fg_i(x)=0$, then $\fg_i \equiv 0$ which is a contradiction to $\fg_i$ being an eigenfunction.

Now we argue that $\gk_i>0$.  Take $x_i \in \{ 1, \cdots, N-1\}$ such that 
\begin{equation*}
\fg_i(x_i)\;=\;\max_{0 \le x \le N}\, \fg_i(x)>0\,.
\end{equation*}
Plugging $x_i$ in \eqref{eigfun:hatAform},  we obtain that $-\gk_i \le 0$. While if $-\gk_i=0$, then by \eqref{eigfun:hatAform} we have $\fg_i(x) \equiv \fg_i(x_i)$ for all $x \in \{ 0, 1, \cdots, N\}$ which is a contradiction to  $0=\fg_i(0)<\fg_i(1)$.  Therefore, we have $-\gk_i<0$.

Now we show that all $(-\gk_i)_{1 \le i \le N-1}$ are different. It is sufficient to consider $\fg_i/\fg_i(1)$. Observe that 
 by \eqref{hatg:iterate}, $-\gk_i$ determines the eigenfunction $\fg_i/\fg_i(1)$. Since the eigenfnctions $(\fg_i/\fg_i(1))_{1\le i\le N-1}$ are different, we obtain that
\begin{equation*}
0\;>\;-\gk_1\;>\;-\gk_2\;>\; \cdots\;>\; -\gk_{N-1}\,.
\end{equation*}
\end{proof}

For any prefixed integer $K_0$, 
 we now  turn to estimate $(\gk_i)_{1 \le i \le K_0}$, and their corresponding shapes of the eigenfunction $(\fg_i)_{1 \le i \le K_0}$. 
For $\gk>0$, in view of \eqref{hatg:iterate}  we  define $\fg^{(\gk)}(0) \colonequals  0$, $\fg^{(\gk)}(1) \colonequals 1$  and for $x \ge 1$

\begin{equation}\label{2hatg:iterate}
 \fg^{(\gk)}(x+1)\;\colonequals \;\left(2-\frac{\gk}{c(x,x+1)} \right) \fg^{(\gk)}(x)-\fg^{(k)}(x-1)\,.
 \end{equation}
 Then  $-\gk$ is an eigenvalue of $\hat A$ if and only if  the iteration above provides $\fg^{(\gk)}(N)=0$. For  simplicity of notations, in the sequel we  always ignore "$(\gk)$" in $\fg^{(\gk)}$. Moreover, 
we define for $x \ge 1$
\begin{equation}\label{def:bb}
 \bb (\gk, x)
\;\colonequals \;-
\frac{\fg(x)-\fg(x-1)}{\fg(x-1)}\,,
\end{equation}
with the convention that  $ \bb (\gk,x)=\overline \infty$ if $\fg(x-1)=0$, 
and  $\overline \bbR \colonequals \bbR \cup \{\overline \infty\}$ being the  Alexandrov compactification of $\bbR$.
Thus $-\gk$ is an eigenvalue if and only if $\bb(\gk, N)=1$, since $\fg(N-1) \neq 0$. 
By \eqref{2hatg:iterate}, we have
\begin{equation}\label{iterate:hatb}
\bb(\gk, x+1)\;=\; \frac{\bb(\gk, x)}{1-\bb(\gk, x)}+\gk r(x,x+1)\,.
\end{equation}
Then, given a fixed $r>0$, we define $\hat\Xi^{(r)}: \bbR_+\times \overline \bbR \to \overline \bbR$ as
\begin{equation}\label{def:hatXi}
\hat \Xi^{(c)}(\gk, \bb)\;=\;\frac{\bb}{1-\bb}+\gk r\,.
\end{equation}
The function $\bb\mapsto \hat\Xi^{(r)}(\gk, \bb)$ may have zero, one or two fixed points depending on the values of $\gk$ and $\bb$, see Figure~\ref{fig:fixpoint2}.
\begin{figure}[H]
	\centering
	\begin{tikzpicture}[scale=0.5,smooth];
	\draw[->] (0,-1)--(0,8) node[right]{$y$};
	\node at (4,6){$\gk r > 4$};
	\draw[->] (-1,0)--(8,0) node[anchor=north]{$\bb$};

	\draw[black = solid,  very thick ] plot
	[domain=-1:0.8](\x,{(\x/(1-\x)+5)*0.8});

	\draw[black = solid,  very thick ] plot
	[domain=1.2:8](\x,{(\x/(1-\x)+5)*0.8});;

	\draw[dashed, black = solid,  very thick ] plot
	[domain=0:8](\x,{\x*0.8});

	\draw (0,0) node[anchor=north east]{\small $0$};

	\draw[densely dashed,blue] (1,7.5) -- (1,0);
	\node at (0.8,0)[below]{$1$};

	\draw[->] (10,-1)--(10,8) node[right]{$y$};
	\node at (14,6){$\gk r = 4$};
	\draw[->] (9,0)--(18,0) node[anchor=north]{$\bb$};

	\draw[black = solid,  very thick ] plot
	[domain=9:10.8](\x,{((\x-10)/(11-\x)+4)*0.8});

	\draw[black = solid,  very thick ] plot
	[domain=11.2:18](\x,{((\x-10)/(11-\x)+4)*0.8});;

	\draw[dashed, black = solid,  very thick ] plot
	[domain=10:18](\x,{(\x-10)*0.8});

	\draw (10,0) node[anchor=north east]{\small $0$};

	\draw[densely dashed,blue] (11,7.5) -- (11,0); \node at (10.8,0)[below]{$1$};

	\draw[->] (20,-1)--(20,8) node[right]{$y$};
	\node at (25,6){$\gk r < 4$};
	\draw[->] (19,0)--(29,0) node[anchor=north]{$\bb$};
	\draw[black = solid,  very thick ] plot
	[domain=19:20.8](\x,{((\x-20)/(21-\x)+3)*0.8});

	\draw[black = solid,  very thick ] plot
	[domain=21.2:29](\x,{((\x-20)/(21-\x)+3)*0.8});

	\draw[dashed, black = solid,  very thick ] plot
	[domain=20:28](\x,{(\x-20)*0.8});

	\draw (20,0) node[anchor=north east]{\small $0$};

	\draw[densely dashed,blue] (21,7.5) -- (21,0); \node at (20.8,0)[below]{$1$};
	\end{tikzpicture}
	\caption{In the figures above,  solid lines depict the function $\hat\Xi^{(r)}(\bb,\gk)$ with $\gk>0$ fixed,  the black dashed lines stand for $y=\bb$, and the blue dashes lines are $\bb=1$. There are  three cases about the number of fixed points of the map $\bb \mapsto \hat\Xi^{(r)}(\gk, \bb)$  according to the relation between $\gk r$ and $4$ as shown above.}
	\label{fig:fixpoint2}
\end{figure}
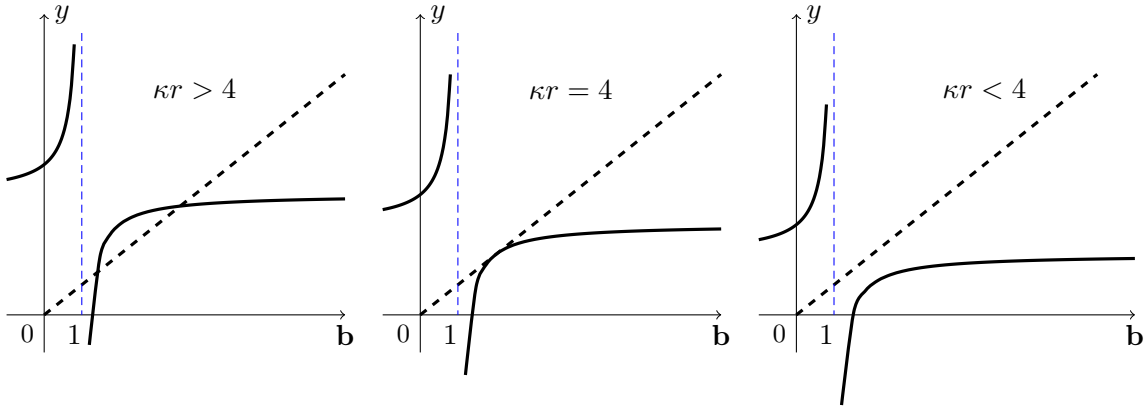

\noindent
If $\bb\mapsto \hat \Xi^{(r)}(\gk,  \bb)$ has fixed points $\bb_1$ and $\bb_2$ (not necessarily distinct) such that  $ \bb_1\le \bb_2$, 
we define $I(\gk, r)=[\bb_1, \bb_2]$.  Otherwise, set $I(\gk, r)=\varnothing$ .
We then define the ``angle mapping'' function
\begin{equation}\label{def:hatvarphi}
\hat \varphi(r,\gk,\theta)\;:=\; \inf\{ \theta'\ge \theta
+ \pi \ind_{I(\gk, r)}(\tan \theta)
\ : \ \tan \theta'=  \hat \Xi^{(r)}(\gk,\tan \theta) \}\,.
\end{equation}
We now recursively define  an ``angle'' $\theta(\gk, x)$ by setting $\theta(\gk, 1)= \pi/2$ (due to $\bb(\gk, 1)=\overline \infty$) and for $x \in  \lint 1, N \rint$,
\begin{equation}\label{def:hatmaptheta}
\hat \theta(\gk,x+1)\;:=\;  \hat \varphi(r(x, x+1), \gk,\theta(\gk,x))
\end{equation}
with the convention $\tan(\pi/2+k\pi)=\overline \infty$ for $k\in \bbZ$.
Similar to  \cite[Lemma 2.4]{Yang2024SpectralLight}, we have:

\begin{lemma}\label{lema:hatthetacont}
	The map
	$\gk \mapsto \hat \theta(\gk,x)$, defined in \eqref{def:hatmaptheta}, is continuous and strictly increasing for any $x \in \lint 2, N \rint$.
\end{lemma}
Hence, $-\gk<0$ is an eigenvalue if and only if $\hat \theta(\gk, N)=\frac{\pi}{4}+ n \pi$ where $1 \le n \le N-1$ and $k \in \bbZ$. Now we head to estimate the  eigenvalues $(-\gk_i)_{1 \le i \le K_0}$.  We recall that $\bb(\gk, N)=\tan(\hat \theta(\gk, N))$.
By  \eqref{iterate:hatb}, we have 
\begin{equation}\label{iterate:laplace}
\bb(\gk, x+1)-\bb(\gk, x)\;=\;
\frac{\bb(\gk, x)^2}{1-\bb(\gk, x)}+\gk r(x, x+1)\,.
\end{equation}
We define $\fB(\gk, x) \colonequals \bb(\gk, x) N$,and $\gk=\ga  N^{-2}$. For simplicity of notations, we ignore $\gk$ in $\bb(\gk, x)$ and $\fB(\gk, x)$. The equation \eqref{iterate:laplace} implies
\begin{equation}\label{iterate:cB}
N \left[\fB(x+1)-\fB(x)   \right]
\;=\;
\frac{\fB(x)^2}{1-\fB(x)N^{-1}}+\ga r(x, x+1)\,,
\end{equation}
with initial condition $\fB(1) \colonequals \overline \infty$. 
With the same motivation as that in \cite[Eq.\ (3.11)]{Yang2024SpectralLight}, 
we define $\fA(x) \colonequals 1/\fB(x)$, i.e. $\bb(x) \colonequals  \frac{1}{\fA(x) N}$ with $\fA(1) \colonequals 0$. By \eqref{iterate:hatb}, we have
\begin{equation*}
\fA(x+1)\;=\;\frac{\fA(x)-\frac{1}{N}}{1+N \gk r(x, x+1) \fA(x)- \gk r(x, x+1)}
\end{equation*}
and then
\begin{equation}\label{rescale:fA}
N \left[\fA(x+1)-\fA(x) \right]\;=\;
-\frac{1+r(x, x+1) \ga \fA(x)^2-\frac{\ga}{N} r(x, x+1)\fA(x)}{1+\frac{r(x, x+1)}{N} \ga \fA(x)-\frac{\ga}{N^2} r(x, x+1)}\,.
\end{equation}
Then by the same strategy as that in  \cite[Section 3]{Yang2024SpectralLight}, we  use \eqref{rescale:fA} and \eqref{iterate:cB} alternatively to obtain 
Theorem \ref{th:lapaceeigenshape}--\eqref{lapace:eigval}. We leave the details for the readers.

\begin{theorem}\label{th:lapaceeigenshape}
  Under the assumption \eqref{LLN}, for any fixed positive integer $K_0$ and  for $1\le i \le K_0$ we have

\begin{equation}\label{lapace:eigval}
\lim_{N \to \infty} \,N^2\left\vert  \gk_i -2\left(1-\cos\left( \frac{\pi i}{N}\right) \right) \right\vert\;=\;0\,, 
\end{equation}
and 
\begin{equation}\label{lapalace:eigshape}
\lim_{N \to \infty} \sup_{x \in \lint 1, N\rint }  \left\vert \frac{\fg_i(x)}{\fg_i(1)} \sin \left( \frac{i\pi  }{N} \right) - \sin \left( \frac{i\pi x }{N} \right) \right\vert\;=\;0\,.
\end{equation}
\end{theorem}

We move to give a sketch for  \eqref{lapalace:eigshape}.
Note that the functions $(\fh_i)_{1 \le i\le N-1}$ given by 
\begin{equation}\label{fun:fg}
\fh_i(x)\; \colonequals \; \sin \left( \frac{i \pi x}{N} \right) \quad x \in \lint 0, N \rint\,,
\end{equation}
are the  eigenfunctions corresponding to the matrix $\hat A$ in \eqref{def:hatA} when $c(j-1,j) \equiv 1$, and their corresponding eigenvalues are 
\begin{equation*}
 -\overline \gk_i  \;\colonequals\; -2\left[1-\cos\left( \frac{i\pi}{N}\right) \right]\,.
\end{equation*}
To prove \eqref{lapalace:eigshape}, 
for $1\le i \le K_0$ we set
$$\tilde \fg_i \colonequals \frac{\sin\left(\frac{i \pi}{N} \right)}{\fg_i(1)} \fg_i\,. $$ 
Moreover, by \eqref{def:bb} and $\bb(\gk_i, x)=\fB(\gk_i, x)/N$, for $x \ge 2$ we have
\begin{equation}\label{rec:fgrandeigfun1st}
\tilde \fg_i(x) \;=\; \left[1- N^{-1} \fB \left(\gk_i,  x
\right) \right] \tilde \fg_i(x-1)\,,
\end{equation}
where we recall $\tilde \fg_i(1)= \sin \left( \frac{i \pi}{N}\right)$. To compare with the deterministic shape of $\fh_i$, 
we define $(\overline \fB(\overline \gk_i, x))_{1 \le x \le N}$ and $(\overline \bb(\overline \gk_i, x))_{1\le x \le N}$ with $\overline \bb(\overline \gk_i, x)=\overline \fB(\overline \gk_i, x)/N$ 
by setting $\overline \fB \left(\overline \gk_i, 1 \right)= \overline \infty$ and for $x \in \lint 1, N\rint$ (with the convention $1/0 \colonequals \overline \infty$ and $1/\overline \infty=0$), 
\begin{equation}\label{def:fBbreveB}
\begin{aligned}
\overline \bb(\overline \gk_i, x ) \;&\colonequals \;-\frac{\fh_i(x)-\fh_i(x-1)}{\fh_i(x-1)} \,,\\
\overline \fB\left(\overline \gk_i,  x+1\right) \;&\colonequals\;
\frac{  \overline \fB\left(\overline \gk_i, x\right)}{1-\overline \fB \left(\overline \gk_i, x\right)N^{-1 }}+ N \overline \gk_i\,,
\end{aligned}
\end{equation}
which are 
the corresponding deterministic versions of \eqref{def:bb} and \eqref{iterate:cB}  with $\gk_i$ replaced  by $\overline \gk_i$ and  $r(x, x+1)\equiv 1$.
Similar to \eqref{rec:fgrandeigfun1st}, we have
\begin{equation}\label{rec:fhdetereigfun1st}
\fh_i(x)\;=\; \left[1- N^{-1} \overline \fB \left(\overline \gk_i,  x \right) \right] \fh_i(x-1)\,.
\end{equation}
Using \eqref{rec:fgrandeigfun1st} and  \eqref{rec:fhdetereigfun1st},  we can apply the strategies in  \cite[Section 4]{Yang2024SpectralLight} to prove Theorem \ref{th:lapaceeigenshape}--\eqref{lapalace:eigshape}.  We leave the details for the readers.

\section{Proof of Proposition \ref{prop:tdmaxconf}}\label{app:propmaxconf}

The aim of this section is to prove Proposition \ref{prop:tdmaxconf}, based on Section \ref{appsec:heighfun} and 
 the censoring inequality established by Peres and  Winkler
\cite{peres2013can}. Recalling \eqref{def:Thetaspins}, 
a censoring scheme $\mathcal{C}: [0, \infty) \to \mathcal{P}(\Theta)$ is a deterministic c\`{a}dl\`{a}g function where $\mathcal{P}(\Theta)$ is the set of all subsets of $\Theta$. 
Recalling
\begin{equation*}
t_{\delta} \;=\;\left(1+ \delta \right)\frac{1}{2 \bar \gl_1}\log k\,,
\end{equation*}
we describe the 
 censoring scheme for the dynamics in terms of  height functions as in \cite[Section 8.2]{lacoin2016mixing}: 
\begin{equation*}
\cC(t)\colonequals 
\begin{cases} 
\emptyset &\text{ if } t \in [0, t_{\delta/2})\,,\\
\{(x_i, z) \in \Theta: x_i \colonequals \lceil i \delta N \rceil\,, \text{ for } i=1, \cdots, \bar K \colonequals \lfloor  1/\delta\rfloor -1\}  &\text{ if } t \in  [ t_{\delta/2},  t_{\delta})\,,\\
\emptyset &\text{ if } t \in [ t_{\delta}, \infty)\,.\\
\end{cases}
\end{equation*}
Let $(  h_t^{\wedge, \cC})_{t \ge 0}$ be the dynamics constructed by the graphical construction in Subsection \ref{subsec:graphconstr} with the following additional rule: an update at  a parallelogram centered at $(x, z)$ at time $t$ is performed if $(x, z) \not\in \cC(t)$.
Moreover, let
$( P_t^{\wedge, \cC})_{t \ge 0}$ denote its corresponding  marginal distribution.  
By   Peres–Winkler censoring inequality  \cite[Theorem 1.1]{peres2013can} and the contraction property (cf. \cite[Exercise 4.2]{LPWMCMT}), for all $t \ge t_{\delta}$ we have  
\begin{equation*}\label{contcen:ineq}
\Vert P_t^\wedge -\mu_{N,k} \Vert_{\TV}
\;\le\; 
\Vert  P_t^{\wedge, \cC} -\mu_{N,k} \Vert_{\TV}
\;\le\; 
 \Vert  P_{t_{\delta}}^{\wedge, \cC} -\mu_{N,k} \Vert_{\TV}\,.
\end{equation*}
Thus, 
in the remaining of this subsection it is  sufficient  to prove that for any $\gep \in (0, 1)$ and $\delta>0$, if $N$ is sufficiently large we have
\begin{equation}\label{censineq}
 \Vert  P_{t_{\delta}}^{\wedge, \cC} -\mu_{N,k} \Vert_{\TV}\;\le\; \gep\,.
\end{equation}
For the proof of \eqref{censineq},
 we first define the skeleton of $\xi \in \gO_{N, k}$ to be
\begin{equation*}
\forall i \in \left\{1, 2, \cdots, \bar K \right\}\,,\quad \quad \bar h^\xi(i) \;\colonequals\; h^\xi(x_i)\,,
\end{equation*}
where $x_i= \lceil i \delta N \rceil$ and $\bar K \colonequals \lfloor  1/\delta\rfloor -1$. As in \cite[Section 8.2]{lacoin2016mixing}, our strategy is to show the following: 
\begin{itemize}
\item At time $t_{\delta/2}$, the distribution of the skeleton of $ h_{t_{\delta/2}}^\wedge$ is close to  its equilibrium. 

\item  In the time interval $[t_{\delta/2}, t_{\delta})$,  we  put the configurations in the segments between
skeleton points to equilibrium.
\end{itemize}
For the first step above, 
we need the following estimate about the height function.

\begin{lemma}\label{lema:upbdheight}
Under the assumptions \ref{assum:2} and \ref{assum:nparticles}, for all $\xi \in \gO_{N, k}$, for all $t \ge t_{\delta/2}$ and all $N$ sufficiently large we have 
\begin{equation}
\max_{x \in \lint 1, N \rint}\, \bE\left[ h_t^\xi(x) \right] \;\le\; 2^6  \bar \Upsilon_N^{-1/2}(K_0+1)\cdot k e^{-\gk_1 t}\,,
\end{equation}
where $\gk_1$ is defined in \eqref{hatA:eigfunval} and $K_0 \colonequals \left\lceil \left( 2+ \frac{3}{\varrho}\right)^{1/2}  \right\rceil $.
\end{lemma}

\begin{proof}
As in \eqref{sol:heightfun}, we have 
\begin{equation*}
\bE \left[ h_t^\xi(x)\right]
\;=\;
\sum_{i=1}^{N-1} e^{- \gk_i t} \fg_i(x) \langle \fg_i, h_0^\xi \rangle_{\nu}\,,
\end{equation*}
and then

\begin{equation}\label{upbd:heightexp}
\begin{aligned}
\left\vert \bE \left[ h_t^\xi(x)\right]\right\vert
\;&\le \;
\sum_{i=1}^{N-1} e^{- \gk_i t} \vert \fg_i(x) \vert \cdot \left \vert \langle \fg_i, h_0^\xi \rangle_{\nu}\right\vert
\;\le\;
\sum_{i=1}^{N-1}k e^{- \gk_i t} \vert \fg_i(x) \vert\,,
\end{aligned}
\end{equation}
where the last inequality is by Cauchy-Schwarz inequality and 
\begin{equation*}
\vert \langle h_0^\xi, h_0^\xi \rangle_\nu \vert
\; \le \;
\max_{1\le y \le N-1}\left \vert  h^\xi_0(y) \right\vert 
\;\le\; k\,.
\end{equation*}

For an upper bound on \eqref{upbd:heightexp}, we first treat those terms with $1 \le i \le K_0$ based on Theorem \ref{th:lapaceeigenshape} and the following observation:
 for all $1 \le i \le K_0$, 
\begin{equation}\label{set:largevalue}
\# \cE_{i}\; \colonequals\;
\# \left\{j \in \lint 1, N-1\rint: \left\vert \sin \left(\frac{i \pi j}{N} \right) \right\vert \ge \frac{1}{4} \right\}\; \ge \; \frac{1}{16} N\,.
\end{equation}
Thus, by Theorem \ref{th:lapaceeigenshape} and for $x \in \cE_i$ we have
\begin{equation*}
\left\vert \frac{\fg_i(x)}{\fg_i(1)} \sin \left(\frac{i \pi}{N} \right) \right \vert \;\ge\; \frac{1}{8}\,,
\end{equation*}
and then
\begin{equation}\label{upbd:gi1}
1
\;=\; \langle \fg_i, \fg_i \rangle_\nu 
\;\ge \;
\sum_{x \in \cE_i} \nu(x) \fg_i(x)^2
\;\ge\;
\sum_{x \in \cE_i} \frac{r(x, x+1)}{\sum_{j=1}^{N-1} r(j, j+1)}  \frac{\fg_i(1)^2}{ 64 \sin^2\left(\frac{i \pi}{N} \right)}\,.
\end{equation}
By Assumption \ref{assum:2}, \eqref{LLN} and \eqref{set:largevalue}, we have
\begin{equation*}
\sum_{x \in \cE_i} \frac{r(x, x+1)}{\sum_{j=1}^{N-1} r(j, j+1)}  \;\ge\;  \bar \Upsilon_N  \frac{ \#   \cE_i}{2N} \;\ge\;    \frac{\bar \Upsilon_N}{32}\,,
\end{equation*}
and therefore \eqref{upbd:gi1} implies
\begin{equation}\label{upbd:fgi1}
\vert \fg_i(1)\vert 
\;\le\;
 \frac{  2^5}{\sqrt{ \bar \Upsilon_N}} \left\vert   \sin \left(  \frac{i \pi}{N  } \right) \right\vert \,.
\end{equation}
Then by \eqref{upbd:fgi1} and  Theorem \ref{th:lapaceeigenshape} for all $1 \le i \le K_0$ and all $x \in \lint 1, N\rint$ we have
\begin{equation*}
\vert \fg_i(x)  \vert \;\le\; 2 \left\vert \frac{ \fg_i(1) }{\sin\left( \frac{i \pi}{N} \right)} \right\vert \; \le\; 2^{6} \bar \Upsilon_N^{-1/2}\,,
\end{equation*}
and  by $\gk_i \ge \gk_1$, 
\begin{equation}\label{estimate:K0--}
\sum_{i=1}^{K_0}k e^{- \gk_i t} \vert \fg_i(x) \vert \;\le \;  2^{6} \bar \Upsilon_N^{-1/2} K_0 k e^{-\gk_1 t}\,. 
\end{equation}

For those terms in the r.h.s\ of \eqref{upbd:heightexp} indexed by 
$i >K_0$, by Assumption \ref{assum:2} and \eqref{LLN}
we have
\begin{equation*}
1\;=\; 
\langle \fg_i, \fg_i \rangle_\nu
\;=\;
\sum_{x=1}^{N-1} \frac{r(x, x+1)}{\sum_{j=1}^{N-1}r(j, j+1)} \fg_i(x)^2
\;\ge\;
\frac{\bar \Upsilon_N}{2N} \fg_i(x)^2\,,
\end{equation*}
and then 
\begin{equation}
\vert \fg_i(x)\vert \;\le\; \sqrt{2N \bar \Upsilon_N^{-1}}\,.
\end{equation}
Moreover, by Theorem \ref{th:lapaceeigenshape} and Lemma \ref{lema:appendeigvlw} we have for all $i>K_0$
\begin{equation*}
\gk_i > \frac{K_0^2 \pi^2}{N^2}  
\end{equation*}
and then for $t \ge t_{\delta/2}$
\begin{equation}\label{estimate:K0++}
\sum_{i=K_0+1}^{N-1} e^{- \gk_i t} \vert \fg_i(x) \vert 
\;\le \;
\sum_{i=K_0+1}^{N-1} e^{- \gk_i t} \sqrt{2N \bar \Upsilon_N^{-1}}
\;\le\;
\sqrt{2 \bar \Upsilon_N^{-1}} N^{3/2}  e^{-t \cdot K_0^2 \pi^2/N^2}
\;\le\; \sqrt{2 \bar \Upsilon_N^{-1}}
e^{-t \cdot \gk_1 }
\end{equation}
where the last inequality holds by choosing
\begin{equation}
K_0\;\ge\; \left[ 2+ \frac{3}{\varrho}\right]^{1/2} \,.
\end{equation}
Combining \eqref{estimate:K0--} and  \eqref{estimate:K0++},  we  conclude the proof.

\end{proof}

We need a version of \cite[Proposition 4.2]{lacoin2016mixing} which is the following.
 
 \begin{proposition}\label{prop:kpartdiseig}
 For all $t \ge t_{\delta/2}$, 
we have

\begin{equation}\label{tvdistance:coupeig}
d_{N, k}(t) \;\le\;  2^{14}K_0^2 k e^{-\gl_1 t}\,,
\end{equation}
where $K_0 \colonequals \left\lceil 2 \left(\frac{3}{\varrho}+1\right)^{1/2} \right\rceil.$

\end{proposition}

To prepare for the proof of Proposition \ref{prop:kpartdiseig},  now we adapt the coupling in \cite[ Proposition 6.5]{lacoin2016mixing} to provide a graphical construction which is not Markovian.
We label particles from $1$ to $k$, and each particle receives a different label  with $0$s representing empty sites. 
The coupling has the following rules (with an abuse of notation, we still denote the process as $(\eta_t^\xi)_{t \ge 0}$):  
for $\xi, \xi' \in \gO_{N, k}$, 
\begin{itemize}

\item if $\eta^\xi_t(x)\ne\eta^{\xi'}_t(x)$ and $\eta^{\xi
}_t(x+1)\ne\eta^{\xi'}_t(x+1)$, then the transition $\eta\to
\eta\circ\tau_{x, x+1}$
occurs independently with rate  $c(x, x+1)$ for each of the two processes;

\item if either $\eta^\xi_t(x)= \eta^{\xi'}_t(x)$ or $\eta
^{\xi}_t(x+1)= \eta^{\xi'}_t(x+1)$ (or both), then the transition
$\eta\to\eta\circ\tau_{x, x+1}$
occurs simultaneously for the two processes with rate $c(x, x+1)$.
\end{itemize}

Let $X^i_t:=(\eta^{\xi}_t)^{-1}(i)$ and $Y^i_t:=(\eta^{\xi
'}_t)^{-1}(i)$ denote the trajectory of the two particles labeled $i$ for
the two coupled dynamics.
The couple $(X^i_t,Y^i_t)_{t \ge 0}$ is a Markov chain with the following
transition rules:

\begin{itemize}

\item if $x\ne y$, then the transitions 
\begin{equation*}
\begin{cases}
(x,y)\to(x+1,y) &\text{ with rate } c(x, x+1)\,,\\
(x,y)\to(x-1,y) &\text{ with rate } c(x-1, x)\,,\\
(x,y)\to(x, y+1) &\text{ with rate } c(y, y+1)\,,\\
(x,y)\to(x, y-1) &\text{ with rate } c(y-1, y)\,,
\end{cases}
\end{equation*}
 provided
the two coordinates stay between $1$ and $N$;

\item if $x=y$, then the transitions 
\begin{equation*}
\begin{cases}
(x,y)\to(x+1,y+1) & \text{ with rate } c(x, x+1)\,,\\
(x,y)\to(x-1, y-1) & \text{ with rate } c(x-1, x)\,,\\
\end{cases}
\end{equation*}
provided the two coordinates stay
between $1$ and $N$.
\end{itemize}
Any other transition is with rate $0$. 
Let $\fL$ denote the corresponding generator for this Markov chain above. 
In particular, once $X^i_t$
and $Y^i_t$ have merged, they stay together from then  on. 
Recalling that $\cL_{N, 1} g_i=-\gl_i g_i$,  $g_i(0)=g_i(1)$ and $g_i(N)=g_i(N+1)$, for $0 \le i <N$ we define
\begin{equation}
\tilde g_i \; \colonequals \; g_i \left(\langle g_i, g_i\rangle_{\mu_{N,1}}\right)^{-1/2}\,,
\end{equation}
where $\mu_{N, 1}$ is the uniform probability measure on $\lint 1, N \rint$, so that $\langle \tilde g_i, \tilde g_i \rangle_{\mu_{N, 1}}=1$.
Moreover,
for $0 \le i<j \le N-1$ we define
\begin{equation}\label{def:uij}
u_{i, j}(x, y) \;\colonequals\; \tilde g_i(x) \tilde g_j(y)- \tilde g_i(y) \tilde g_j(x)\,.
\end{equation}

\begin{lemma}\label{lema:appeigfunsq}
We have that $(u_{i, j})_{0 \le i<j \le N-1}$ is a base of eigenfunctions of the generator $\fL$.
\end{lemma}

\begin{proof}
Note that $(X_t^i, Y_t^i)_{t \ge 0}$ either lives in the upper triangle or the lower triangle of the $N$ by $N$ square, and once they hit the diagonal line they move only in the diagonal line since then. Without loss of generality, we assume $(X_t^i, Y_t^i)_{t \ge 0}$ stays in the lower triangle.

We now show that 
\begin{equation}\label{pos:einfun}
\fL u_{i, j}\;=\;-(\gl_i+\gl_j) u_{i, j}\,.
\end{equation}
For $x=y$, $u_{i, j}(x, x)=0$ and thus \eqref{pos:einfun} holds for the diagonal points. 
We now treat the case $x \neq y$: 
\begin{equation}\label{def:genfL}
\begin{aligned}
(\fL u_{i, j})(x, y)
\;&=\;
c(x, x+1)\left[ u_{i, j}(x+1, y)-u_{i, j}(x, y)\right]\\
&+\;
c(x-1, x)\left[ u_{i, j}(x-1, y)-u_{i, j}(x, y)\right]\\
&+\;
c(y, y+1)\left[ u_{i, j}(x, y+1)-u_{i, j}(x, y)\right]\\
&+\;
c(y-1, y)\left[ u_{i, j}(x, y-1)-u_{i, j}(x, y)\right]\,.
\end{aligned}
\end{equation}
Plugging the definition of $u_{i, j}$ in \eqref{def:genfL}  and  using the fact that $(\tilde g_i)_{0 \le i \le N-1}$ are eigenfunctions of $\cL_{N, 1}$, we obtain \eqref{pos:einfun}.

Now we move to show that $(u_{i, j})_{0 \le i<j \le N-1}$ are orthogonal under the usual inner product
$\langle f, g \rangle \colonequals \sum_{1\le x \le y \le N} f((x, y)) g((x, y))$ where $f, g: \lint 1, N \rint^2 \mapsto \bbR$. 
 Observe that 
\begin{equation*}
u_{i, j}(x, y)=-u_{i, j}(y, x)\,,\;\; u_{i, j}(x, x)=0
\end{equation*} 
  and then for $0 \le i<j <N$ and $0 \le m<n <N$,
\begin{equation}\label{norma:uij}
\begin{aligned}
2 \langle u_{i, j}, u_{m, n} \rangle
\;&=\;
\sum_{1 \le x, y \le N} u_{i, j}(x, y) u_{m, n}(x, y)\\
\;&=\;
\sum_{1 \le x, y \le N}\left[ \tilde g_i(x) \tilde g_j(y)-\tilde g_i(y) \tilde g_j(x) \right] \cdot 
\left[\tilde g_m(x) \tilde g_n(y)-\tilde g_m (y) \tilde g_n (x) \right]\\
\;&=\;
\sum_{1 \le x, y \le N} \tilde g_i(x) \tilde g_j(y) \tilde g_m(x) \tilde g_n(y)-\tilde g_i(x) \tilde g_j(y) \tilde g_m(y) \tilde g_n(x)\\
&+\sum_{1 \le x, y \le N}
-\tilde g_i(y) \tilde g_j(x) \tilde g_m(x)\tilde g_n(y)+\tilde g_i(y) \tilde g_j(x) \tilde g_m(y) \tilde g_n(x)\\
\;&=\;
\sum_{1 \le x \le N} \tilde  g_i(x) \tilde g_m(x) \sum_{1 \le y \le N}  \tilde g_j(y) \tilde g_n(y)-\sum_{1 \le x \le N} \tilde g_i(x) \tilde g_n(x) \sum_{1 \le y \le N}  \tilde g_j(y) \tilde g_m(y)\\
&-\sum_{1 \le x \le N} \tilde g_j(x) \tilde g_m(x) \sum_{1 \le y \le N} \tilde g_i(y) \tilde g_n(y)+\sum_{1 \le x \le N} \tilde g_j(x) \tilde g_n(x) \sum_{1 \le y \le N}  \tilde g_i(y) \tilde g_m(y)\\
\;&=\;
N^2\ind_{\{ i=m, j=n\}}- N^2\ind_{\{ i=n, j=m\}}-N^2\ind_{\{ i=n, j=m\}}+N^2\ind_{\{ i=m, j=n\}}\\
\;&=\;2N^2 \left( \ind_{\{ i=m, j=n\}}- \ind_{\{ i=n, j=m\}}\right)\,.
\end{aligned}
\end{equation}
Note that in the second indicator above if $j=m$,
the condition $i<j=m<n$ guarantees that $\{ i=n, j=m\}$ does not hold. Therefore $(u_{i, j})_{0 \le i<j \le N-1}$ are orthogonal.
Since the dimension of $\fL$ (viewing the diagonal line as cemetery) is $N(N-1)/2$ and the dimension of vector space spanned by $(u_{i, j})_{0 \le i<j \le N-1}$ is also $N(N-1)/2$, then $(u_{i, j})_{0 \le i<j \le N-1}$ is a base of eigenfunctions of $\fL$.
\end{proof}

With Lemma \ref{lema:appeigfunsq} at hand, we are ready to prove Proposition \ref{prop:kpartdiseig}.

\begin{proof}[Proof of Proposition \ref{prop:kpartdiseig}]
In fact,  \eqref{tvdistance:coupeig} follows from  a union bound and the following inequality: for all $t \ge t_{\delta/2}$

\begin{equation}
\bP_{x_0, y_0} \left[ X_t^i \neq Y_t^i \right]\;\le\;   2^{14}K_0^2 e^{-\gl_1 t}  \,.
\end{equation}

Without loss of generality, we assume $(X_0^i, Y_0^i)=(x_0, y_0)$ with $x_0<y_0$, and then 
\begin{equation}\label{prob:nomergediag}
\begin{aligned}
\bP_{x_0, y_0} \left[ X_t^i \neq Y_t^i \right]\;&=\;
\sum_{1 \le x<y \le N} e^{t \fL}\left((x_0, y_0), (x, y)\right)\\
&=\;
\delta_{(x_0, y_0)} e^{t \fL} \ind_{\{1\le x<y \le N\}}\,,
\end{aligned}
\end{equation}
where $\delta_{(x_0, y_0)}$ is a  row vector of size $N(N+1)/2$ and $\ind_{1\le x<y \le N}$ is a  column vector of size $N(N+1)/2$.  
Furthermore,  we write
\begin{equation*}
\ind_{\{ 1\le x< y \le N\}}\;=\; \sum_{0\le i<j \le N-1} a_{i, j} u_{i, j}
\end{equation*}
where
\begin{equation*}
a_{i, j} \;\colonequals\; \frac{\langle \ind_{1\le x<y \le N}, u_{i, j} \rangle}{\langle  u_{i, j}, u_{i, j} \rangle}\,.
\end{equation*}
Then, by \eqref{pos:einfun}  we have
\begin{equation}
\begin{aligned}
e^{t \fL} \ind_{\{1 \le x<y \le N \}}
\;&=\;
\sum_{0 \le i<j \le N-1} a_{i, j} e^{t \fL} u_{i, j}\\
\;&=\;\sum_{0 \le i<j \le N-1} a_{i, j}  e^{-(\gl_i+\gl_j)t} u_{i, j}\,.
\end{aligned}
\end{equation}
Then the r.h.s.\ in \eqref{prob:nomergediag} is equal to 
\begin{equation}\label{upbd:probnomerg}
\sum_{0\le i<j \le N-1 } a_{i, j} \delta_{(x_0, y_0)}e^{-(\gl_i+\gl_j)t} u_{i, j}
\;=\;
\sum_{0\le i<j \le N-1 } a_{i, j} e^{-(\gl_i+\gl_j)t} u_{i, j}(x_0, y_0)\,.
\end{equation}
By Cauchy-Schwarz inequality and \eqref{norma:uij} we have
\begin{equation*}
\vert a_{i, j}  \vert
\;\le \; \frac{\langle \ind_{\{ 1 \le x<y \le N\}}, \ind_{\{ 1 \le x<y \le N\}} \rangle^{1/2} }{\langle u_{i, j}, u_{i, j} \rangle^{1/2}}
\;=\;
\frac{1}{N} \sqrt{\frac{N(N-1)}{2}} \;\le\; 1\,.
\end{equation*}
Then the r.h.s.\ in \eqref{upbd:probnomerg}
can be bounded from above by 
\begin{equation}\label{prob:nodiaghit}
\sum_{0\le i<j \le N-1 }  e^{-(\gl_i+\gl_j)t} \vert  u_{i, j}(x_0, y_0)\vert\,.
\end{equation} 
In order to provide upper bound on $\vert u_{i, j}(x_0, y_0) \vert$, 
  we head to provide an upper bound on $\max_{1 \le x \le N}\vert \tilde g_i(x) \vert$ with $1 \le i \le K_0$,  based on  \eqref{eigfun:shape}. 
Note that for all $1 \le i \le K_0$, 
\begin{equation*}
\# \cG_{i}\; \colonequals\;
\# \left\{j \in \lint 1, N\rint: \left\vert \cos \left(\frac{i \pi (j-1/2)}{N} \right) \right\vert \ge \frac{1}{4} \right\}\; \ge \; \frac{1}{16} N\,.
\end{equation*}
Then by  \eqref{eigfun:shape} and for $x \in \cG_i$, we have
\begin{equation*}
\left\vert \frac{\tilde g_i(x)}{\tilde g_i(1)} \cos \left(\frac{i \pi/2}{N} \right) \right \vert \;\ge\; \frac{1}{8}\,,
\end{equation*}
and then
\begin{equation*}
1
\;=\; \langle \tilde g_i, \tilde g_i \rangle_{\mu_{N,1}} 
\;\ge \; \frac{1}{N}
\sum_{x \in \cG_i}  \tilde g_i(x)^2
\;\ge\; \frac{1}{N} 
\sum_{x \in \cG_i}  \frac{\tilde g_i(1)^2}{2^{6} \cos^2\left( \frac{i \pi}{2N}\right)} \;\ge\; \frac{1}{2^{10}} \frac{\tilde  g_i(1)^2}{\cos^2 \left( \frac{i \pi}{2N} \right)}\,,
\end{equation*} 
 which implies $\vert \tilde g_i(1)\vert \le 2^5 \vert \cos \left( \frac{i \pi}{2 N} \right) \vert. $ Then by \eqref{eigfun:shape}, we obtain
 \begin{equation}
 \max_{1 \le x \le N} \vert \tilde g_i(x)\vert \;\le\;  2 \vert \tilde g_i(1) \vert \frac{1}{ \vert \cos \left(\frac{i \pi}{2N} \right)\vert } \;\le\; 2^6\,.
 \end{equation}
Thus for $0 \le i< j \le K_0$,  by \eqref{def:uij}
we have
\begin{equation}\label{upbd:uijsmalli}
\max_{1\le x< y \le N} \vert u_{i, j}(x, y) \vert \;\le\; 2^{13}\,.
\end{equation}

Now for all $i \in \lint 1, N \rint$, using
\begin{equation*}
1\;=\; \langle \tilde g_i, \tilde g_i \rangle_{\mu_{N, 1}} \;\ge\; \tilde g_i(x)^2/N\,,
\end{equation*} 
 we obtain $\max_{1\le x \le N} \vert \tilde g_i(x)\vert \le \sqrt{N}$, and then by \eqref{def:uij}
 \begin{equation}\label{upbd:uijbigi}
 \max_{1 \le x <y \le N} \vert u_{i, j}(x, y) \vert \;\le\; 2N\,. 
 \end{equation}

Now we are ready to provide an upper bound on the r.h.s.\ of \eqref{prob:nodiaghit} by using \eqref{upbd:uijsmalli} and \eqref{upbd:uijbigi}: 
\begin{equation}
\begin{aligned}
&\sum_{0\le i<j \le K_0 }  e^{-(\gl_i+\gl_j)t} \vert  u_{i, j}(x_0, y_0)\vert 
+ \sum_{0\le i<j \le N-1,\,j >K_0 }  e^{-(\gl_i+\gl_j)t} \vert  u_{i, j}(x_0, y_0)\vert \\
&\le\;
\sum_{0\le i<j \le K_0 } 2^{13} e^{-(\gl_i+\gl_j)t} +
 \sum_{0\le i<j \le N-1,\,j >K_0 } 2N e^{-(\gl_i+\gl_j)t}\\
 &\le\; 2^{13}K_0^2 e^{-\gl_1 t}+N^3 e^{-\gl_{K_0} t}\\
 &\le\;
 2^{14}K_0^2 e^{-\gl_1 t}
\end{aligned}
\end{equation}
 where   we have  \eqref{eigv:distinct} in the second last inequality, and have
 used 
 \eqref{eigfun:shape} for the estimate on $\gl_{K_0}$ 
  and $t \ge t_{\delta/2}$  with 
 \begin{equation*}
 K_0\; \ge \; \left\lceil 2\sqrt{\frac{3}{\varrho}+1} \right\rceil\,.
 \end{equation*}
 
\end{proof}

By Theorem \ref{th:gapshapeder} and Theorem \ref{th:lapaceeigenshape}, we have
\begin{equation}\label{2paramequals}
\lim_{N \to \infty} N^2\left\vert \gl_1-\gk_1 \right\vert\;=\;0\,.
\end{equation}
By Lemma \ref{lema:upbdheight} and \eqref{2paramequals}, we have the condition in \cite[Lemma 8.4]{lacoin2016mixing}: for all $t \ge t_{\delta/2}$ and all $N$ sufficiently large,
\begin{equation}\label{cond:lema84Lac}
\sum_{i=1}^{\bar K}  \bE \left[ h_t^\wedge(x_i) \right]
\;\le\;
 2^6  \bar \Upsilon_N^{-1/2}\bar K (K_0+1)\cdot k e^{-\gk_1 t}
 \;\le\;
 \bar K  \sqrt{k}  \delta\,,
\end{equation}
where we have used the assumption on $\bar \Upsilon_N$ in Assumption \ref{assum:2} and Assumption \ref{assum:nparticles} in the last inequality. Then by \eqref{cond:lema84Lac} and \cite[Lemma 8.4]{lacoin2016mixing}, we have \cite[Proposition 8.3]{lacoin2016mixing}. Following the proof in \cite[Proposition 8.2]{lacoin2016mixing} with Theorem \ref{th:gapshapeder} and Proposition \ref{prop:kpartdiseig}, we conclude the proof of Proposition \ref{prop:tdmaxconf}.
\qed

\bibliographystyle{alpha}
\bibliography{library.bib}

\begin{thebibliography}{{Yan}24b}

\bibitem[BBHM05]{benjamini2005mixing}
Itai Benjamini, Noam Berger, Christopher Hoffman, and Elchanan Mossel.
\newblock Mixing times of the biased card shuffling and the asymmetric
  exclusion process.
\newblock {\em Transactions of the American Mathematical Society},
  357(8):3013--3029, 2005.

\bibitem[CLR10]{caputo2010aldous}
Pietro Caputo, Thomas~M. Liggett, and Thomas Richthammer.
\newblock Proof of {A}ldous' spectral gap conjecture.
\newblock {\em J. Amer. Math. Soc.}, 23(3):831--851, 2010.

\bibitem[Dia96]{diaconis1996cutoff}
Persi Diaconis.
\newblock The cutoff phenomenon in finite markov chains.
\newblock {\em Proceedings of the National Academy of Sciences},
  93(4):1659--1664, 1996.

\bibitem[DSC93]{DSC93}
Persi Diaconis and Laurent Saloff-Coste.
\newblock Comparison theorems for reversible {M}arkov chains.
\newblock {\em Ann. Appl. Probab.}, 3(3):696--730, 1993.

\bibitem[Dur10]{Durrett}
Rick Durrett.
\newblock {\em Probability: theory and examples}, volume~31 of {\em Cambridge
  Series in Statistical and Probabilistic Mathematics}.
\newblock Cambridge University Press, Cambridge, fourth edition, 2010.

\bibitem[Fag08]{faggionato2008random}
Alessandra Faggionato.
\newblock Random walks and exclusion processes among random conductances on
  random infinite clusters: homogenization and hydrodynamic limit.
\newblock {\em Electronic Journal of Probability}, 13:2217--2247, 2008.

\bibitem[FGS16]{FGS16}
Tertuliano Franco, Patr\'{\i}cia Gon\c{c}alves, and Marielle Simon.
\newblock Crossover to the stochastic {B}urgers equation for the {WASEP} with a
  slow bond.
\newblock {\em Comm. Math. Phys.}, 346(3):801--838, 2016.

\bibitem[FN17]{FN17}
Tertuliano Franco and Adriana Neumann.
\newblock Large deviations for the exclusion process with a slow bond.
\newblock {\em Ann. Appl. Probab.}, 27(6):3547--3587, 2017.

\bibitem[FRS21]{FRS19}
Simone Floreani, Frank Redig, and Federico Sau.
\newblock Hydrodynamics for the partial exclusion process in random
  environment.
\newblock {\em Stochastic Process. Appl.}, 142:124--158, 2021.

\bibitem[GK13]{gantert2012cutoff}
Nina Gantert and Thomas Kochler.
\newblock Cutoff and mixing time for transient random walks in random
  environments.
\newblock {\em ALEA Lat. Am. J. Probab. Math. Stat.}, 10(1):449–484, 2013.

\bibitem[HS19]{Hermon_Salez}
Jonathan Hermon and Justin Salez.
\newblock A version of {A}ldous' spectral-gap conjecture for the zero range
  process.
\newblock {\em Ann. Appl. Probab.}, 29(4):2217--2229, 2019.

\bibitem[KKS75]{kesten1975limit}
Harry Kesten, Mykyta~V Kozlov, and Frank Spitzer.
\newblock A limit law for random walk in a random environment.
\newblock {\em Compositio Mathematica}, 30(2):145--168, 1975.

\bibitem[KL99]{LandimHydrodynamicsbk}
Claude Kipnis and Claudio Landim.
\newblock {\em Scaling limits of interacting particle systems}, volume 320 of
  {\em Grundlehren der Mathematischen Wissenschaften [Fundamental Principles of
  Mathematical Sciences]}.
\newblock Springer-Verlag, Berlin, 1999.

\bibitem[KOV89]{KOV89}
C.~Kipnis, S.~Olla, and S.~R.~S. Varadhan.
\newblock Hydrodynamics and large deviation for simple exclusion processes.
\newblock {\em Comm. Pure Appl. Math.}, 42(2):115--137, 1989.

\bibitem[Lac16a]{Lacoin2016profile}
Hubert Lacoin.
\newblock The cutoff profile for the simple exclusion process on the circle.
\newblock {\em Ann. Probab.}, 44(5):3399--3430, 2016.

\bibitem[Lac16b]{lacoin2016mixing}
Hubert Lacoin.
\newblock Mixing time and cutoff for the adjacent transposition shuffle and the
  simple exclusion.
\newblock {\em The Annals of Probability}, 44(2):1426--1487, 2016.

\bibitem[Lac17]{Lacoin2017diffusivecutoff}
Hubert Lacoin.
\newblock The simple exclusion process on the circle has a diffusive cutoff
  window.
\newblock {\em Ann. Inst. Henri Poincar\'e{} Probab. Stat.}, 53(3):1402--1437,
  2017.

\bibitem[Lig12]{liggett2012interacting}
Thomas~Milton Liggett.
\newblock {\em Interacting particle systems}, volume 276.
\newblock Springer Science \& Business Media, 2012.

\bibitem[LL19]{labbe2016cutoff}
Cyril Labb{\'e} and Hubert Lacoin.
\newblock Cutoff phenomenon for the asymmetric simple exclusion process and the
  biased card shuffling.
\newblock {\em The Annals of Probability}, 47(3):1541--1586, 2019.

\bibitem[LL20]{labbe2018mixing}
Cyril Labb\'{e} and Hubert Lacoin.
\newblock Mixing time and cutoff for the weakly asymmetric simple exclusion
  process.
\newblock {\em Ann. Appl. Probab.}, 30(4):1847--1883, 2020.

\bibitem[LP17]{LPWMCMT}
David~A. Levin and Yuval Peres.
\newblock {\em Markov chains and mixing times}.
\newblock American Mathematical Society, Providence, RI, 2017.
\newblock Second edition of [ MR2466937], With contributions by Elizabeth L.
  Wilmer, With a chapter on ``Coupling from the past'' by James G. Propp and
  David B. Wilson.

\bibitem[LY24]{Lacoin2021aseprandom}
Hubert Lacoin and Shangjie Yang.
\newblock Mixing time for the asymmetric simple exclusion process in a random
  environment.
\newblock {\em Ann. Appl. Probab.}, 34(1A):388--427, 2024.

\bibitem[Mor06]{Morris06}
Ben Morris.
\newblock The mixing time for simple exclusion.
\newblock {\em Ann. Appl. Probab.}, 16(2):615--635, 2006.

\bibitem[MS77]{MacWilliams1977error}
F.~J. MacWilliams and N.~J.~A. Sloane.
\newblock {\em The theory of error-correcting codes. {I}}, volume Vol. 16 of
  {\em North-Holland Mathematical Library}.
\newblock North-Holland Publishing Co., Amsterdam-New York-Oxford, 1977.

\bibitem[Oli13]{oliveira2013mixing}
Roberto~Imbuzeiro Oliveira.
\newblock Mixing of the symmetric exclusion processes in terms of the
  corresponding single-particle random walk.
\newblock {\em The Annals of Probability}, 41(2):871--913, 2013.

\bibitem[PW13]{peres2013can}
Yuval Peres and Peter Winkler.
\newblock Can extra updates delay mixing?
\newblock {\em Communications in Mathematical Physics}, 323(3):1007--1016,
  2013.

\bibitem[Qua92]{Quastel92}
Jeremy Quastel.
\newblock Diffusion of color in the simple exclusion process.
\newblock {\em Comm. Pure Appl. Math.}, 45(6):623--679, 1992.

\bibitem[Rez91]{rez91}
Fraydoun Rezakhanlou.
\newblock Hydrodynamic limit for attractive particle systems on {${\bf Z}^d$}.
\newblock {\em Comm. Math. Phys.}, 140(3):417--448, 1991.

\bibitem[Ros81]{Rost81}
H.~Rost.
\newblock Nonequilibrium behaviour of a many particle process: density profile
  and local equilibria.
\newblock {\em Z. Wahrsch. Verw. Gebiete}, 58(1):41--53, 1981.

\bibitem[Sch19]{schmid2019mixing}
Dominik Schmid.
\newblock Mixing times for the simple exclusion process in ballistic random
  environment.
\newblock {\em Electronic Journal of Probability}, 24, 2019.

\bibitem[Sol75]{solomon1975random}
Fred Solomon.
\newblock Random walks in a random environment.
\newblock {\em The annals of probability}, pages 1--31, 1975.

\bibitem[Szn04]{sznitman2004topics}
Alain-Sol Sznitman.
\newblock Topics in random walks in random environment.
\newblock In {\em School and conference on probability theory: 13-17 May 2002},
  volume~17, pages 203--266. The Abdus Salam International Centre for
  Theoretical Physics, 2004.

\bibitem[Wil04]{wilson2004mixing}
David~Bruce Wilson.
\newblock Mixing times of lozenge tiling and card shuffling markov chains.
\newblock {\em The Annals of Applied Probability}, 14(1):274--325, 2004.

\bibitem[{Yan}21]{Yang2021thesis}
Shangjie {Yang}.
\newblock {Mixing Time for Interface Models and Particle System}.
\newblock {\em Doctoral thesis, available at
  https://shjyang.github.io/thesis.pdf}, 2021.

\bibitem[{Yan}24a]{Yang2024oneslow}
Shangjie {Yang}.
\newblock {Spectral gap and mixing time for simple exclusion process with one
  slow bond}.
\newblock {\em In preparation}, 2024.

\bibitem[{Yan}24b]{Yang2024SpectralLight}
Shangjie {Yang}.
\newblock {The spectral gap and principle eigenfunction of the random
  conductance model in a line segment}.
\newblock {\em arXiv e-prints}, page arXiv:2408.07139, August 2024.

\bibitem[Yau97]{Yau97}
Horng-Tzer Yau.
\newblock Logarithmic {S}obolev inequality for generalized simple exclusion
  processes.
\newblock {\em Probab. Theory Related Fields}, 109(4):507--538, 1997.

\bibitem[Zei04]{zeitouni2004part}
Ofer Zeitouni.
\newblock Part ii: Random walks in random environment.
\newblock In {\em Lectures on Probability Theory and Statistics}, pages
  190--312. Springer, 2004.

\end{thebibliography}

\end{document}